\documentclass{amsart}
\usepackage{amssymb}
\usepackage{amsbsy, amsthm, amsmath,amstext, amsopn}
\usepackage[all]{xy}
\usepackage{amsfonts}
\usepackage{amscd}
\newtheorem{thm}{Theorem} [section]
\newtheorem{lemma}[thm]{Lemma}

\newtheorem{corollary}[thm]{Corollary}
\newtheorem{prop}[thm]{Proposition}

\theoremstyle{definition}

\newtheorem{defn}[thm]{Definition}

\theoremstyle{remark}

\newtheorem{remark}[thm]{Remark}

\begin{document}

\numberwithin{equation}{section}

\newcommand{\hs}{\mbox{\hspace{.4em}}}
\newcommand{\ds}{\displaystyle}
\newcommand{\bd}{\begin{displaymath}}
\newcommand{\ed}{\end{displaymath}}
\newcommand{\bcd}{\begin{CD}}
\newcommand{\ecd}{\end{CD}}

\newcommand{\on}{\operatorname}
\newcommand{\proj}{\operatorname{Proj}}
\newcommand{\bproj}{\underline{\operatorname{Proj}}}

\newcommand{\spec}{\operatorname{Spec}}
\newcommand{\Spec}{\operatorname{Spec}}
\newcommand{\bspec}{\underline{\operatorname{Spec}}}
\newcommand{\pline}{{\mathbf P} ^1}
\newcommand{\aline}{{\mathbb A} ^1}
\newcommand{\pplane}{{\mathbf P}^2}
\newcommand{\pinfty}{{\mathbf P}^\infty}
\newcommand{\aplane}{{\mathbf A}^2}
\newcommand{\coker}{{\operatorname{coker}}}
\newcommand{\ldb}{[[}
\newcommand{\rdb}{]]}

\newcommand{\Sym}{\operatorname{Sym}^{\bullet}}
\newcommand{\Symp}{\operatorname{Sym}}
\newcommand{\Pic}{\bf{Pic}}
\newcommand{\Aut}{\operatorname{Aut}}
\newcommand{\PAut}{\operatorname{PAut}}

\newcommand{\too}{\twoheadrightarrow}
\newcommand{\C}{{\mathbf C}}
\newcommand{\Z}{{\mathbf Z}}
\newcommand{\Q}{{\mathbf Q}}
\newcommand{\R}{{\mathbf R}}
\newcommand{\Cx}{{\mathbf C}^{\times}}
\newcommand{\Cbar}{\overline{\C}}
\newcommand{\Cxbar}{\overline{\Cx}}
\newcommand{\cA}{{\mathcal A}}
\newcommand{\cS}{{\mathcal S}}
\newcommand{\cV}{{\mathcal V}}
\newcommand{\cM}{{\mathcal M}}
\newcommand{\bA}{{\mathbf A}}
\newcommand{\cB}{{\mathcal B}}
\newcommand{\cC}{{\mathcal C}}
\newcommand{\cD}{{\mathcal D}}
\newcommand{\D}{{\mathcal D}}
\newcommand{\cs}{{\mathbf C} ^*}
\newcommand{\boldc}{{\mathbf C}}
\newcommand{\cE}{{\mathcal E}}
\newcommand{\cF}{{\mathcal F}}
\newcommand{\bF}{{\mathbf F}}
\newcommand{\cG}{{\mathcal G}}
\newcommand{\G}{{\mathbb G}}
\newcommand{\cH}{{\mathcal H}}
\newcommand{\CI}{{\mathcal I}}
\newcommand{\cJ}{{\mathcal J}}
\newcommand{\cK}{{\mathcal K}}
\newcommand{\cL}{{\mathcal L}}
\newcommand{\baL}{{\overline{\mathcal L}}}
\newcommand{\M}{{\mathcal M}}
\newcommand{\Mf}{{\mathfrak M}}
\newcommand{\bM}{{\mathbf M}}
\newcommand{\bm}{{\mathbf m}}
\newcommand{\cN}{{\mathcal N}}
\newcommand{\theo}{\mathcal{O}}
\newcommand{\cP}{{\mathcal P}}
\newcommand{\cR}{{\mathcal R}}
\newcommand{\Pp}{{\mathbb P}}
\newcommand{\boldp}{{\mathbf P}}
\newcommand{\boldq}{{\mathbf Q}}
\newcommand{\bbL}{{\mathbf L}}
\newcommand{\cQ}{{\mathcal Q}}
\newcommand{\cO}{{\mathcal O}}
\newcommand{\Oo}{{\mathcal O}}
\newcommand{\OX}{{\Oo_X}}
\newcommand{\OY}{{\Oo_Y}}
\newcommand{\otY}{{\underset{\OY}{\ot}}}
\newcommand{\otX}{{\underset{\OX}{\ot}}}
\newcommand{\cU}{{\mathcal U}}\newcommand{\cX}{{\mathcal X}}
\newcommand{\cW}{{\mathcal W}}
\newcommand{\boldz}{{\mathbf Z}}
\newcommand{\qgr}{\operatorname{q-gr}}
\newcommand{\gr}{\operatorname{gr}}
\newcommand{\rk}{\operatorname{rk}}
\newcommand{\Sh}{\operatorname{Sh}}
\newcommand{\SH}{{\underline{\operatorname{Sh}}}}
\newcommand{\End}{\operatorname{End}}
\newcommand{\uEnd}{\underline{\operatorname{End}}}
\newcommand{\Hom}{\operatorname{Hom}}
\newcommand{\uHom}{\underline{\operatorname{Hom}}}
\newcommand{\uHomY}{\uHom_{\OY}}
\newcommand{\uHomX}{\uHom_{\OX}}
\newcommand{\Ext}{\operatorname{Ext}}
\newcommand{\bExt}{\operatorname{\bf{Ext}}}
\newcommand{\Tor}{\operatorname{Tor}}

\newcommand{\inv}{^{-1}}
\newcommand{\airtilde}{\widetilde{\hspace{.5em}}}
\newcommand{\airhat}{\widehat{\hspace{.5em}}}
\newcommand{\nt}{^{\circ}}
\newcommand{\del}{\partial}

\newcommand{\supp}{\operatorname{supp}}
\newcommand{\GK}{\operatorname{GK-dim}}
\newcommand{\hd}{\operatorname{hd}}
\newcommand{\id}{\operatorname{id}}
\newcommand{\res}{\operatorname{res}}
\newcommand{\lrar}{\leadsto}
\newcommand{\im}{\operatorname{Im}}
\newcommand{\HH}{\operatorname{H}}
\newcommand{\TF}{\operatorname{TF}}
\newcommand{\Bun}{\operatorname{Bun}}

\newcommand{\F}{\mathcal{F}}
\newcommand{\Ff}{\mathbb{F}}
\newcommand{\nthord}{^{(n)}}
\newcommand{\Gr}{{\mathfrak{Gr}}}

\newcommand{\Fr}{\operatorname{Fr}}
\newcommand{\GL}{\operatorname{GL}}
\newcommand{\gl}{\mathfrak{gl}}
\newcommand{\SL}{\operatorname{SL}}
\newcommand{\ff}{\footnote}
\newcommand{\ot}{\otimes}
\def\Ext{\operatorname {Ext}}
\def\Hom{\operatorname {Hom}}
\def\Ind{\operatorname {Ind}}
\def\bbZ{{\mathbb Z}}

\newcommand{\nc}{\newcommand}
\nc{\ol}{\overline} \nc{\cont}{\on{cont}} \nc{\rmod}{\on{mod}}
\nc{\Mtil}{\widetilde{M}} \nc{\wb}{\overline} \nc{\wt}{\widetilde}
\nc{\wh}{\widehat} \nc{\sm}{\setminus} \nc{\mc}{\mathcal}
\nc{\mbb}{\mathbb}  \nc{\K}{{\mc K}} \nc{\Kx}{{\mc K}^{\times}}
\nc{\Ox}{{\mc O}^{\times}} \nc{\unit}{{\bf \on{unit}}}
\nc{\boxt}{\boxtimes} \nc{\xarr}{\stackrel{\rightarrow}{x}}

\nc{\Ga}{\G_a}
 \nc{\PGL}{{\on{PGL}}}
 \nc{\PU}{{\on{PU}}}

\nc{\h}{{\mathfrak h}} \nc{\kk}{{\mathfrak k}}
 \nc{\Gm}{\G_m}
\nc{\Gabar}{\wb{\G}_a} \nc{\Gmbar}{\wb{\G}_m} \nc{\Gv}{G^\vee}
\nc{\Tv}{T^\vee} \nc{\Bv}{B^\vee} \nc{\g}{{\mathfrak g}}
\nc{\gv}{{\mathfrak g}^\vee} \nc{\RGv}{\on{Rep}\Gv}
\nc{\RTv}{\on{Rep}T^\vee}
 \nc{\Flv}{{\mathcal B}^\vee}
 \nc{\TFlv}{T^*\Flv}
 \nc{\Fl}{{\mathfrak Fl}}
\nc{\RR}{{\mathcal R}} \nc{\Nv}{{\mathcal{N}}^\vee}
\nc{\St}{{\mathcal St}} \nc{\ST}{{\underline{\mathcal St}}}
\nc{\Hec}{{\bf{\mathcal H}}} \nc{\Hecblock}{{\bf{\mathcal
H_{\alpha,\beta}}}} \nc{\dualHec}{{\bf{\mathcal H^\vee}}}
\nc{\dualHecblock}{{\bf{\mathcal H^\vee_{\alpha,\beta}}}}
\newcommand{\ramBun}{{\bf{Bun}}}
\newcommand{\ramBuno}{\ramBun^{\circ}}

\nc{\Buntheta}{{\bf Bun}_{\theta}} \nc{\Bunthetao}{{\bf
Bun}_{\theta}^{\circ}} \nc{\BunGR}{{\bf Bun}_{G_\R}}
\nc{\BunGRo}{{\bf Bun}_{G_\R}^{\circ}}
\nc{\HC}{{\mathcal{HC}}}
\nc{\risom}{\stackrel{\sim}{\to}} \nc{\Hv}{{H^\vee}}
\nc{\bS}{{\mathbf S}}
\def\Rep{\operatorname {Rep}}
\def\Conn{\operatorname {Conn}}

\nc{\Vect}{{\operatorname{Vect}}}
\nc{\Hecke}{{\operatorname{Hecke}}}

\newcommand{\ZZ}{{Z_{\bullet}}}
\nc{\HZ}{{\mc H}\ZZ} \nc{\eps}{\epsilon}

\nc{\CN}{\mathcal N} \nc{\BA}{\mathbb A}

\nc{\ul}{\underline}

\nc{\bn}{\mathbf n} \nc{\Sets}{{\on{Sets}}} \nc{\Top}{{\on{Top}}}
\nc{\IntHom}{{\mathcal Hom}}

\nc{\Simp}{{\mathbf \Delta}} \nc{\Simpop}{{\mathbf\Delta^\circ}}

\nc{\Cyc}{{\mathbf \Lambda}} \nc{\Cycop}{{\mathbf\Lambda^\circ}}

\nc{\Mon}{{\mathbf \Lambda^{mon}}}
\nc{\Monop}{{(\mathbf\Lambda^{mon})\circ}}

\nc{\Aff}{{\on{Aff}}} \nc{\Sch}{{\on{Sch}}}

\nc{\bul}{\bullet}
\nc{\module}{{\operatorname{-mod}}}
\nc{\grmodule}{{\operatorname{-grmod}}}
\nc{\fgmodule}{{\operatorname{-fgmod}}}
\nc{\topmodule}{{\operatorname{-cmod}}}

\nc{\dstack}{{\mathcal D}}

\nc{\BL}{{\mathbb L}}

\nc{\BD}{{\mathbb D}}

\nc{\BR}{{\mathbb R}}

\nc{\BT}{{\mathbb T}}

\nc{\SCA}{{\mc{SCA}}}
\nc{\DGA}{{DGA}}

\nc{\DSt}{{DSt}}

\nc{\lotimes}{{\otimes}^{\mathbf L}}

\nc{\bs}{\backslash}

\nc{\Lhat}{\widehat{\mc L}}

\newcommand{\Coh}{\rm{Coh}}

\nc\la{\langle}
\nc\ra{\rangle}

\nc{\QCoh}{\on{QC}}
\nc{\QC}{\on{QC}}
\nc{\Perf}{{\rm{Perf}}}
\nc{\Cat}{{\on{Cat}}}
\nc{\dgCat}{{\on{dgCat}}}
\nc{\bLa}{{\mathbf \Lambda}}

\nc{\RHom}{\mathbf{R}\hspace{-0.15em}\on{Hom}}
\nc{\REnd}{\mathbf{R}\hspace{-0.15em}\on{End}}
\nc{\colim}{\on{colim}}
\nc{\oo}{\infty}
\nc{\Mod}{{\on{Mod}}}

\nc\fh{\mathfrak h} \nc\al{\alpha} \nc\BGB{B\bs G/B}
\nc\QCb{{\on{QC}}^\flat} \nc\qc{\on{QC}}

\def\twD{\mathbf D}
\def\twR{\mathbf R}

\title{Loop Spaces and connections}

\author{David Ben-Zvi}
\address{Department of Mathematics\\University of Texas\\Austin, TX 78712-0257}
\email{benzvi@math.utexas.edu}
\author{David Nadler}
\address{Department of Mathematics\\Northwestern University\\Evanston, IL 60208-2370}
\email{nadler@math.northwestern.edu}

\begin{abstract}
  We examine the geometry of loop spaces in derived algebraic geometry
  and extend in several directions the well known connection between
  rotation of loops and the de Rham differential.  Our main result, a
  categorification of the geometric description of cyclic homology,
  relates $S^1$-equivariant quasicoherent sheaves on the loop space of
  a smooth scheme or geometric stack $X$ in characteristic zero with
  sheaves on $X$ with flat connection, or equivalently $\D_X$-modules.
  By deducing the Hodge filtration on de Rham modules from the
  formality of cochains on the circle, we are able to recover
  $\D_X$-modules precisely rather than a periodic version.  More
  generally, we consider the rotated Hopf fibration $\Omega S^3\to
  \Omega S^2\to S^1$, and relate $\Omega S^2$-equivariant sheaves on
  the loop space with sheaves on $X$ with arbitrary connection, with
  curvature given by their $\Omega S^3$-equivariance.
  \end{abstract}

\maketitle

\tableofcontents

\nc{\fg}{\mathfrak g}

\nc{\Map}{\on{Map}} \nc{\fX}{\mathfrak X}

\nc{\fingen}{{\text{{\it fg}}}}

\nc{\perf}{{\on{-perf}}}


\section{Introduction}\label{intro}

This is the first of a two paper series.\footnote{This paper is a much
  expanded version of half of the preprint~\cite{BN07}.}  It contains
general results on loop spaces in derived algebraic geometry.  It
establishes the intimate relation between equivariant sheaves on the
free loop space of a stack and sheaves with connection on the stack
itself.  One can view this as a categorification of the familiar
relation from topology and algebra between equivariant homology of
free loop spaces (cyclic homology) and topological cohomology of
spaces (singular or de Rham cohomology).  The second paper
\cite{reps} of the series contains an application of this theory to
flag varieties of reductive groups. In particular, we deduce
consequences for the representation theory of Hecke algebras and Lie
groups.

There is a surprising aspect to the relation between the theory
developed in this first paper and the applications undertaken in the
second.  It is well known that flat connections in the form of
$\D$-modules play a central role in representation theory. We show
here that arbitrary connections can be understood in terms of the
geometry of so called ``small loops'', or more precisely, loops in the
formal neighborhood of constant loops (for schemes, all loops are
small, but the distinction becomes important for stacks).
Thus it is not surprising that small loops provide an
alternative language to discuss geometric constructions of
representations.  But what is striking is that ``large loops"
naturally arise in the dual {Langlands parametrization} of
representations. For example, we will see in the second paper that the
notion of loop space organizes much of the seeming cacophany of spaces
parametrizing representations of Lie groups.

This paper is organized as follows. In the remainder of the
Introduction, we summarize our results (roughly following the overall
structure of the paper).  In Section \ref{prelim}, we review the basic
notions of $\oo$-categories and derived algebraic geometry which
provide the context of our arguments.  In Section~\ref{affinization},
we develop the derived analogue of affinization and study group
actions on affine derived schemes.  In Section~\ref{affine loop
  section}, we study equivariant sheaves on loop spaces of
schemes.  In Section~\ref{koszul dual section}, we provide a Koszul
dual description of the previous results.  Finally, in
Section~\ref{general loops section}, we generalize the preceding
results to loop spaces of geometric stacks.



\subsection{Loop spaces and odd tangent bundles}
The free loop space $\cL X = \Map(S^1, X)$ of a topological space $X$
comes equipped with many fascinating structures. Perhaps the most
fundamental is the $S^1$-action by loop rotation which recovers the
constant loops $X\subset \cL X$ as its fixed points. In this setting,
the theory of equivariant localization for $S^1$-actions, relating the
topology of a space with $S^1$-action to that of its fixed points, has
been applied to spectacular effect by Witten and many others. Our aim
is to contribute a categorified variation on this theme in the setting
of sheaves in algebraic geometry.

To approach our results, let's first review some well known facts
about functions and $S^1$-equivariant functions on loop spaces (for
which an excellent reference is \cite{Loday}, see also references
therein). The universal paradigm is that Hochschild homology and
cyclic homology of functions on a space provide a purely algebraic or
categorical approach to functions and $S^1$-invariant functions
respectively on its free loop space.  For example, for a topological
space $X$, Hochschild homology and cyclic homology of cochains on $X$
calculate the cohomology and $S^1$-equivariant cohomology of $\cL X$
\cite{Jones}. When applied to a commutative ring or more generally a
scheme $X$, Hochschild homology gives the Dolbeault cohomology
of differential forms on $X$ 
while cyclic homology gives the de Rham
cohomology of $X$. Thus we can view the de Rham cohomology of a scheme
$X$ as a realization of the $S^1$-invariant functions on some kind of
free loop space of $X$.

The above interpretation of de Rham cohomology is quite familiar in
the $\Z/2\Z$-graded supergeometry of mathematical physics. Namely,
consider the complex cohomology of the circle
$H^*(S^1,\C)\simeq\C[\eta]/(\eta^2)$, with $\eta$ of degree $1$, as
the $\Z/2\Z$-graded supercommutative ring of functions on the odd line
$\C^{0|1}$.  For a smooth manifold $X$, the mapping space
$\Map(\C^{0|1},X)$ is the odd tangent supermanifold $\BT_X[-1]$ with
functions the $\Z/2\Z$-graded algebra of differential forms
$\Omega_X^{-\bullet}=\Sym_{\cC^\infty_X} (\Omega_X[1])$.  Thus we see
that $\BT_X[-1]$ is a linear analogue of the free loop space: it
parametrizes maps from a linear analogue of the circle, and its
functions are the Hochschild homology of smooth functions on $X$.
Furthermore, the linear analogue of the $S^1$-action of loop rotation
is translation along the odd line $\C^{0|1}$.  The de Rham
differential $d$, thought of as an odd, square zero vector field on
$\BT_X[-1]$, is an infinitesimal generator of this action.  Thus we
see that the linear analogue of $S^1$-equivariant functions on the
loop space is the cyclic homology of functions on $X$ in the form of
its de Rham complex.

Similarly, in $\Z$-graded super-algebraic geometry, we can consider the spectrum of
the cohomology $H^*(S^1,\C)\simeq \C[\eta]/(\eta^2)$, with $\eta$ of
degree $1$. 
Note that  the cohomology is the {free}
supercommutative $\C$-algebra on $\eta$, and so its spectrum can be thought of as a shifted version
 $\aline[1]$
of the affine line.  
In other words, the cohomology
is the $\Z$-graded supercommutative ring of functions on
an infinitesimal of degree $-1$. 
To emphasize the group structure on $\aline[1]$, we will also often refer to it as the classifying stack $B\Ga$ (just as one uses the alternative notation $\Ga$ to emphasize the group structure on the 
affine line $\aline$).
For a smooth scheme $X$, the mapping
space $\Map(B\Ga,X)$ is the shifted tangent bundle $\BT_X[-1]$ with
functions differential forms $\Omega_X^{-\bullet}=\Sym_{\cO_X}
(\Omega_X[1])$ placed in negative (or homological) degrees, and the
linear analogue of loop rotation is again the degree $-1$, square zero
vector field given by the de Rham differential~$d$.

In algebraic geometry, there is another more traditional path
to the relation between free loop spaces and shifted
tangent bundles. Let's model the circle as a simplicial complex with
two points connected by two line segments, and try to interpret
concretely what maps from such an object to a smooth scheme
$X$ should be.  First, mapping two points to $X$ defines the product
$X\times X$.  One of the line segments connecting the points says the
points are equal: we should impose the equation $x=y$ that defines the
diagonal $\Delta\subset X\times X$. Then the other line segment says
the points are equal again, so we should impose the equation $x=y$
again, or in other words, take the self-intersection
$\Delta\cap\Delta\subset X\times X$. Of course, we could interpret
this naively as being a copy of $X$ again, however the intersection is
far from transverse: indeed the
tensor product 
$$
\Oo_{\Delta\cap \Delta} = \Oo_\Delta\otimes_{\Oo_{X\times X}} \Oo_\Delta
$$ needs to be
derived since there are higher Tor terms. 
The derived version of this tensor product is
in fact the definition of the Hochschild homology of $\cO_X$.
We can use the Koszul
complex to arrive at the homological Hochschild-Kostant-Rosenberg isomorphism
$$\Oo_\Delta\otimes^{\BL}_{\Oo_{X\times X}}\Oo_\Delta
\simeq 
\Omega_X^{-\bullet}.
$$
Thus  
we again find that the
shifted tangent bundle $\BT_X[-1]$ whose functions are
differential forms $\Omega_X^{-\bullet}$ 
plays  the role of  the free loop space of $X$.


\subsection{Setting of derived algebraic geometry}

We would like to generalize the ideas in the above discussion to a
 setting for algebraic geometry where the target $X$ could not only be a scheme
 but also perhaps a scheme with symmetry 
in the form of a stack. Thus we seek a theoretical framework that encompasses both 
stacks and the sophisticated derived intersections
 involved in calculating loop spaces.

Derived algebraic geometry, whose foundations have been developed
recently by To\"en and Vezzosi \cite{HAG1,HAG2} and Lurie
\cite{topoi,HA,dag7}, provides a natural setting in which
we can apply constructions from algebraic topology -- such as homotopy
colimits, homotopy limits, and mapping spaces -- to the classical
objects of algebraic geometry. 

Derived algebraic geometry generalizes classical algebraic geometry
simultaneously in two directions: it encompasses the theory of stacks,
allowing for derived versions of
quotients (and more general gluings or colimits) of schemes, and also
 the theory of differential graded schemes, allowing for derived
versions of intersections (and more general equations or limits). 
 In both cases,
passing to the derived or corrected versions of the operation guarantees that no
meaningful geometric information is lost. For example, if we were to work
within the theory of stacks alone, the mapping object $\Map(S^1, X)$
would make sense, but it would only recover $X$ not its odd tangent bundle
when $X$ is a smooth scheme.

In broad outline, to arrive at derived algebraic geometry from
classical algebraic geometry, one takes the following steps. 

First, one replaces commutative rings by a notion of derived
commutative rings. There are many choices with different geometric and
homotopical flavors, but for our applications we will work over a
fixed ground $\Q$-algebra $k$.  By a {derived ring}, we will mean a
commutative differential graded $k$-algebra which is cohomologically
trivial in positive degrees.  In other words, a derived ring is a
commutative ring object in non-positively graded cochain complexes
over $k$.
  Following usual conventions, we will use the term
affine derived scheme to mean an object of the opposite category of
derived rings. One can view a general derived scheme as an ordinary
scheme over $k$ equipped with a sheaf of derived rings whose zeroth
cohomology is identified with the structure sheaf of the scheme,
and whose other cohomology sheaves are quasicoherent.
Introducing derived rings allows for derived intersections where we
replace the tensor product of rings with its derived functor.

Second, one considers functors of points on derived rings that take
values not only in sets but in simplicial sets or equivalently
topological spaces. (For example, derived rings naturally give such
functors via the Dold-Kan correspondence.) This allows for derived
quotients by keeping track of gluings in the enriched theory of
spaces.  Roughly speaking, a {\em derived stack} is a functor from
derived rings to topological spaces satisfying an appropriate sheaf
axiom with respect to the \'etale topology on derived rings. Examples
include all of the schemes of classical algebraic geometry, stacks of
modern algebraic geometry, and topological spaces of homotopy theory
(in the form of locally constant stacks).  From this perspective, an
affine derived scheme is simply a representable functor.  Any derived
stack has an ``underlying'' underived stack, obtained by restricting
its functor of points to ordinary rings. In fact, derived stacks can
be viewed as formal thickenings of underived stacks, just as
supermanifolds are infinitesimal thickenings of manifolds.

One of the complicated aspects of the above theory is that the domain
of derived rings, target of topological spaces, and functors of points
themselves must be treated with the correct enriched homotopical
understanding.  We recommend To\"en's excellent survey \cite{Toen} for
more details and references.  It was our introduction to many of the
notions of derived algebraic geometry, in particular derived loop
spaces. Of particular relevance to this paper are the closed monoidal
structure on derived stacks (the Cartesian product and mapping space
functors on topological spaces induce analogous functors on derived
stacks), and the theory of cotangent complexes for Artin derived
stacks \cite{HAG2,HA}. The reader will find in Section~\ref{prelim} a
very brief synopsis of the theory of $\oo$-categories which provides
perhaps the most succinct categorical foundations.
%


\subsubsection{Free loop space}

The homotopy theory of topological spaces naturally maps to the
world of derived algebraic geometry via (the sheafification of) {
  constant} functors from derived rings to topological spaces.
  Any
topological space defines a derived stack (in fact, with trivial
derived structure) which we will refer to by the same name.  So in
particular, the integers $\Z$ define an abelian group stack, as does
its classifying space the circle $S^1=B\Z$.

With the circle in hand,
we define the {\em free loop space} of a derived stack $X$ to
be the derived mapping stack
$$
\cL X= \Map_{}(S^1,X).
$$ To $S$ a test derived ring, $\cL X$ assigns the space of maps
$\Spec S \times S^1 \to X $ between sheaves of topological spaces on
the \'etale topology of derived rings.

It follows that for $X=\Spec R$ the spectrum of a derived ring, the
loop space $\cL X = \Map(S^1, X)$ is equivalent to the spectrum of the
Hochschild chain complex
$$HC(R)= R\ot S^1=R\ot_{R\ot R} R$$ of the derived ring $R$.  In the
case that $R$ is an ordinary ring and regular, so that $X$ is an
underived smooth affine scheme, we can use the Koszul resolution to
calculate its loop space and arrive at the odd tangent bundle
$$\cL X\simeq \BT_X[-1]=\Spec \Sym_{\cO_X}(\Omega_X[1])$$ in agreement with
our discussion above. In particular, $\cL X$ is a derived thickening
of $X$ itself. We will see below that this picture extends to arbitrary
derived schemes.

On the other hand, for a stack $X$, it is easy to see that its loop
space $\cL X$ is a derived thickening of the {inertia stack} of $X$
 parametrizing points of $X$ equipped with automorphisms. For
example, the loop space $\cL (BG)$ of the classifying space $BG=pt/G$
of a group is the quotient $G/G$ by the conjugation action. We
discuss this example in more detail in Section \ref{stack intro}. In
\cite{reps}, we study in detail a collection of targets associated with
quotients of flag varieties.

\subsection{Loops in derived schemes}

We first focus on loop spaces of derived schemes, and discuss
generalizations to stacks below. By a derived scheme, we will always
mean a quasi-compact derived scheme with affine diagonal over a base
$\Q$-algebra $k$. The basic building blocks of derived algebraic
geometry are affine derived schemes, and all of our results for
schemes and stacks are built up from this base case via descent along
gluing constructions.

We show that the identification of loop spaces with odd tangent
bundles extends to arbitrary derived schemes, once we replace the
tangent bundle with the {\em tangent complex}, as developed in the
present context in \cite[1.4]{HAG2} and \cite[7.4]{HA} following
Quillen and Illusie. The tangent complex $\BT_X$ is the natural
derived analogue of the tangent bundle, for example in the sense that
$\BT_X$ can be identified with the derived mapping stack from the dual
numbers $\Spec k[\epsilon]/(\epsilon^2)$ to the derived stack $X$.
The tangent complex of an affine derived scheme is concentrated in
non-negative degrees, 
degrees) while the tangent complex of an underived smooth stack lies
in non-positive degrees.

We will take repeated advantage of the notion of {\em affinization} in
the sense of To\"en \cite{Toen affine} (reviewed and extended in
Section \ref{affinization}). The affinization $\Aff(X)$ of a derived
stack $X$ is defined by the universal property that any map $X\to
\Spec R$ to an affine derived scheme factors through a canonical map
$X\to \Aff(X)$.  The affinization of a topological space $S$ over a
ground ring $k$ captures the topology of $S$ seen by its cochains over
$k$. In particular, since we take $k$ to be a $\Q$-algebra, this
construction is (the $k$-linear) part of the rational homotopy theory
of $S$. (See \cite{Toen affine} for applications of affinization to
the algebro-geometric study of homotopy types.)

In the case of $S^1=B\Z$, the affinization map is given by the
canonical map
$$ 
\xymatrix{
S^1 =B\Z\ar[r] & B\Ga
}$$ induced by the canonical map from $\Z$ to the
affine line $\aline\simeq \Ga$ (we use the notation $\Ga$ to emphasize
the group structure on the affine line). In other words, the induced
pullback on functions
$$
\xymatrix{
 k[\eta]/(\eta^2) = \cO(B\Ga) \ar[r] & \cO (S^1) \simeq H^*(S^1, k),
\quad\mbox{with $\eta$ of degree $1$},
}
$$ is an equivalence.  

It is easy to see that loop spaces are local objects in the Zariski
topology (see Lemma~\ref{lem affine}). As a result, many results on
loops in affine derived schemes carry over without change to arbitrary
derived schemes (which we always assume to have affine diagonal).  (On
the other hand, loops are not local in the \'etale topology. We will
turn to the more complicated situation of loops in stacks in Section
\ref{stack intro} of the introduction.)  Hence as a consequence of
affinization, we are able to establish (as Proposition
\ref{prop loops=odd tangents}) the following form  of the homological
Hochschild-Kostant-Rosenberg theorem for loops in derived schemes.

\begin{prop}
\label{state loops=odd tangents}
Fix a derived scheme $X$ over $k$, and consider its cotangent complex
$\Omega_X$.

  (1) The mapping derived stack $ \Map(B\Ga, X) $ is equivalent to the
  odd tangent bundle $\BT_X[-1] = \Spec \Sym_{\cO_X} (\Omega_X[1])$.

  (2) The affinization map $S^1 \to B\Ga$ induces an equivalence
$$\xymatrix{
\Map(B\Ga, X)=\BT_X[-1] \ar[r]^-\sim& \cL X=\Map(S^1,X).}
$$ Thus we have a canonical Hochschild-Kostant-Rosenberg equivalence
  for Hochschild chains $$HC(\cO_X) \simeq \Sym_{\cO_X}(
  \Omega_X[1]).$$

\end{prop}


\subsection{Rotating loops and the de Rham
  differential}
  
 Next we turn to the $S^1$-action on the loop space $\cL X= \Map(S^1,
 X)$ of a derived scheme by loop rotations, or in other words, the action
 induced by the regular $S^1$-action on the domain $S^1$.
    
Let us first discuss the general notion of group actions on
stacks. The action of a group derived stack $G$ on a derived stack $X$
involves an action map $G \times X\to X$ together with associativity
and higher coherences.  When $G= S^1$, the identification $S^1=B\Z$
implies the action map $S^1 \times X\to X$ is the same as an
automorphism (or self-homotopy) of the identity map of $X$.  When $X$
is a derived scheme and $G= B\Ga$, the action map $B\Ga\times X\to X$
is the same as a section of the odd tangent bundle $\BT_X[-1]$, or in
other words, a vector field of degree $-1$. The associativity
constraint for a $B\Ga$-action is equivalent to requiring the square of this vector
field to be homotopic to zero. In general, there are
infinitely many higher coherences that need to be specified to give an
action. We will see below that in the case of loop spaces, all of the
higher coherences are encoded in these initial data.

As another consequence of affinization,
we are able to 
identify $S^1$-actions on affine derived schemes with $B\Ga$-actions
 (see Corollary~\ref{mixed algebras} below, or \cite{TV
  cyclic} for a model categorical version).  The result can be viewed
as a nonlinear analogue of the classical identification 
(see \cite{Kassel,Loday,Dwyer Kan}, or 
Corollary \ref{mixed modules} below) between $k$-modules with an $S^1$-action
(cyclic modules) and $k$-modules with a square zero endomorphism of
degree $-1$ (mixed complexes).
%
When we apply this to the canonical $S^1$-action on the
loop space $\cL X$ of a derived scheme, we obtain the following (as
Proposition~\ref{prop rotation=de rham}).

\begin{prop}
\label{rotation is de Rham} Let $X$ be a derived scheme
over $k$.
  Under the identification $\cL X\simeq \BT_X[-1]$, the
  rotation $S^1$-action on $\cL X$ corresponds to the translation
  $B\Ga$-action on $\BT_X[-1]$.  In turn, the vector
  field of degree $-1$ coming from the $B\Ga$-action on $\BT_X[-1]$ is
  given by the {de Rham differential}.
\end{prop}



The natural $\Gm$-action on $\Ga$, and hence $B\Ga$, induces a
corresponding $\Gm$-action on the mapping space $\Map(B \Ga, X)$, and
thus on the loop space $\cL X$.  From the perspective of the
affinization identification $B\Ga \simeq \Aff(S^1)$, the $\Gm$-action
on $B\Ga$ results from the formality $\cO(S^1) \simeq H^*(S^1, k)$.
Under the identification $\BT_X[-1]\simeq \Map(B \Ga, X)$, the
$\Gm$-action on $\BT_X[-1]$ coincides with the natural dilation
$\Gm$-action. Equivalently, it recovers the grading on the algebra
$\Sym_{\cO_X}(\Omega_X[1])$, and hence the Hodge filtration on the de
Rham complex.  Keeping track of this $\Gm$-action leads to formality
statements that significantly simplify the equivariant geometry of
loop spaces. In particular, we deduce the following uniqueness
statement in Section~\ref{formal section} from an examination of the
dilation $\Gm$-weights on $\Omega_X^{-\bul}$ (see also Remark \ref{TV remark}).

\begin{thm}\label{formal uniqueness}
For $X$ a smooth scheme over $k$, the tautological $B\Ga$-action on
$\BT_X[-1]=\Map(B\Ga,X)$ is the unique such action (up to equivalence)
extending the de Rham differential and compatible with $\Gm$-actions.
\end{thm}

\begin{remark}
Let us remind the reader that local notions of morphisms of schemes
naturally extend to derived schemes in the following way. Given a map
of derived rings $R\to S$, one can ask that the induced map of
discrete rings $\pi_0(R) \to \pi_0(S)$ satisfies the property of
interest, and that the higher homotopy groups of $S$ are simply base
changes of those of $R$ along the morphism $\pi_0(R) \to \pi_0(S)$.

As a special case, a derived scheme $X\to \Spec k$ is smooth if and
only if it is a smooth underived scheme.
\end{remark}


\subsection{Sheaves on loop spaces}
Now we will consider the interaction of loop rotation with
quasicoherent sheaves on loop spaces. Quasicoherent sheaves on a
derived stack $X$ form a stable $\oo$-category defined as follows. 

For
$X=\Spec R$ an affine derived scheme, we set $\QC(X)=\Mod_R$, the
$\oo$-category of $R$-modules. Objects of $\Mod_R$ are
 differential graded modules over the differential graded $k$-algebra
$R$, and morphisms are module maps localized with respect to quasi-isomorphisms. 
 Given a map $f:\Spec S\to \Spec R$, the pullback
$f^*:\QC(\Spec R)\to \QC(\Spec S)$ is
the usual base change functor $f^*(M) = M\ot_R S$. 

We can tautologically write any derived stack as a  colimit
of affine derived schemes, and then extend the above assignment to a
colimit-preserving functor $\QC$ from the $\oo$-category of
derived stacks to the opposite of the $\oo$-category of stable
$\oo$-categories. So informally speaking, a quasicoherent sheaf on a derived stack $X$
is a compatible collection of modules on all affine derived schemes over $X$.

\subsubsection{$S^1$-equivariant sheaves}
An $S^1$-action on a derived stack $Z$ induces an automorphism of the
identity functor of $\qc(Z)$.  For any quasicoherent sheaf $\cF\in
\qc(Z)$, the resulting automorphism of $\cF$ is given by the monodromy
of $\cF$ along the $S^1$-orbits.  To equip $\cF$ with an
$S^1$-equivariant structure is to trivialize this automorphism, or
equivalently, to kill the variation (the difference between the
automorphism and the identity).  In what follows, we write $\qc(X)^{S^1}$ for the stable $\oo$-category
of equivariant quasicoherent sheaves.

In the case of the loop space $\cL X\simeq {\BT_X}[-1]$ of a derived
scheme $X$, its functions are the de Rham algebra $\Omega_X^{-\bullet}
= \Sym_{\cO_X}(\Omega_X[1])$, and loop rotation corresponds to the
vector field of degree $-1$ given by the de Rham
differential. The following theorem demonstrates that in the smooth
and underived setting, these data fully account for $S^1$-equivariant
sheaves. Throughout what follows, we keep track of the additional natural $\Gm$-weight grading 
to that $\Omega_X^{-\bullet}$
sits along the diagonal in the cohomological degree and $\Gm$-weight bigrading.
In particular, the generating one-forms $\Omega_X[1] \subset \Omega_X^{-\bullet}$ sit in cohomological degree $-1$ and 
$\Gm$-weight grading $-1$.

Let $\Omega_X^{-\bullet}[d]$ be the $\Z$-graded sheaf of differential graded $k$-algebras
on $X$ obtained from the de Rham algebra by adjoining the de Rham
differential in cohomological degree $-1$ and $\Gm$-weight grading $-1$ , so that $[d, \omega] = d\omega$, for
$\omega \in \Omega^{-\bullet}_X$, and $d^2=0$.  We write
$\Omega_X^{-\bullet}[d]\module$ for the stable $\oo$-category of
quasicoherent sheaves on $X$ equipped with a compatible structure of
differential graded
$\Omega_X^{-\bullet}[d]$-module. 
We also append the subscript $\Z$ to the notation to denote
differential graded modules with an additional compatible $\Gm$-weight grading.

\begin{thm}\label{basic thm}
  For a smooth scheme $X$ over $k$, there are canonical
  equivalences
$$ \QCoh(\cL X)^{S^1} \simeq \Omega_X^{-\bullet}[d]\module$$ 
   $$ \QCoh(\cL X)^{B\Ga\rtimes \Gm} \simeq \Omega_X^{-\bullet}[d]\module_\Z$$
\end{thm}
%

\begin{remark}
The finiteness and formality of the cochain algebra $C^*(S^1)$ are the primary technical inputs to the theorem.
\end{remark}

\begin{remark} For $X$ an arbitrary derived scheme with 
affine diagonal, we obtain a version of Theorem \ref{basic thm}
describing $S^1$-equivariant sheaves on $\cL X$ as modules over
$\Omega_X^{-\bullet}$ adjoined chains on $S^1$, whose action is related
via Proposition \ref{rotation is de Rham} with the de Rham
differential. In the absence of a more naive description, such as
given by the formality of Theorem \ref{formal uniqueness}, we may
interpret this as a definition of de Rham modules on $X$.
\end{remark}

\begin{remark}
Later in the introduction, we will interpret the graded equivalence of
the theorem as providing the Hodge filtration on the $\oo$-category of
de Rham modules on $X$.
\end{remark}

\begin{remark}\label{TV remark} Theorem~\ref{basic thm} and the results preceding it are sketched in the preprint
\cite{BN07}.  An alternative approach to the necessary foundations, in
particular the formality of Theorem \ref{formal uniqueness} (in fact,
a model categorical refinement thereof) is developed in the recent
work \cite{TV cyclic}, and applied in \cite{TV Chern long}.
\end{remark}

In differential geometry, there is a familiar relationship between 
quasicoherent sheaves $\cE$ with flat connection $\nabla$ on $X$ and modules for 
 $\Omega_X^{-\bullet}[d]$.
Namely, 
the connection 
$\nabla:\cE\to \cE\ot\Omega_X$ naturally 
extends to an $\Omega_X^{-\bullet}[d]$-module structure on the pullback
$$
\pi^*\cE = \cE\ot_{\cO_X}\Omega_X^{-\bullet}.
$$
The action of the 
distinguished element $d$
is given by the formula
$$
d(\sigma\ot \omega) = \nabla(\sigma)\wedge \omega + \sigma\ot d\omega,
\qquad
\sigma\ot \omega \in\cE\ot_{\cO_X}\Omega_X^{-\bullet}.
$$ The fact that $\nabla$ is flat leads directly to the identity
$d^2=0$.  To further refine this to an equivalence, we will consider a
suitable form of Koszul duality below.


\subsubsection{Hopf fibration}
We now examine what happens when we replace equivariance for circle
actions by a rougher notion of invariance under rotation.  In the case
of sheaves on loop spaces, we will find that this precisely corresponds
to allowing curvature.

Recall that the action of a monoid $G$, with identity $e\in G$, on a
scheme or stack $X$ consists of a pointed action map $G\times X\to X$,
so that $\{e\}\times X\to X$ is homotopic to the identity, together
with an associativity constraint and a collection of higher
coherences. We can begin to unwind this data by first considering what
information is contained in the pointed action map alone. When we take
$G$ to be the circle $S^1$ or its affinization $B\Ga$, we are asking to what extent we can
understand actions from automorphisms of the identity or odd vector
fields respectively.

For any pointed derived
stack $Y_*$, we can apply 
a variant of the James construction
in topology and introduce the free monoid on $Y_*$ via based loops
in the suspension
 $$F(Y_*) = \Omega \Sigma Y_*.$$ 
  Then a pointed map $Y_*\times X\to X$ is equivalent to 
an action of the monoid $F(Y_*)$ on $X$. If $Y_*$ itself has a monoid
structure then we also have a canonical homomorphism $F(Y_*)\to Y_*$, and
we can measure obstructions for an $F(Y_*)$-action to descend to a
$Y_*$-action.

 Now let's focus on the case $Y_*= S^1$ pointed by the identity so that
 $$F(S^1)\simeq\Omega\Sigma S^1\simeq\Omega S^2.$$ The canonical map
 $\Omega S^2\to S^1$ determined by the monoidal structure on $S^1$ is
 given by applying $\Omega$ to the standard inclusion 
 $$
 \xymatrix{
 S^2\simeq \C \pline \ar[r] &
 BS^1\simeq \C {\bf P}^\infty.}
 $$ In turn, this map is the classifying
 map for the Hopf fibration 
 $$\xymatrix{
 S^1\ar[r]& S^3\ar[r] & S^2.}
 $$

Now we have the following variation on Theorem~\ref{basic thm}. In fact, it is a simpler homotopical statement and our arguments do not require that $X$ be smooth or underived.
In agreement with our previous conventions,
we will work with the cohomology 
$$
\cO(S^2) \simeq C^*(S^2, k)  \simeq H^*(S^2, k) \simeq k[u]/(u^2)
$$
as a graded algebra with $u$ in cohomological degree $2$ and $\Gm$-weight grading $1$.
Consider the $\Z$-graded sheaf of 
differential graded $k$-algebras on $X$ given by the tensor product 
$$
\xymatrix{
\Omega_{X, S^2}^{-\bullet} = \Omega_{X}^{-\bullet}\otimes \cO(S^2) \simeq \Omega_{X}^{-\bullet}[u]/(u^2)
}
$$
with differential 
$$
\xymatrix{
\delta(\omega \otimes 1) = d\omega \otimes u & \delta(\omega \otimes u) = 0
}
$$
where as usual $d$ denotes the de Rham differential.
One arrives at $\Omega_{X, S^2}^{-\bullet}$ by taking functions along the fibers of the map 
$$
\xymatrix{
\cL X/\Omega S^2 \ar[r] & S^2 \ar[r] &  pt.
}
$$
 Our particular presentation comes from thinking of the $E_2$-page of the spectral sequence for the first map
 in the above sequence.

We write $\Omega_{X, S^2}^{-\bullet}\module$ for the stable $\oo$-category of quasicoherent
sheaves on $X$ equipped with a compatible structure of 
differential graded
$\Omega_{X, S^2}^{-\bullet}$-module.
As before, we also append the subscript $\Z$ to the notation to denote
differential graded modules with an additional compatible weight grading.

\begin{thm} \label{intro free thm}
For a derived scheme $X$ over $k$, there are canonical equivalences
$$ \QCoh(\cL X)^{\Omega S^2} \simeq 
\Omega_{X, S^2}^{-\bullet}\module
$$ 
$$ \QCoh(\cL X)^{\Aff(\Omega S^2)\rtimes\Gm} \simeq
\Omega_{X, S^2}^{-\bullet}\module_\Z
$$
\end{thm}

\begin{remark}
The finiteness and formality of the cochain algebra $C^*(S^2)$ are the primary technical inputs to the theorem.
\end{remark}

We will interpret the theorem in terms of sheaves with not necessarily flat connection in the section immediately following.


\subsection{Koszul dual interpretation: relation to $\D$-modules}\label{Koszul intro}
Throughout the discussion here, we assume that $X$ is a smooth
underived scheme.

As discussed above, modules over the de Rham complex are intimately
related to sheaves with flat connection or $\D$-modules.  In fact, as
explained by \cite{Kap dR, BD Hitchin}, there is a natural Koszul
equivalence between appropriate derived categories of modules over the
de Rham complex and $\D$-modules.  This is a noncommutative version of
the familiar pattern relating sheaves on an odd tangent bundle
(modules over an exterior algebra) and sheaves on an even cotangent
bundle (modules over a symmetric algebra). More specifically, it is a
case of the nonhomogeneous quadratic duality studied in
\cite{quadratic duality, quadratic algebras} which exchanges
differential graded quadratic algebras (in particular, the de Rham
complex and its curved versions) and filtered quadratic algebras.  In
our setting, Koszul duality arises from $S^1$-equivariance, as in~\cite{GKM}: the exterior algebra arises as the chains on $S^1$, and
its Koszul dual as the cochains on $BS^1$.

In Section \ref{graded section} below, we discuss the general yoga of
categories of graded modules over graded algebras and their
independence of shears of grading, which accounts for the different
conventions of \cite{BGS} and \cite{GKM}, as well as the notion of
periodic localization (as it appears in cyclic homology theory).  We
apply these notions in Section \ref{koszul dual section}, where we
discuss Koszul duality for de Rham modules. In our setting, the Koszul
dual to the de Rham algebra $\Omega^{-\bul}_X = \Sym_{\cO_X}(\Omega_X[1])$ is
the bigraded symmetric algebra $\Sym_{\cO_X}(\BT_X[-2])$. Turning on the the de
Rham differential is Koszul dual to a noncommutative deformation: the
Koszul dual to the extended algebra $\Omega^{-\bul}_X[d]$ is the
corresponding cohomologically graded Rees algebra $\cR_X$ built out of
the canonical filtration of the algebra of differential
operators. Furthermore, the algebras have natural additional weight
gradings coming from dilations, and are Koszul dual as bigraded
algebras. Thus appropriate derived categories of graded
modules over the algebras are equivalent.

\begin{remark}
Any complications in the statements below stem from the familiar issue in Koszul duality that
 one must impose some finiteness  on sheaves to obtain equivalences of
 categories. Geometrically speaking, we need to impose such conditions in
 order for the augmentation section to see all of the structure of the sheaves. 
\end{remark}
 
%
 
 \subsubsection{Flat modules}
 
 First, we can complete the cotangent bundle along the
 zero section, and consider complete modules over the corresponding
 complete graded Rees algebra $\wh \cR_X$.   
 We write $\wh \cR_X\topmodule$ for the stable $\oo$-category of quasicoherent
sheaves on $X$ equipped with the structure of complete $\wh \cR_X$-module.
The following results from Theorem~\ref{basic  thm} by considering the natural $\cO_{\cL X}$-augmentation module $\cO_X$.
 
\begin{corollary} For  a smooth underived scheme $X$,
there are canonical equivalences of stable $\oo$-categories
$$
 \qc(\cL X)^{B\Ga \rtimes \G_m}
 \simeq
\wh\cR_{X}\topmodule_\Z
$$
$$
 \qc(\cL X)^{B\Ga}
 \simeq
\wh\cR_{X}\topmodule
$$
\end{corollary}

Alternatively, we can restrict
 to stable $\oo$-categories of  suitably finite modules.
 Let us write  $\Perf_X(\cL X) $ for the small, stable $\oo$-subcategory of $\qc(\cL X)$ of quasicoherent
sheaves whose pushforward along the canonical map $\cL X\to X$ are perfect.
Note that the structure sheaf $\cO_{\cL X} \simeq \Sym_{\cO_X}(\Omega_X[1])$ is perfect over 
$\cO_X$, and the $\cO_{\cL X}$-augmentation module $\cO_X$ is obviously perfect over $\cO_X$ (though not over $\cO_{\cL X}$). 
 We write $\cR_X\perf$ for the stable $\oo$-category of quasicoherent
sheaves on $X$ equipped with a compatible structure of perfect $ \cR_X$-module.

\begin{corollary}\label{finite version}
For  a smooth underived scheme $X$,
there are canonical equivalences
 of small, stable $\oo$-categories 
 $$
 \Perf_{X}(\cL X)^{B\Ga \rtimes \G_m}
\simeq
\cR_{X}\perf_\Z.
$$
$$
 \Perf_X(\cL X)^{B\Ga}
\simeq
\cR_{X}\perf.
$$
\end{corollary}

\subsubsection{Periodic localization}

To recover $\D_X$-modules from $\cR_X$-modules, we must pass to {\em
  periodic} modules in the equivalence of Corollary~\ref{finite version}.  All of the categories are compatibly linear over the ring
$H^*(BS^1) \simeq k[u]$, with $u$ of cohomological degree $2$ and $\Gm$-weight grading $1$.  On the one hand, any
category of $S^1$ or $B\Ga$-equivariant sheaves is linear over the
equivariant cohomology ring $\cO(B\Ga) \simeq H^*(BS^1)$.  On the
other hand, by construction, $k[u]$ maps to the graded Rees algebra
$\cR_X$ as a central subalgebra.
 
 When we invert $u$ and localize the categories, we will denote the
 result by the subscript ${per}$.  In general, since $u$ is of the
 cohomological degree $2$, a periodic category is only
 $\Z/2\Z$-graded. But when we localize a mixed category compatibly
 graded by $\Gm$-weight, the periodic category maintains a usual cohomological degree.
Using the shear equivalence of Section \ref{graded
   section}, we find that the $\oo$-category of $B\Ga\rtimes
 \Gm$-equivariant sheaves on the loop space naturally has the
 structure of $\Gm$-equivariant category over the ungraded line
 $\aline$. We can then simply realize periodic localization as
 restriction to the open $\Gm$-orbit $\Gm\subset \aline$.

\begin{corollary}\label{small corollary}
There are canonical equivalences
 of small, stable $\oo$-categories 
$$
\Perf_X(\cL X)^{B\Ga \rtimes \G_m}_{per}
   \simeq \cD_X\perf.
$$
$$
\Perf_X(\cL X)^{B\Ga}_{per}
 \simeq \cD_X\perf \ot_k k[u, u^{-1}].
$$
The second equivalence is of $\Z/2\Z$-periodic stable $\oo$-categories.
\end{corollary}

Thus we think of loop
spaces and their circle action as a useful geometric counterpart to
cotangent bundles and their quantization.  Our original motivation for
this story came from the observation that in applications to
representation theory, it is often easier to identify the de Rham
differential rather than the ring of differential operators.

\subsubsection{The Hodge filtration}\label{Hodge}
As we discussed above, the $\oo$-category $\qc(\cL X)^{B\Ga\rtimes
  \Gm}$ of graded de Rham modules naturally localizes over the stack
$\aline/\Gm$, where the coordinate on the line is identified (after
shear of grading) with the generator for the cohomology of $BS^1$.
Recall that by the general yoga of the Rees construction, objects over
$\aline/\Gm$ should be thought of as filtrations on their restrictions
to $pt = \Gm/\Gm \subset \aline/\Gm$. In particular, the Hodge filtration on nonabelian de Rham
cohomology of a scheme $X$ was defined by Simpson \cite{Simpson} in
this fashion, replacing local systems or $\D_X$-modules by their
filtered versions as Rees modules. (See also \cite{ST}
for the related development of the general theory of $\D$-modules on
stacks.)
The Koszul duality of de Rham
modules and Rees modules above identifies Simpson's Hodge filtration
 with our filtration on de Rham modules. We thus are led to the
following interpretation of our results: 

\begin{corollary}
The $\oo$-category $\Perf_X(\cL X)^{B\Ga\rtimes \Gm}$ over
$\aline/\Gm$ coincides with Simpson's Hodge filtration on nonabelian
de Rham cohomology (specifically, of the $\oo$-category of perfect $\D_X$-modules).
\end{corollary}


\subsubsection{Curved modules}
We have seen that $S^1$-equivariant sheaves on the loop space $\cL X$
are equivalent to suitable modules for the graded Rees algebra $\cR_X$. To arrive at this, we first
described all  $S^1$-equivariant sheaves in terms of the extended de Rham algebra $\Omega^{-\bullet}_X[d]$,  and then applied Koszul duality with respect to the $\Omega^{-\bullet}_X[d]$-augmentation module $\cO_X$.

For ${\Omega S^2}$-equivariant 
 sheaves, we have a more direct path. Following \cite{BB}, we define the complete subprincipal graded Rees
algebra to be the base change
$$
\wh\cR^{sp}_{X}=
\wh\cR_X \otimes_{\cO(BS^1)} \cO(S^2) \simeq \wh\cR_X \ot_{k[u]} k[u]/(u^2)
$$
over the central parameter.  We write $\wh\cR^{sp}_X\topmodule$ for the stable $\oo$-category of quasicoherent
sheaves on $X$ equipped with a compatible structure of complees $ \cR^{sp}_X$-module.
As before, we also append the subscript $\Z$ to the notation to denote
compatibly $\Z$-graded modules.
The following results from Theorem~\ref{intro free thm} by considering the natural $\Omega^{-\bullet}_{X, S^2}$-augmentation module~$\cO_X \otimes \cO(S^2)$.

\begin{corollary} 
For a smooth underived scheme $X$, there are canonical equivalences
$$ \QCoh(\cL X)^{\Omega S^2} \simeq 
\wh\cR^{sp}_{X}\module
$$ 
$$ \QCoh(\cL X)^{\Aff(\Omega S^2)\rtimes\Gm} \simeq
\wh \cR^{sp}_{X}\module_\Z
$$
\end{corollary}

Alternatively, we can consider the natural $\Omega^{-\bullet}_{X, S^2}$-augmentation module~$\Omega^{-\bullet}_{X}$.
Let $\Omega_X^{-\bullet}\langle d\rangle$ be the $\Z$-graded sheaf of 
differential graded $k$-algebras on $X$ 
obtained from the de Rham algebra by adjoining
  an element $d$ of degree $-1$ with $[d, \omega] =
  d\omega$, for $\omega \in \Omega^{-\bullet}_X$, but with no equation
  on powers of $d$.
  We write $\Omega_X^{-\bullet}\langle d\rangle\topmodule$ for the stable $\oo$-category of quasicoherent
sheaves on $X$ equipped with a compatible structure of 
complete differential graded
$\Omega_X^{-\bullet}\langle d\rangle$-module.
As before, we also append the subscript $\Z$ to the notation to denote
compatibly $\Z$-graded modules.

\begin{corollary}
For a smooth underived scheme $X$, there are canonical equivalences
$$ \QCoh(\cL X)^{\Aff(\Omega S^2)} \simeq 
\Omega_X^{-\bullet}\langle d\rangle\topmodule
$$ 
$$ \QCoh(\cL X)^{\Aff(\Omega S^2)\rtimes\Gm} \simeq
\Omega_X^{-\bullet}\langle d\rangle\topmodule_\Z
$$
\end{corollary}

Recall the discussed relationship between quasicoherent sheaves $\cE$
with flat connection $\nabla$ on $X$ and modules for
$\Omega_X^{-\bullet}[d]$. It generalizes to a similar relationship
between quasicoherent sheaves $\cE$ with not necessarily flat
connection $\nabla$ and modules for $\Omega_X^{-\bullet}\langle 
d \rangle$.  Namely, 
the connection $\nabla:\cE\to \cE\ot\Omega_X$ naturally extends to an
$\Omega_X^{-\bullet}\langle d \rangle$-module structure on the pullback
$$
\pi^*\cE = \cE\ot_{\cO_X}\Omega_X^{-\bullet}.
$$
The action of the 
distinguished element $d$
is given by the formula
$$
d(\sigma\ot \omega) = \nabla(\sigma)\wedge \omega + \sigma\ot d\omega,
\qquad
\sigma\ot \omega \in\cE\ot_{\cO_X} \Omega_X^{-\bullet}.
$$
Here we do not require that $\nabla$ is flat nor correspondingly that the square of  $d $ vanishes. The square of $d $ is given by the curvature 
$R = \nabla^2:\cE \to \cE \ot \Omega_X^{\wedge 2}$.

\subsubsection{Curvature}
Before continuing on, we pause here to informally interpret the curvature of
connections from the viewpoint of loop spaces.
We will place ourselves in the context of the above discussed Koszul duality, and in particular 
only have in mind sufficiently finite modules.

Let us return to the Hopf fibration $S^1\to S^3 \to S^2$ and consider  its rotation
$$
\xymatrix{
\Omega S^3 \ar[r] \ar[d]& \Omega S^2 \ar[d]\\
\{e\}\ar[r] & S^1
}
$$
 Since the functor
 $\Omega$ preserves limits, the above is a pullback square of
monoids. 

On the one hand, by restriction,
an $\Omega S^2$-equivariant sheaf $\cF$ 
on $\cL X$ may be considered as an $\Omega S^3$-equivariant sheaf. Note that the $\Omega S^3$-action on the loop
space is trivial, since the $\Omega S^2$-action descends to $S^1$.
Thus the $\Omega S^3$-equivariant structure on $\cF$ is in fact $\cO_{\cL X}$-linear.  Moreover, since
$\Omega S^3\simeq F(S^2)$ is free, the $\Omega S^3$-equivariant structure  is nothing more than a degree
$2$ endomorphism of $\cF$. When $\cF$  is
the pullback of a sheaf $\cE$ on $X$ with $\Omega S^2$-equivariant
structure given by a connection $\nabla$ on $\cE$, it is immediate
that the endomorphism is precisely multiplication by the curvature form of $\nabla$.

On the other hand, the rotated Hopf fibration is not a pushout diagram
of monoids. In fact, a resolution of $S^1$ by free monoids has
infinite length, since $BS^1=\C {\bf P}^\infty$ is not a finite
complex. Given a space with $\Omega S^2$-action, or in other words, a
space over $B(\Omega S^2)=\C\pline$, the induced $\Omega S^3$-action
is the obstruction to extending to a space over the projective plane
$\C\pplane$. After this, there still remain infinitely many
obstructions to overcome to extend to a space over $BS^1=\C {\bf
  P}^\infty$, or in other words, a space with $S^1$-action. (These
obstructions are related to the Massey products studied in
\cite{Stasheff}.) Nevertheless, in our setting a strong formality
statement holds, resulting in the contractibility of all higher
data. Namely, in proving Theorem~\ref{basic thm}, we show that an
$\Omega S^2$-equivariant sheaf on $\cL X$ equipped with a
trivialization of the equivariant structure of the kernel $\Omega S^3$
canonically descends to an $S^1$-equivariant sheaf. In down to earth
terms, a flat connection is simply a connection with vanishing
curvature form.
 
Now consider a quasicoherent sheaf $\cE$ with connection $\nabla$ on
$X$. Assume that the curvature $R = \nabla^2:\cE \to \cE \ot
\Omega_X^{\wedge 2}$ is central (or scalar) in the sense that it
factors as a tensor product $R = \id_\cE\ot \omega$, for some $\omega
\in \Omega_X^{\wedge 2}$.  Then the $\Omega_X^{-\bullet}\langle
  d\rangle$-module structure on the pullback
$$
\pi^*\cE = \cE\ot_{\cO_X} \Omega_X^{-\bullet}
$$ descends to the localization where $d^2 $ acts by
multiplication by $\omega$.  Thus if rather than trivializing the
$\Omega S^3$-equivariant structure, we insist that it factors through
a specified character, then we obtain sheaves with connection with
specified central curvature.

To summarize, sheaves with connection on $X$ naturally give $\Omega
S^2$-equivariant sheaves on $\cL X$ with curvature given by the
$\Omega S^3$-equivariance. When the curvature vanishes, or in other words, the
connection is flat, they descend to $S^1$-equivariant sheaves.  When
the curvature is central, they can be understood as $\Omega
S^2$-equivariant sheaves on which $\Omega S^3$ acts by a character.

\subsection{Interpretation: categorified cyclic homology}
For the moment, let us suspend the assumption that we are working over a fixed $\Q$-algebra $k$.
The definitions and results quoted immediately below hold over any derived ring 
(and ultimately over the sphere spectrum).

The $\oo$-category $\qc(\cL X)$ of quasicoherent sheaves on a loop
space $\cL X$ carries a rich collection of structures in addition to
loop rotation.  In a joint paper with J.~Francis~\cite{BFN}, we
studied the structures on $\QC(\cL X)$ induced by the tensor product
on $\qc(X)$ and the composition of loops $\cL X\times_X \cL X\to \cL
X$.  We introduced the class of {\em perfect stacks} which includes
all (quasicompact, quasiseparated) schemes, and most commonly arising
stacks in characteristic zero. We then established the following
characterization of $\qc(\cL X)$ for a perfect stack $X$, as an object
of the symmetric monoidal $\oo$-category $\St$ of presentable stable
$\oo$-categories.

\begin{thm}[\cite{BFN}] Let $X$ be a perfect stack. The $\infty$-category 
$\QCoh(X)$ is self-dual as an object of $\St$.  The $\infty$-category
  $\QCoh(\cL X)$ is naturally equivalent both to the Hochschild
  cohomology category (or Drinfeld center) and the Hochschild homology
  category of $\QC(X)$ as an algebra in $\St$.  As a Drinfeld center,
  $\qc(\cL X)$ is endowed with a canonical $\cE_2$ (or braided tensor)
  structure, and this structure agrees with the $\cE_2$-structure
  induced by the composition of loops.
\end{thm}

Thus in the context of the above theorem, we can view
Theorem~\ref{basic thm} as a statement about the categorified version
of the {cyclic} structure on Hochschild chains discovered by Connes
\cite{Connes}.

To explain this, let $\cC$ denote a symmetric monoidal
$(\oo,1)$-category and $A\in\cC$ an algebra object. The Hochschild
homology $HH(A)$ of $A$ is derived abelianization of $A$, i.e., the
pairing of $A$ with itself, considered as left and right module over
its enveloping algebra $A\ot A^{op}$:
$$HH(A)=A\ot_{A\ot A^{op}} A,$$ or explicitly as the geometric
realization of the cyclic bar construction of $A$. For $A$ commutative
this has the alternative description $HH(A)=S^1\ot A$ in terms of the
tensoring of commutative algebras over simplicial sets, in particular
revealing an $S^1$ action on Hochschild homology. This $S^1$ action
can be extended to arbitrary associative $A$ through the cyclic
symmetry of the cyclic bar construction, or geometrically thanks to
the theory of topological chiral homology (an aspect of topological
field theory), which realizes $HH(A)$ as an integration
$HH(A)=\int_{S^1}A$ of $A$ over $S^1$, see \cite[Theorem
  5.3.3.11]{HA}.

\begin{defn}  The {\em cyclic homology} of $A$ is the $S^1$-invariants
$HH(A)^{S^1}$.
\end{defn}

Now let us specialize the above to the context of this paper, namely where we work
over a fixed $\Q$-algebra $k$. Then the
results of this paper and \cite{BFN} combine to describe the cyclic
homology of $\QC(X)$ thought of as an algebra object in the $\oo$-category  $\St_k$
of presentable stable $k$-linear $\oo$-categories.

\begin{corollary} For $X$ a quasicompact quasiseparated 
derived scheme over $k$, the $k$-linear cyclic homology of $\QCoh(X)$ is canonically
equivalent to the $\oo$-category $\Omega_X^{-\bullet}[d]\module$ of de Rham modules on $X$.
\end{corollary}

Our results in Section \ref{Hodge} show that  $\Omega^{-\bul}_{X,d}$-modules, or
$\D_X$-modules, are an analogue of periodic
cyclic homology, and moreover, endow it with a Hodge filtration given by
the $\oo$-category of $B\Ga\rtimes \Gm$ equivariant sheaves on the loop space,
or Rees modules.

We can also view the cyclic homology of $\QC(X)$ as an invariant of
its $\oo$-category of module categories, which are ``derived
categorical sheaves'' over $X$. In particular, it plays a central role
as the natural recipient for the Chern character from the secondary
K-theory of $X$ in the theory developed by To\"en and Vezzosi, see
\cite{TV Chern short, TV Chern long}.


\subsection{Loops in derived stacks}\label{stack intro}
We now consider the generalizations of our results from schemes to
stacks. We will only consider geometric stacks $X$, by which we
connote an Artin derived stack with affine diagonal. (We refer the
reader to the survey~\cite{Toen} for an overview of Artin stacks and
\cite{HAG2} for detailed foundations, in particular the theory of
tangent complexes.)  Concretely, such stacks have an affine cover $U$
which is smooth and affine over $X$. Thus $X$ can be written as the
geometric realization $X\simeq|U_\bul|$ of the associated \v{C}ech cover,
the simplicial affine scheme with $k$-simplices given by $k+1$-fold fiber products $U\times_X \cdots \times_X
U$ and with smooth structure maps. (One could in fact use the
existence of such presentations as the definition of a geometric
stack.) This allows us to reduce questions on $X$ to questions on
affine schemes which are local in the smooth topology.

The fundamental difference between loops in schemes and stacks is a
consequence of the well known fact that loops in a topological space
$X$ are non-local objects: a loop in $X$ does not always come from a
loop in a cover of $X$, since a ``big" loop can not be expressed in
terms of ``small" loops.  Similarly, loops in a derived stack $X$ are
not local objects with respect to the \'etale topology (not to mention
smooth or flat topologies). In other words, one does not expect to be
able to recover the loop space $\cL X$ from the
loop space $\cL U$ of a cover $U\to X$, and the assignment of loops $U
\mapsto \cL U$, for \'etale maps $U\to X$,
does not in general form a cosheaf. 
This makes the theory of loops in
stacks much more intricate than that of schemes.

If we restrict to ``small" loops, then it is not unreasonable to think
that we can recover small loops in $X$ from small loops in a
cover. While this statement is not true on the level of points, we
will show that locality holds for small loops on the level of
functions and sheaves.

Let's illustrate the basic phenomena with the example $X=BG=pt/G$ of
the classifying stack of an algebraic group $G$.  In this case, it is
easy to see that the loop space is the adjoint quotient
$$\cL BG=G/G.$$ This is the stack of $G$-local systems on the circle
(for example, rank $n$ local systems for $G=GL_n$).  As an
illustration of the non-local behavior of loops, note that loops
$\cL(pt)=pt$ in the cover $pt\to BG$ can not see nontrivial conjugacy
classes in $\cL X\simeq G/G$.

Similarly, for a general geometric stack $X$, with underlying underived stack 
$\ul X$, the loop space $\cL X$
is a derived thickening of the {inertia stack} $I\ul{X}$. The
inertia stack is an affine group scheme over $\ul{X}$, parametrizing
points of $\ul{X}$ together with automorphisms.

Unlike for schemes, the loop space of a geometric stack is also quite
different from its odd tangent bundle. For example, in the case $X =
BG$, the latter is the adjoint quotient of the Lie algebra
$$\BT_{BG}[-1]=\g/G.$$

To interpolate between loops and odd tangents on geometric stacks, we 
introduce the space of
formal loops.

\begin{defn} The {\em formal loop space} $\Lhat X$ of a derived stack $X$
  is the formal completion $\wh{\cL X}_X$ of the loop space along the
  constant loops $X\to \cL X$.
\end{defn}

To continue the above example, the formal completion of $\cL BG$ along the constant
loops is the adjoint quotient of the formal group 
$$\widehat{\cL} BG\simeq \widehat{G}/G.
$$
Furthermore, the exponential map provides a canonical identification
$$\xymatrix{
\widehat{\g}/G \ar[r]^-\sim&
 \wh G/G } $$ 
with the formal completion of the shifted tangent bundle along its zero section.

We show in Theorem~\ref{formal} that this picture extends to arbitrary
geometric stacks.

\begin{thm}\label{exponential map}
There is a canonical equivalence
$$
\xymatrix{exp:\wh{\BT}_X[-1]\ar[r] & \Lhat X}$$ from the formal odd tangent bundle
to the formal loop space of $X$.
\end{thm}

There is also an intermediate notion of loop space suggested by
rational homotopy theory. The algebra of $k$-linear cochains
$C^*(M,k)$ on a topological space (which is the algebra of functions
$\cO(M)$ on the locally constant stack $M$, and hence on its
affinization $\Aff(M)$) recovers the $k$-homotopy type of $M$ for 
simply connected $M$. For general $M$, this algebra only sees the
pro-unipotent completion of the fundamental group.  Namely,
$C^*(M,k)=\cO(M)$ can be described as the self-Exts of the constant
sheaf (the trivial local system) on $M$. Thus its modules describe pro-unipotent 
local systems on $M$. With this picture in mind for
$M=S^1$, we make the following definition.

\begin{defn} The {\em unipotent loop space} $\cL^u X=\Map(B\Ga,X)$ is
  the space of loops which factor through the affinization
  $\Aff(S^1)\simeq B\Ga$ of the circle.
\end{defn}

The $S^1$-action on the unipotent loop space factors through a
 $B\Ga$-action. Moreover, we have a $\Gm$-action coming from rotation of
$B\Ga$, and the two combine to an action of $B\Ga \rtimes \Gm$.

In the example $X=BG$, the unipotent loop space classifies group
homomorphisms $\Ga\to G$, or in other words, algebraic one-parameter
subgroups, up to conjugation. (For $G=GL_n$ these are unipotent rank
$n$ local systems on the circle.) Thus we find
$$\cL^u BG\simeq G^u/G,$$ where $G^u$ is the formal neighborhood of
the unipotent cone of $G$ (so the reduced points of $G^u$ are
unipotent elements of $G$).  Since any Lie algebra element formally
exponentiates to a one-parameter subgroup, we have a canonical
inclusion $$\Lhat BG\subset \cL^u BG.$$ (In particular, note that the
unipotent loop space for $BG$ is not reduced!)  We find in
Theorem~\ref{formal} that this picture too extends to all geometric
stacks:

\begin{thm} 
Formal loops are unipotent: the canonical inclusion $\Lhat X\to \cL X$
factors through the canonical map $\cL^u X \to \cL X$ into a sequence
$$
\xymatrix{\Lhat X \ar[r] &  \cL^u X \ar[r] & \cL X.}
$$ 
In particular, $\Lhat X$ inherits a $B\Ga\rtimes\Gm$-action from $\cL^u
X$.
\end{thm}

\subsubsection{Descent for formal loops}

None of our three notions of loop space is truly local. Nevertheless,
in Section \ref{descent for forms} we find that the algebra of
functions on formal loops is a local object. Namely, we first show
that Hochschild chains (or equivalently, differential forms) on a
derived ring satisfy smooth descent. (The assertion of {\'etale}
descent for the Hochschild chain complex is a theorem of Weibel and
Geller \cite{HH etale}.)

\begin{prop} The assignment of Hochschild chains
  $$\xymatrix{
  R\ar@{|->}[r] & HC(R)=R\ot S^1 = R\ot_{R\ot R} R \simeq \Sym (\Omega_R[1])
  }$$ satisfies smooth descent: it forms a sheaf with respect to the
  smooth topology on derived rings.
\end{prop}

To extend this to a sheaf on geometric stacks, we need to observe that
for a derived ring $R$, Hochschild chains $HC(R)\simeq \Sym
(\Omega_R[1])$ are automatically complete. Equivalently, the loop
space $\cL \Spec R$ of an affine derived scheme is a formal thickening
of $\Spec R$.  As can be seen from the example $BG$, this is no longer
true for general stacks. Instead, the sheaf extending Hochschild
chains on affines is given by the completed symmetric algebra of
one-forms. We will refer to the formal completion of the odd tangent
bundle along the zero section as the formal odd tangent bundle
$$\wh{\BT}_X[-1]=\colim_{n\to \infty} \Spec_X \Sym(\Omega_X[1])/
\on{Sym}^{>n}(\Omega_X[1]).$$ 

We establish the following descent
results for completed algebras of forms.

\begin{thm}\label{state completed descent} Fix a geometric stack $X$.

  (1)
  The assignment of the complete graded
  algebra of differential forms
  $$\xymatrix{U\ar@{|->}[r] & \lim_{n\to \oo}\Sym(\Omega_U[1])/
\on{Sym}^{>n}(\Omega_U[1])
}  $$ 
  defines a sheaf on the smooth site of $X$.

  (2) 
  The assignment  of
 complete compatibly graded modules 
  $$
  \xymatrix{U\ar@{|->}[r] & \qc(\wh{\BT}_U[-1])^{\Gm}}
  $$
 forms a sheaf of   $\oo$-categories with $B\Ga$-action on the smooth site of $X$.
\end{thm}

Theorems \ref{exponential map} and  \ref{state completed descent} 
together establish that the formal loop space behaves like a local
construction from the point of view of functions and sheaves.
For a smooth cover $U\to X$, with corresponding \u Cech simplicial scheme $U_\bul$,
we have the local-to-global identities
$$\cO_{\Lhat X}\simeq\lim \cO(\Lhat U_\bul) \hskip .4in \qc(\Lhat
X)^{\Gm}\simeq\lim \qc(\Lhat U_\bul)^{\Gm}$$ 
compatible in the first case with the de Rham vector field, and in the second case
with the resulting $B\Ga$-action.

With the latter identity in hand, our results on loop
spaces of derived schemes, in particular Theorems~\ref{basic thm} and \ref{intro free thm} and their corollaries, extend to
formal loop spaces of smooth underived geometric stacks. For example, we have the identification
of $S^1$-equivariant sheaves on formal loop spaces
$$ \QCoh(\Lhat X)^{S^1} \simeq \Omega_X^{-\bullet}[d]\module$$ 
   $$ \QCoh(\Lhat X)^{B\Ga\rtimes \Gm} \simeq \Omega_X^{-\bullet}[d]\module_\Z$$
   and the identification of $\Omega S^2$-equivariant sheaves on formal loop spaces
$$ \QCoh(\Lhat X)^{\Omega S^2} \simeq 
  \Omega_X^{-\bullet}\langle d\rangle\topmodule
$$ 
$$ \QCoh(\Lhat X)^{\Aff(\Omega S^2)\rtimes\Gm} \simeq
  \Omega_X^{-\bullet}\langle d\rangle\topmodule_\Z
$$

Furthermore,
we have the Koszul dual formulation for $S^1$-equivariant sheaves
$$
 \qc(\Lhat X)^{B\Ga}
 \simeq
\wh\cR_{X}\topmodule
$$
$$
 \qc(\Lhat X)^{B\Ga \rtimes \G_m}
 \simeq
\wh\cR_{X}\topmodule_\Z
$$
 and similarly for $\Omega S^2$-equivariant sheaves
$$ \QCoh(\Lhat X)^{\Omega S^2} \simeq 
 \wh\cR^{sp}_X\topmodule
$$ 
$$ \QCoh(\Lhat X)^{\Aff(\Omega S^2)\rtimes\Gm} \simeq
 \wh \cR^{sp}_X\topmodule_\Z
$$

Finally, to state an analogue of Corollary~\ref{finite version} for
perfect modules, for simplicity let us restrict to a situation
tailored to many applications. Suppose that $\cL X$ turns out to be
underived in the sense that its structure sheaf is discrete.  Let us
write $\Perf_{X}(\Lhat X) $ for the small, stable $\oo$-category of
quasicoherent sheaves whose restriction along the canonical map $ X\to
\Lhat X $ are perfect.  (We caution the reader that this is not the
same as the previously introduced small, stable $\oo$-category
$\Perf_{X}(\cL X) $ of quasicoherent sheaves whose pushforward along
the canonical map $\cL X\to X$ are perfect.)  As usual, we write
$\cR_X\perf$ for the stable $\oo$-category of quasicoherent sheaves on
$X$ equipped with a compatible structure of perfect $ \cR_X$-module.
Under the assumption that $\cL X$ is underived, we have the further
identifications
 $$
 \Perf_{X}(\Lhat X)^{B\Ga \rtimes \G_m}
\simeq
\cR_{X}\perf_\Z
$$
$$
 \Perf_{X}(\Lhat X)^{B\Ga}
\simeq
\cR_{X}\perf
$$

\subsubsection{Sheaves on unipotent loops}

Finally, there is a useful version of the above story for unipotent loops (in place of formal loops).
Let us continue with the assumption that $\cL X$ is underived in the sense that its structure sheaf is discrete.

Let  $\Perf_{X}(\cL^u X) $ denote the subquotient $\oo$-category of $\qc(\cL^u X) $ where we kill
 the kernel of the restriction along the canonical map 
$
X\to \cL^u X$,
and require that the restriction produces perfect sheaves on $X$.
Then there is a canonical equivalence
$$
 \Perf_{X}(\Lhat X) \simeq  \Perf_{X}(\cL^u X)
$$
Note that objects on the two sides admit the same characterization since there is no kernel to restriction
along the canonical map
$X\to \Lhat X$.

Putting this together with our previous results provides the following.

\begin{thm}
For a smooth geometric stack $X$ such that $\cL X$ is underived, there are canonical equivalences
$$
 \Perf_{X}(\cL^u X)^{B\Ga \rtimes \G_m}
\simeq
\cR_{X}\perf_\Z
$$
$$
 \Perf_{X}(\cL^u X)^{B\Ga}
\simeq
\cR_{X}\perf
$$
\end{thm}

As a consequence, our results on sheaves on formal loop spaces
generalize to sheaves on unipotent loop spaces once we insist on
$\Gm$-equivariance and bounded weights. This result will be applied in
\cite{reps} to situations of interest in representation theory (such
as equivariant Steinberg varieties).

\medskip
\noindent{\bf Acknowledgments} We have benefited tremendously from
many conversations on the subjects of loop spaces, cyclic homology and
homotopy theory. We would like to thank Kevin Costello, Ian
Grojnowski, David Kazhdan, Mike Mandell, Tom Nevins, Tony Pantev,
Brooke Shipley, Dima Tamarkin, Constantin Teleman, Boris Tsygan, Amnon
Yekutieli, and Eric Zaslow for their many useful comments. We would
like to especially thank Andrew Blumberg for his explanations of
homotopical algebra, and Jacob Lurie and Bertrand To\"en for
introducing us to the beautiful world of derived algebraic geometry
and patiently answering our many questions. Finally, we are very
grateful to a conscientious referee who provided detailed comments on
a first submitted version.

The first author was partially supported by NSF grant DMS-0449830,
and the second author by NSF grant DMS-0600909. Both authors would
also like to acknowledge the support of DARPA grant
HR0011-04-1-0031.



\section{Preliminaries}\label{prelim}

\subsection{$\oo$-categories}

Roughly speaking, an $\infty$-category (or synonymously $(\oo,
1)$-category) encodes the notion of a category whose morphisms form
topological spaces and whose compositions and associativity properties
are defined up to coherent homotopies.   They offer a context for homotopy theory
which lies between the coarse world of homotopy
categories and the fine world of model categories. Ordinary categories
fit into this formalism as discrete $\oo$-categories.

The theory of
$\infty$-categories has many alternative formulations (as topological
categories, Segal categories, quasicategories, etc; see \cite{Bergner}
for a comparison between the different versions). We will follow the
conventions of \cite{topoi}, which is based on Joyal's
quasi-categories \cite{Joyal}. Namely, an $\infty$-category is a
simplicial set satisfying a weak version of the Kan condition
guaranteeing the fillability of inner horns. The underlying vertices
play the role of the set of objects while the fillable inner horns
correspond to sequences of composable morphisms. The book \cite{topoi}
presents a detailed study of $\infty$-categories, developing analogues
of many of the common notions of category theory. An overview of the
$\infty$-categorical language, including limits and colimits, appears
in \cite[Chapter 1.2]{topoi}.
$\infty$-categories and the more traditional settings of model
categories or homotopy categories is that coherent homotopies are
naturally built into the definitions. Thus for example, all functors
are derived and the natural notions of limits and colimits correspond
to {homotopy} limits and colimits.

To pursue homotopical algebra within this formalism, one needs a
theory of monoidal and symmetric monoidal $\oo$-categories.  Such a
theory is developed in \cite[2,4]{HA} including appropriate notions of
algebra objects, module categories over algebra objects, and so on.
The setting of stable $\infty$-categories, as developed in
\cite[1]{HA}, provides an analogue of the additive setting of
homological algebra.  A stable $\infty$-category can be defined as an
$\infty$-category with a zero-object, closed under finite limits and
colimits, and in which pushouts and pullbacks coincide
\cite[1.1.3]{HA}. A key result of \cite[1.1.2]{HA} is that stable
$\oo$-categories are enhanced versions of triangulated categories, in
the sense that their homotopy categories are canonically triangulated.
A representative example is the $\infty$-categorical enhancement of
the derived category of modules over a ring.

Given a triangulated category which is linear over a commutative
$\Q$-algebra $k$, we may consider enhancing its structure in three
equivalent ways, promoting it to (1) a differential graded (dg)
category, (2) an $A_{\infty}$-category, or (3) an $\infty$-category.
(Among the many excellent references for dg and $A_\infty$-categories,
we recommend the survey \cite{Keller}.)  In particular, $k$-linear
stable $\infty$-categories are equivalent to $k$-linear
pre-triangulated dg categories (that is, those whose homotopy category
is triangulated) and the reader may substitute the latter for the
former throughout the present paper. (See~\cite[Theorem 3.19]{bgt} for
a general equivalence between stable linear categories and spectral
categories which reduces to the above assertion for the Eilenberg-Mac
Lane spectrum of a $\Q$-algebra $k$.)

The collection $\St_k$ of stable presentable $k$-linear
$\oo$-categories itself forms a symmetric monoidal $\oo$-category.  To
see this, recall that the $\oo$-category of stable presentable
$\oo$-categories is symmetric monoidal~\cite[6.3.2]{HA}. It contains
the $\cE_\oo$-algebra object $\Mod_k$ of modules over the commutative
ring $k$. By definition, the $\oo$-category $\St_k$ is the module
category $\Mod_{\Mod_k}(\St)$ of $\Mod_k$-modules inside of
$\St$. Thus $\St_k$ inherits a symmetric monoidal structure given by
tensoring over the $\cE_\oo$-algebra object $\Mod_k$.

Note as well that the forgetful functor $\St_k \to \St$ is conservative: a $k$-linear functor $F:\cC \to \cD$ is an equivalence
if and only if it is an equivalence as a functor of stable $\oo$-categories. Furthermore, in the $\oo$-setting, stability is a property not a structure, so it suffices in turn to ask if $F$ is an equivalence on plain $\oo$-categories.



\subsection{Derived stacks}

Derived algebraic geometry provides a useful language to discuss various
objects and constructions of representation theory and topological field theory. 
The basic objects are derived stacks which
are built out of affine derived schemes in the same way that higher stacks
are built out of ordinary affine schemes. 
In this paper, we are primarily interested
in stable $\oo$-categories of quasicoherent sheaves on derived stacks. We will discuss below what we mean by 
 affine
derived scheme, and then briefly review the notion of derived stack and quasicoherent sheaf.

For more details on derived algebraic geometry,
we refer the reader to To\"en's extremely useful survey~\cite{Toen},
and to the papers of To\"en-Vezzosi \cite{HAG1,HAG2} and Lurie
\cite{topoi, HA, dag7}. In particular, we were introduced to
derived loop spaces by \cite{Toen}. 

\subsubsection{Derived rings}
Derived algebraic geometry is a form of algebraic geometry in which
the usual notion of discrete commutative ring is replaced by a
homotopical generalization.  There is not a single choice of
generalization, but a variety of notions suited to different
applications.  All the possible versions of commutative rings we might
consider can be interpreted as commutative ring objects (satisfying
some further conditions) in some symmetric monoidal $\oo$-category
$\cC$.

For our intended applications in representation theory and topological
field theory, the most relevant choice is to fix a discrete
commutative algebra $k$ over the rational numbers, and take $\cC$ to
be the stable symmetric monoidal $\oo$-category $\Mod_k$ of cochain complexes of $k$-modules
(localized with respect to quasi-isomorphisms). 
Thanks to~\cite[Theorem 4.4.4.7]{HA}, one can regard a commutative ring
object in  $\Mod_k$ as nothing more than a commutative dg
$k$-algebra.  We will further insist that all such rings be {connective} in
the sense that they are cohomologically trivial in positive degrees.

\begin{defn}
By a {\em derived $k$-algebra}, we mean a connective commutative
dg $k$-algebra, where $k$ is a rational commutative algebra.
We write $\DGA^-$ for the $\oo$-category of derived  $k$-algebras.
\end{defn}

We will also write $DGA_k$ for the corresponding $\oo$-category of
arbitrary (not necessarily connective) commutative dg $k$-algebras.

\begin{remark}

If $k$ is a discrete commutative algebra over the rational numbers, 
then the stable symmetric monoidal $\oo$-categories of 
 cochain complexes of $k$-modules, 
simplicial $k$-modules, and $k$-modules in spectra are all equivalent. 
This devolves from an underlying Quillen equivalence of model categories.
Thus there is a single common notion of derived $k$-algebra.
\end{remark}

\subsubsection{Derived stacks}
Fix once and for all a discrete commutative rational algebra $k$.

We begin with a reminder of the functor of points for higher stacks.
Let $\cS$ denote the
$\infty$-category of simplicial sets, or equivalently (compactly
generated Hausdorff) topological spaces.  Let $Alg_k$ denote the
category of discrete commutative $k$-algebras. By definition, the
category of affine schemes over $k$ is the opposite category
$Alg_k^{op}$.  A {\em stack} over $k$ is a functor
$$X:Alg_k\to \cS
$$ which is a sheaf in the \'etale
topology. This means that $X$ preserves those colimits 
corresponding to \'etale covers of affine schemes. We write $Stack_k$
for the $\oo$-category
of stacks over $k$.

Familiar examples
of stacks include ordinary schemes and Deligne-Mumford or
Artin stacks in groupoids. At the other extreme, topological spaces
$K\in \cS$ naturally define stacks, as the sheafification of
constant functors $K:Alg_k\to \cS$. One can choose a
  simplicial presentation of $K$ (for example, that given by singular
  chains) and consider $K$ as a functor to simplicial sets. Of course,
  any two simplicial presentations lead to equivalent stacks.  The
inclusion of schemes into stacks valued in discrete sets preserves
limits, such as fiber products, but alters colimits, such as group
quotients.

\begin{defn}
The
$\infty$-category of {\em affine derived schemes} over $k$ is the
opposite $\oo$-category $(\DGA_k^-)^{op}$ to derived $k$-algebras. 
\end{defn}

The Yoneda embedding gives a
fully faithful embedding of $(\DGA_k^-)^{op}$ into the $\oo$-category
of functors $\DGA_k^-\to \cS$. We will be interested in gluing
together affine derived schemes into more general functors. 

\begin{defn}A {\em
  derived stack} over $k$ is a functor
$$X:\DGA_k^-\to \cS
$$ which is a sheaf in the \'etale topology, so in other words, $X$
  preserves those colimits corresponding to \'etale covers of affine
  derived schemes.  We write $DSt_k$ for the
  $\oo$-category of derived stacks over $k$.
\end{defn}

Note that $Alg_k$ sits as a full subcategory of $\DGA_k^-$, and its
inclusion preserves limits (though not colimits).  Equivalently,
ordinary affine schemes sit as a full subcategory of affine derived
schemes, and their inclusion preserves colimits (though not
limits). This enables us to relate derived and underived stacks by
adjoint functors:
$$ \xymatrix{ i:Stack_k \ar@<-0.7ex>[r] &\ar@<-0.7ex>[l] DSt_k:r }
$$ 
where $i$ is fully faithful.
The functor $r$ takes a derived stack and restricts its domain from
$\DGA^-_k$ to $Alg_k$.  It admits a left adjoint $i$ which extends the
embedding of affine schemes into derived stacks by continuity (that
is, to preserve colimits). In other words, given a stack $X$, we write
$X$ as a colimit of ordinary affine schemes $\Spec R_\alpha$, and then
$i(X)$ is the colimit of the same affine schemes $\Spec R_\alpha$ but
now considered as derived stacks.  Thus considering stacks as derived
stacks preserves the notion of colimit but typically alters the notion
of limit. For instance, the fiber product of affine schemes $$\Spec
R_1\times_{\Spec R_0}\Spec R_2 \in DSt_k
$$
considered as derived stacks is given by the spectrum of
the {derived} tensor (or Tor) algebra 
$$R_1\otimes_{R_0} R_2\in\DGA_k^-.
$$

Another important difference between stacks and derived stacks is
apparent in the calculation of mapping spaces.  The $\infty$-category
of derived stacks is naturally a closed symmetric monoidal
$\infty$-category, with monoidal structure given by the point-wise
Cartesian product of functors. The inner Homs (closed monoidal
structure) are given by the sheaf Homs, which are the derived mapping
stacks $\Map(X,Y).$ Note that colimits (in particular gluings) in the
first factor are converted to limits (in particular fiber products) in
the mapping space. Thus when $X=K$ is a simplicial set, built
concretely out of discrete sets with identifications, we can think of
$\Map(K,Y)$ as a collection of equations imposed on products of copies
of $Y$. Even for $Y$ affine this construction may lead to interesting
derived structure.

Finally, we will often restrict to the following reasonable class of derived stacks.

\begin{defn}\label{geometric stack} 
  We use the term {\em geometric stack} to refer to an Artin derived stack with affine
  diagonal.
\end{defn}

We refer the reader to the survey~\cite{Toen} for an overview of Artin stacks.
The definition  of Artin derived stack is inductive. A $-1$-Artin derived stack is an affine derived scheme;
a morphism  of derived stacks is $-1$-Artin if its fibers over affines
are $-1$-Artin derived stacks.
In general, 
a derived stack $X$ is $n$-Artin if there exists a disjoint union of affine derived schemes $U$ and a smooth $(n-1)$-representable and surjective morphism $U \to X$;
a morphism of  derived stacks is called $n$-representable if  its fibers over affines
are $n$-Artin derived stacks. An Artin derived stack is an $n$-Artin derived stack for some $n$.


\subsubsection{Quasicoherent sheaves}\label{dgcats}

To any derived stack $Z\in\DSt_k$, there is assigned a stable
presentable $k$-linear $\oo$-category $\QC(Z)$ called the
$\oo$-category of quasicoherent sheaves on $Z$ (see
\cite[p.~36]{Toen}, and \cite{dag8} for a detailed study). Let us
briefly recall its construction.

First, for a derived $k$-algebra $A$, we define $\QC(\Spec A)$ to be the
stable $\oo$-category of $A$-modules $\Mod_A$.  Equivalently, we can take
$\Mod_A$ to be the dg category of dg $A$-modules
localized with respect to quasi-isomorphisms.  Given a map $p:\Spec
A\to \Spec B$, we have the usual pullback functor
$$
p^*:\QC(\Spec B)\to \QC(\Spec A) \qquad p^*(M) = M\ot_B A
$$ 

Now in general, for tautological reasons, any derived stack $Z$ is the
colimit of the diagram $Z_\bullet$ of all affine derived schemes over
$Z$.  We define $\QC(Z)$ to be the limit (in the $\oo$-category
$\St_k$ of stable presentable $k$-linear $\oo$-categories) of the
corresponding diagram $\QC(Z_\bullet)$ with respect to pullback.  Thus
we can interpret the functor
$$
\qc:\DSt_k \to  \St_k^{op}
$$ 
as the unique
colimit-preserving extension of the functor
$$
\Mod:(DGA_k^-)^{op} \to  \St_k^{op}
$$
that assigns modules to the spectrum of a derived $k$-algebra. In other words, we can view
 quasicoherent sheaves as the continuous linearization of derived stacks.


%

\subsection{Graded modules}\label{graded section}
This section contains some technical material needed later in the paper. The reader
should probably skip it until needed. We record here some useful structural properties of modules
over $\Z$-graded $k$-algebras.

Given a (not necessarily commutative) dg $k$-algebra $R$, we write
$R\module$ for the stable $\oo$-category of dg $R$-modules.  Given a
$\Z$-graded dg $k$-algebra $R$, we write $R\module_\Z$ for the stable
$\oo$-category of compatibly $\Z$-graded dg $R$-modules. By descent,
working with $\Z$-graded algebras and modules is equivalent to working
over the classifying stack $B\Gm$.

We will always use the term {\em degree} to refer to cohomological degree,
and the term {\em weight} to refer to the $\Z$-grading on a
 $\Z$-graded dg algebra
and its $\Z$-graded dg modules. For a $\Z$-graded dg module $M$,
we will denote by $M[n]$ the customary degree shift by $-n$, and by $M\la m\ra$ the weight shift by $m$. So for example, if $M$ is concentrated in degree $0$ and weight $0$, then $M[n]\la m\ra$
is concentrated in degree $-n$ and weight $m$.

Suppose $R$ is a $\Z$-graded dg $k$-algebra concentrated in nonnegative weights.
We will write $\wh R$ for the completion of $R$ with respect
to the positive weight direction. 
To be precise,
for $k > 0$, let $R_{>k} \subset R$ denote the ideal of positive weight spaces  greater than $k$.
We have a natural diagram of dg algebras
$$
\xymatrix{
R/R_{>0} & \ar[l] R/R_{>1}& \ar[l] R/R_{>2} & \ar[l] \cdots
}
$$
We take $\wh R$ to be the limit dg algebra of this diagram.
We will write $\wh R\topmodule_\Z$ for the stable
$\oo$-category of  compatibly $\Z$-graded complete dg $\wh R$-modules. To be precise,
the above diagram of dg algebras provides a diagram of stable $\oo$-categories of graded modules
$$
\xymatrix{
R/R_{>0}\module_\Z \ar[r] &  R/R_{>1}\module_\Z \ar[r] & R/R_{>2}\module_\Z  \ar[r] & \cdots
}
$$
We take $\wh R\topmodule_\Z$ to be the limit stable $\oo$-category of this diagram.

Finally, given a $\Z$-graded stable $\oo$-category $\cC$, we will write $\cC_+ \subset \cC$,
respectively $\cC_- \subset \cC$,
for the full $\oo$-subcategory  of objects with weights bounded below, respectively bounded above.

\subsubsection{Shear shift}

It is useful to recognize that $R\module_\Z$
is independent of certain shifts. 
For $i\in \Z$, let $R^i\subset R$ denote the summand of weight $i$.
For $n\in \Z$, we will write $R_{[n]}$ for the  dg
$k$-algebra defined by the shear shifting
$$
R_{[n]} = \oplus_{i\in \Z} R^i[ni].
$$
Of course when $n=0$, we recover our original algebra $R_{[0]} = R$. It is worth commenting
that in general if $R$ is commutative, then only the even shifted versions $R_{[2n]}$ are commutative as well. For example, if $R= k[t]$ with $t$ of degree $0$ and weight $1$, then $R_{[2n]}$ is free commutative,
but $R_{[2n +1]}$ is not commutative.

Now we have the following useful independence of shifts.

\begin{prop}\label{shift indep}
For $n\in \Z$, there is a canonical weight-preserving equivalence of stable $\oo$-categories
$$
R\module_\Z \simeq R_{[n]}\module_\Z.
$$

\end{prop}

\begin{proof}
We will send the free module $R$ to the free module $R_{ [n]}$, and then extend
preserving colimits and shifts by weight and cohomological degree. To have this make sense,
we need only check that the endomorphisms of  $R$ and $ R_{[n]}$
agree. Considered as plain algebras,   $R$ and $ R_{[n]}$ of course have endomorphisms
  $R$ and $ R_{ [n]}$
themselves respectively.
Those endomorphism that survive in the $\Z$-graded categories are the invariants
of weight $0$. In both cases, the invariants are simply the weight $0$ summand
$R^{0}$.
\end{proof}


\subsubsection{Periodic localization} Consider the $\Z$-graded polynomial algebra $k[t]$ with 
$t$ of degree $0$ and weight $1$. In geometric language, we have $\aline = \Spec k[t]$ equipped
with the standard $\Gm$-action. Consider as well the $\Z$-graded commutative $k[t]$-algebra $k[t, t^{-1}]$.
 In geometric language, we have $\Gm= \Spec k[t, t^{-1}]$ equipped
with the standard $\Gm$-action. The algebra map  $k[t]\to k[t, t^{-1}]$ corresponds
to the standard $\Gm$-equivariant open embedding $\Gm \to \aline$.

Now let $R$ be an arbitrary algebra object in $k[t]\module_\Z$,
so that $R\module_\Z$ consists of $R$-module objects in $k[t]\module_\Z$. 
In geometric language, we have that $R$ is an
algebra object in $\qc(\aline/\Gm)$, and $R\module_\Z$ consists of $R$-module objects in $\qc(\aline/\Gm)$.

\begin{defn}
The {\em periodic localization} $R\module_{per}$ is the stable $\oo$-category
defined by the following equivalent constructions. 

(1) Geometric formulation: let $pt = \Gm/\Gm \subset \aline/\Gm$ be the standard open embedding.
We define $R\module_{per}$ to be the base change
$$
R\module_{per}  =  R\module_\Z \ot_{\qc(\aline/\Gm)} \qc(pt).
$$

(2) Algebraic formulation: in the language of $\Z$-graded rings, the above formulation (1) translates to the base change
$$
R\module_{per}  = R\module_\Z \ot_{k[t]\module_\Z} k[t, t^{-1}]\module_\Z.
$$

(3) Formulation in terms of generator: standard adjunctions show that the above formulations (1) and (2) are equivalent to the
stable $\oo$-category of modules for the localized algebra
$$
R\module_{per}  = R \ot_{k[t]} k[t, t^{-1}]\module_\Z.
$$
\end{defn}

In the definition, the tensor products in formulations (1) and (2) are calculated in the $\oo$-category
of stable, presentable $\oo$-categories. Note that the periodic localization is no longer $\Z$-graded
since it results from the base change to a point.

\medskip

For $R$ an arbitrary algebra object in  $k[t]\module_\Z$, it will also be useful to speak of the periodic localization for the shear shifted algebra $R_{[2n]}$. Note that $R_{[2n]}$ is 
naturally an algebra object in $k[t]_{[2n]}\module_\Z$, and that $k[t]_{[2n]} = k[v]$
where $v$ has degree $-2n$ and weight $1$. We can then form the periodic localization
$$
R_{[n]}\module_{per}  = R\module_{\Z}  \ot_{k[v]\module_\Z} k[v, v^{-1}]\module_\Z.
$$
and the equivalence of Proposition~\ref{shift indep}
induces an equivalence
$$
R_{[n]}\module_{per}  \simeq R\module_{per}.
$$

\subsubsection{Sheafified version}

It is worthwhile to explain a version of the preceding discussion in a sheafified context.

Fix a derived stack $X$ with stable $\oo$-category $\qc(X)$
of quasicoherent sheaves. Suppose $R_X$ is a $\Z$-graded algebra object in
 the monoidal  $\oo$-category of endofunctors of $\qc(X)$. 
Then we have the following sheafified versions 
of the preceding constructions. 

We write $R_X\module_\Z$ for the stable
$\oo$-category of $\Z$-graded $R_X$-modules in $\qc(X)$. 
When $R_X$ is concentrated in nonnegative weights,
we  write $\wh R_X$ for the completion of $R_X$ with respect
to the positive weight direction. 
We write $\wh R_X\topmodule_\Z$ for the stable
$\oo$-category of  compatibly $\Z$-graded complete $\wh R_X$-modules.

For $n\in \Z$, we write $R_{X, [n]}$ for the algebra defined by the shear shifting
$$
R_{X, [n]} = \oplus_{i\in \Z} R_X^i[ni].
$$
There is a canonical weight-preserving equivalence of stable $\oo$-categories
$$
R_X\module_\Z \simeq R_{X, [n]}\module_\Z.
$$

Finally, 
consider once again the $\Z$-graded polynomial algebra $k[t]$ with 
$t$ of degree $0$ and weight $1$ so that we have the affine line $\aline = \Spec k[t]$
with the standard $\Gm$-action.
Via the correspondence $X\times (\aline/\Gm) \leftarrow X\times \aline \to X$, we can view the 
stable $\oo$-category $\qc(X \times (\aline/\Gm))$ as $\Z$-graded modules in $\qc(X)$ over the
constant $\Z$-graded
algebra object $\cO_{X\times\aline} \simeq \cO_X \otimes_k k[t]$. 
 In turn, we can view $\cO_X \otimes_k k[t]$ as a $\Z$-graded algebra object in the monoidal
 $\oo$-category of endofunctors of $\qc(X)$.

Now suppose $R_X$ is a $\Z$-graded $\cO_X \otimes_k k[t]$-algebra object in
 the monoidal  $\oo$-category of endofunctors of $\qc(X \times \aline)$.  
 Then we have the periodic localization 
of $R_X\module_\Z$
defined by the base change
$$
R_X\module_{per}  =  R_X\module_\Z \ot_{k[t]\module_\Z} k[t, t^{-1}]\module_\Z.
$$
It also admits the two alternative formulations discussed above.
Similarly, we have the periodic localization for the shear shifted algebra
$$
R_{X, [n]}\module_{per}  = R_X\module_{\Z}  \ot_{k[v]\module_\Z} k[v, v^{-1}]\module_\Z
$$
where $v$ has degree $-2n$ and weight $1$,
and it satisfies the compatibility
$$
R_{X, [n]}\module_{per}  \simeq R_X\module_{per}.
$$


\section{Affinization}\label{affinization}

In this section, we study the relationship between a derived stack $X$
and its commutative ring of derived global functions $\cO(X)$.  We
continue with $k$ a fixed discrete commutative rational algebra.

\subsection{Background}
Recall that $Alg_k$ denotes the category of discrete commutative $k$-algebras,
so its opposite category $Alg_k^{op}$ is the category of affine schemes over $k$.
Let $Scheme_k$ denote the category of all schemes over $k$.
Recall the standard adjoint pair of functors 
$$
\xymatrix{ \cO:Scheme_k\ar@<-0.7ex>[r] &\ar@<-0.7ex>[l] Alg_k^{op} :\Spec
}
$$ 
where  $\cO$ is the global functions functor, 
and $\Spec$ is the inclusion of affine schemes into all schemes.
The functor of affinization is the monad
$$\Aff = \Spec \cO: Scheme_k \to Scheme_k$$
that sends a scheme $X$ to the spectrum
of its ring of global functions $\cO(X)$. 
In particular, for any scheme $X$, there
is an affine scheme $\Aff(X) $ and a universal map $X\to \Aff(X)$ through which
we can factor any map from $X$ to an affine scheme.

Now let   $\DGA^+_k$ denote the $\oo$-category of coconnective commutative
dg $k$-algebras.
topological space $K$, we can assign its $k$-valued cochains
$C^*(K,k)$ as an object of $\DGA_k^+$.  More generally, to
any derived stack $X$, we can assign its global functions
  $\cO(X)$, also known as the coherent cohomology of the structure
  sheaf, as a (not necessarily connective or coconnective) commutative
  algebra
$$X=\colim \Spec R_\alpha,\hskip.1in R_ \alpha\in \DGA_k^-, \hskip.3in
  \cO(X)=\lim R_ \alpha\in\DGA_k.$$ (More precisely, we are glossing
  over set theoretic issues, which can be dealt with by requiring
  standard smallness conditions on our stacks, see e.g. \cite{HAG1}.)
  If $X$ is an ordinary stack (for example, an ordinary scheme), then
  $\cO(X)\in \DGA_k^+$, while if $X$ is an affine derived scheme, then
  $\cO(X)\in \DGA_k^-$.

In \cite{Toen affine}, To\"en introduces {\em affine stacks} as the full
subcategory of underived stacks represented by the opposite category
$(\DGA_k^+)^{op}$ of coconnective commutative
$k$-algebras.\footnote{A word of caution: the embedding of
  $(\DGA_k^+)^{op}$ into stacks does not preserve colimits, so it is
  not recommended to think of affine stacks as simplicial affine
  schemes as one often thinks of stacks.}  In particular, the
affine stack corresponding to $C^*(K,k)$ is the algebraic avatar of
the $k$-linear homotopy theory of $K$.


\subsection{Affine derived stacks}

We will be interested in the analogue of affinization for derived
stacks, and thus we would like to consider the larger class of
``affine derived stacks'' corresponding to all (not neccessarily
coconnective) commutative dg $k$-algebras.  To do so, we introduce the
functor $$\Spec:\DGA_k^{op}\to \DSt,\hskip.2in
\Spec(R)(A)=\Hom_{\DGA_k}(R,A), \hskip.1in A\in \DGA_k^-.$$ In other
words, $\Spec$ assigns to $R \in \DGA_k^{op}$ the functor
it represents when restricted to connective commutative dg
$k$-algebras (which is automatically a stack, since it distributes
over limits in $\DGA_k^-$). 

\begin{prop} There is a natural adjunction
$$
\xymatrix{
\cO:DSt_k\ar@<-0.7ex>[r] &\ar@<-0.7ex>[l] \DGA_k^{op} :\Spec
}
$$  
\end{prop}

\begin{proof}
Note first that $\cO$ is colimit preserving, and so by the adjoint
functor theorem possesses a right adjoint. We will identify this
adjoint with $\Spec$ by testing maps into $\Spec$. Let $R\in\DGA_k$,
and let $A\in \DGA_k^-$.  By definition of the functor $\Spec R$, maps
of derived stacks $\Spec A \to \Spec R$ are given by maps of dg
$k$-algebras $R\to A$ .  Using this, we establish the desired
adjunction as follows. Consider any $X\in DSt$, and write $X= \colim
\Spec A_i$. Then we have equivalences
\begin{eqnarray*}
\Hom_{\DSt}(X,\Spec R) &\simeq & \Hom_{\DSt}(\colim \Spec A_{i}, \Spec R)\\
&\simeq & \lim \Hom_{\DSt}(\Spec A_{i}, \Spec R)\\
&\simeq& \lim \Hom_{\DGA_k}(R, A_{i})\\
&\simeq& \Hom_{\DGA_k}(R, \lim A_{i})\\
&\simeq& \Hom_{\DGA_k^{op}}(\cO(X),R).
\end{eqnarray*}

\end{proof}

\begin{defn} We call the endofunctor 
$$
\Aff:\DSt_k \to \DSt_k \qquad
\Aff(X) = \Spec\cO(X)
$$ 
  the {\em affinization functor}. We refer to derived
  stacks $X$ for which the canonical morphism $X\to \Aff(X)$ is an equivalence as {\em affine
    derived stacks}.\end{defn}

Note that $\Aff$ has a canonical monad structure, and any morphism
from $X$ to an affine derived stack factors through the canonical morphism $X\to
\Aff(X)$.
  
We will only use the above constructions in the following specific
setting. We will consider morphisms $Y \to \Spec A$, where $Y$ is a
topological space, and $A\in \DGA_k^-$ so that $\Spec A$ is an affine
derived scheme. Then by adjunction, the original morphism canonically
factors into morphisms
  $$
  \xymatrix{
  Y \ar[r] & \Aff(Y) = \Spec\cO(Y) \ar[r] & \Spec (\cO (\Spec A)) \simeq \Spec A.
}  $$
  Note that we need the generality of affine derived stacks introduced above since we need the adjunction simultaneously for
  $\cO(Y)\in \DGA^+_k$ as well as $A\in \DGA_k^-$.
  
  \begin{remark}
  We would like to thank To\"en for pointing out that the derived version of $\Spec$ is not faithful. For example, $\Spec$ of the formal commutative dg alegbra $k[u, u^{-1}]$, with $|u|=2$, defines the empty stack:
  there are no morphisms from $k[u, u^{-1}]$ to a connective commutative dg algebra since $u$ is invertible
  but would need to vanish under such a morphism.
  \end{remark}
%






\subsection{Affinization of groups} 
We would like to understand the impact of affinization on group
derived stacks and their actions on other derived stacks. We will
develop a minimum of theory to apply to the group stack $S^1$.  All we
will use is that $S^1$ is a finite CW complex so that we can appeal to
the following immediate lemma.

\begin{lemma}\label{finite CW}
Let $\pi:X\to \Spec k$ be a finite CW complex regarded as a derived
stack. The global sections functor $\pi_*:\QC(X)\to \Mod_k$ preserves
colimits (the structure sheaf $\cO_X$ is a compact object), and the
pullback functor $\pi^*:\Mod_k \to \qc(X)$ preserves limits (global
functions $\cO(X)$ form a dualizable $k$-module).  The same assertions
hold for the affinization $\pi':\Aff(X) \to \Spec k$.
\end{lemma}

\begin{proof}
For $X$, both assertions follow from the fact that $X$ is a finite
colimit of copies of $\Spec k$.  For $\Aff(X)$, consider the natural
morphism $\alpha:X\to \Aff(X)$, and observe that $\pi'_* \simeq
\pi_*\alpha^*$ and $\pi'^{*} \simeq \alpha_* \pi^*$.
\end{proof}

\begin{defn} We define the full $\oo$-subcategory  $\DSt^{fc}_k\subset \DSt_k$
  of {\em functionally compact} derived stacks to consist of
    derived stacks $X$ whose structure sheaf $\cO_X$ is a
  compact object of $\QC(X)$, or in other words, for which the
  functor of global sections $\Gamma:\QC(X)\to \Mod_k$ preserves colimits.
\end{defn}

Thus the above lemma shows that finite CW complexes are functionally
compact.  The following technical assertion motivates the first
conclusion of the lemma (considerations of descent will motivate
the second conclusion).

\begin{prop}\cite[Proposition 3.10, see also Remark
  3.11]{BFN} \label{package} Suppose $X,Y \to \Spec k$ are
  functionally compact stacks.  Then
  morphisms $f:X\to Y$ satisfy the projection formula:
  $$
  f_*(f^*M \ot_{\cO_X} N) \simeq M \ot_{\cO_Y} f_*N,
  \quad
  M\in \qc(Y), N\in \qc(X).
  $$
Furthermore, if $g:\tilde Y\to Y$ is any map
of derived stacks, the resulting base change map $g^* f_* \to \tilde f_* \tilde g^*$ is an equivalence.
\end{prop}

\begin{prop}
  The restriction of the affinization functor $\Aff$ to functionally
  compact stacks is naturally monoidal.
  \end{prop}

\begin{proof}
 For $X, Y$ functionally compact, the projection formula provides a
 canonical equivalence $$\cO(X\times Y)\simeq \cO(X)\otimes \cO(Y).$$
 Thus affinization commutes with finite products of functionally
 compact stacks, and hence by \cite[Corollary 2.4.1.9]{HA}, it
 naturally defines a monoidal functor.
\end{proof}

\begin{corollary}
For a group $G$ in finite CW complexes, its affinization $\Aff(G)$
inherits a canonical structure of group derived stack, and the
affinization map $G\to \Aff(G)$ inherits the structure of a group
homomorphism.
\end{corollary}

\begin{proof}
Group objects are defined by diagrams involving solely the monoidal structure. 
\end{proof}


In what remains of this section, we consider the role of affinization in understanding $G$-actions
on affine derived schemes. We begin with the linear case of $G$-actions on plain modules.

\subsubsection{$G$-actions on modules}

Given a group derived stack $G$, we have the natural derived stack $BG$ given by
the usual classifying space construction. As usual, it comes equipped with a universal $G$-bundle
$\pi: \Spec k = EG\to BG$.

\begin{defn} We define the $\oo$-category of $k$-modules with $G$-action to be the $\oo$-category
of quasicoherent sheaves $\qc(BG)$. We say that a $k$-module $M$ is equipped with a $G$-action
if we are given $\cM \in\qc(BG)$ with an identification
  $\pi^*\cM\simeq M$.\end{defn}

\begin{lemma}\label{G modules on BG} Let $G$ be a group in finite CW complexes.
There is a canonical $k$-linear equivalence 
$$
\QC(BG)\simeq
  \cO(G)\on{-comod}
  $$ identifying the structures of $G$-action
  and $\cO(G)$-comodule on $k$-modules.
  The same assertion holds for the affinization $\Aff(G)$.
\end{lemma}

\begin{proof}
The assertion is a standard descent argument. Namely, consider the
usual $k$-linear adjunction
  $$
\xymatrix{
\pi^*:\QC(BG)\ar@<-0.7ex>[r] &\ar@<-0.7ex>[l]   \QC(\Spec k) :\pi_*
}
$$ associated to the map $\pi:\Spec k =EG\to BG$. It provides a
natural $k$-linear morphism
$$
\xymatrix{
\tilde \pi^*:\QC(BG)\ar[r] &
  \pi^*\pi_*\on{-comod}
  }
  $$ which we seek to show is a $k$-linear equivalence. Recall that a
$k$-linear equivalence is nothing more than a $k$-linear functor which
is an equivalence on the underlying plain $\oo$-categories. Thus we
can forget the $k$-linear structure and simply show that $\tilde\pi^*$
is an equivalence.
  
Observe that the above adjunction satisfies the criteria of the comonadic form of
Lurie's Barr-Beck theorem. Namely,  $\pi^*$ is conservative, and preserves limits by Lemma~\ref{finite CW}.
Moreover, we may invoke Proposition~\ref{package}, and by repeated base change in the diagram
$$
\xymatrix{
\ar[d]_-{\pi_1} G \simeq \Spec k \times_{BG} \Spec k \ar[r]^-{\pi_2} & \Spec k\ar[d]^-\pi \\
\Spec k \ar[r]^-\pi & BG
}
$$ 
see that the comonad $\pi^*\pi_*$ is 
equivalent to tensoring with the coalgebra 
$$
\pi^*\pi_*\cO_{\Spec k} \simeq \pi_{2*} \pi^*_1\cO_{\Spec k} \simeq
\cO(G).
$$ 
The proof for $\Aff(G)$ is exactly the same.
\end{proof}

%
%
%

Recall that the canonical morphism $G\to \Aff(G)$ is a group homomorphism, and hence the induced
identification $\cO(G)\simeq\cO(\Aff(G))$ respects coalgebra structures.

\begin{corollary}\label{affinized modules} Let $G$ be a group in finite CW complexes.
Then pullback along the canonical map $G\to \Aff(G)$ provides an
equivalence from the $\infty$-category of $k$-modules  with
$\Aff(G)$-action to that of $k$-modules with $G$-action.
\end{corollary}

\subsubsection{$G$-actions on affine derived schemes}
Now we consider the nonlinear setting of $G$-actions on affine derived schemes.

\begin{prop}Let $G$ be a  functionally compact group derived stack.
Then pullback along the canonical map $G\to \Aff(G)$ provides an
equivalence from the $\infty$-category of affine derived schemes with
$\Aff(G)$-action to that of affine derived schemes with $G$-action.

\end{prop}

\begin{proof} Note that the structure sheaf of an affine derived scheme is a compact object.
  Since the affinization functor $\Aff$
  restricted to the $\oo$-category of derived stacks with compact structure sheaf is monoidal, it takes module objects for $G$ to
  module objects for $\Aff(G)$. For $X$ an affine derived scheme with
  $G$-action, it follows that $X=\Aff(X)$ carries a canonical
  $\Aff(G)$-action. Moreover, the identity map $X\to \Aff(X)=X$
  intertwines the $G$ and $\Aff(G)$-actions on $X$.  Thus the
  $G$-action on $X$ factors through the corresponding $\Aff(G)$-action
  (equivalently, the action of the kernel of $G\to \Aff(G)$ is
  trivialized).  This provides an inverse construction to pullback
  along the map $G\to \Aff(G)$.
\end{proof}

It is worth noting that we can transport via Lemma~\ref{G modules on
  BG} the symmetric monoidal structure on the $\oo$-category $\QC(BG)$
to one on the $\oo$-category $\cO(G)\on{-comod}$. With this in mind,
we can restate the above proposition as saying that there is a
canonical equivalence between the $\oo$-category of affine derived
schemes with $G$-action and the opposite of the $\oo$-category of
connective commutative ring objects in
$\cO(G)$-comodules.


\subsection{Application: circle actions}

Let's apply the preceding to the circle $S^1=B\Z=K(\Z,1)\in
\cS$ considered as a derived stack. 
It is a group object
\cite[7.2.2.1]{topoi} in spaces, and thus also in derived stacks.  
Its algebra of cochains is given
explicitly as follows:
$$\cO(S^1) = C^*(S^1,k)\simeq H^*(S^1,k)=k[\eta]/(\eta^2),\quad
|\eta|=1.$$ We can think of this algebra as a graded version of the
dual numbers, and so think of maps out of the affinization $\Aff(S^1)$
as giving ``tangents" of degree $-1$.  The ring $\cO(S^1)$ is the
{free} commutative $k$-algebra (in the graded sense) on a single
generator of degree one (recall that we always work rationally).  Thus its spectrum $\Aff(S^1)$ is the
classifying stack of the additive group or shifted affine line.

\begin{lemma}
The affinization morphism is an equivalence of group derived stacks
$$
\xymatrix{
\Aff(S^1) \ar[r]^-\sim & B\Ga=K(\Ga,1)=\aline[1].
}$$
\end{lemma}

\begin{proof}
The canonical
homomorphism $\Z\to \Ga$ induces a homomorphism of group stacks $B\Z
\to B\Ga$.  The corresponding pullback on cochains is an equivalence,
and so the induced group homomorphism $\Aff(B\Z)\to B\Ga$ is an
equivalence. Hence the morphism $B\Z\to B\Ga$ agrees with the affinization morphism for $S^1 \simeq B\Z$.  
\end{proof}

The formality of the cochain algebra $\cO(S^1)\simeq H^*(S^1,k)$
implies that $ \Aff(S^1)\simeq B\Ga$ inherits a canonical action of
the multiplicative group $\Gm$ induced by its grading action on
$H^*(S^1,k)$. This will be a crucial structure in our analysis of
sheaves on loop spaces. We fix conventions so that $\Gm$ acts
trivially on the subspace $H^0(S^1, k)$, and rescales the subspace
$H^1(S^1, k)$ with weight one. Note that under the identification $
\Aff(S^1)\simeq B\Ga$, the $\Gm$-action determined by formality is
concretely given by the usual scaling $\Gm$-action on $\Ga$.


\subsubsection{Circle actions on modules and mixed complexes}

We will denote the homology, or ``group algebra", of the circle by 
$$
\bLa=
H_{-*}(S^1,k)=k[\lambda]/(\lambda^2), 
\quad
|\lambda|=-1.
$$
With our conventions, $H_{-1}(S^1,k)$ has $\Gm$-weight grading $-1$.
Convolution equips $\bLa$ with a commutative algebra structure dual
to the coalgebra structure on $\cO(S^1)=H^*(S^1,k)$. Thus
$\cO(S^1)$-comodules are the same thing as $\bLa$-modules.
Note that the  differential graded notion
or more relaxed $A_\oo$-notion of $\bLa$-module are equivalent, and so there is no risk
of ambiguity in our terminology.
Thus Corollary \ref{affinized modules}  identifies circle actions on
modules with square zero endomorphisms of degree $-1$, or so-called
{\em mixed complexes}. 

\begin{corollary}\label{mixed modules}
  The $\oo$-category of $k$-modules with $S^1$-action is
  equivalent to that of differential graded $\bLa$-modules (so-called  mixed complexes).
\end{corollary}
  
This is a well known story in other contexts, see \cite{Loday, Kassel,
  Dwyer Kan}.  The Borel homotopy theory of $S^1$-actions (or homotopy
theory over $BS^1$) is often modelled using the theory of cyclic
objects. In particular, the $\oo$-category of $k$-modules with
$S^1$-action is the Dwyer-Kan simplicial localization of the model
category of cyclic $k$-modules. The above corollary is the (localized)
$\oo$-categorical version of the result of Dwyer-Kan \cite{Dwyer Kan}
giving a Quillen equivalence of model categories between
$\bLa$-modules and cyclic $k$-modules.

\subsubsection{Circle actions on affine derived schemes}
We have seen that group actions on affine derived schemes
are the same thing as actions of their affinizations.
We have also seen that the latter are equivalent to connective commutative ring
objects in comodules over the group coalgebra.
Applying this to the circle, and using the duality of the the group coalgebra
with the group algebra, we have the following.

\begin{corollary}\label{mixed algebras}
  The $\oo$-category of affine derived schemes with
  $S^1$-action is equivalent to that of affine derived schemes
   with $B\Ga$-action.
   \end{corollary}

It is worth noting that we can transport via Lemma~\ref{G modules on BG} the  symmetric monoidal structure on  
the $\oo$-category $\QC(BS^1)$
to one on  the $\oo$-category of differential graded $\bLa$-modules. With this in mind,  we can restate the above corollary
as saying that  there is a canonical equivalence between
the $\oo$-category of affine derived schemes with
  $S^1$-action and
the opposite of
   the $\oo$-category of connective
  commutative ring objects in mixed complexes.

\begin{remark}[Hopf algebras and formality]
An important result of To\"en and
  Vezzosi~\cite{TV
    cyclic}  establishes the formality of $\cO(S^1)$  as a
  differential graded Hopf algebra.
  Such formality implies that the symmetric
  monoidal $\oo$-category of mixed complexes can be modelled
  concretely by the obvious tensor product structure on cochain
  complexes with a square zero endomorphism of degree $-1$. In particular, algebra objects therein can be modelled concretely by algebras
with a square zero derivation of degree $-1$.
We will see that when it comes to the loop space of an affine derived scheme, 
using
the additional structure of weights,  one can easily
 arrive at such a description of the $S^1$-action on the algebra of functions.
 \end{remark}


\subsection{Unwinding higher coherences} \label{sect: unwinding}

\nc{\GrDSt}{GrDSt}

Let us return to the general setting of a group derived stack $G$
acting on a derived stack $X$.  Recall that this involves an action
map $G\times X\to X$ along with higher coherent structures reflecting
that it is a group action.  We would like to investigate what
information is contained in the action map alone, and what is
contained in the higher coherent structures.


Let $\cS_*, \DSt_*$ denote pointed spaces and pointed derived stacks, respectively.
Recall the familiar suspension-loop adjunction
 $$
\xymatrix{
\Sigma:\cS_* \ar@<-0.7ex>[r] &\ar@<-0.7ex>[l]  \cS_*:\Omega
}
$$ We continue to use the same notation for its (target-wise)
extension to $\DSt_*$.

 To any $Y\in \cS_*$, we have the free monoid monad
$$
F(Y)= \Omega\Sigma Y\in \cS_*,
$$ (which is in fact a derived group object) along with a canonical
map $Y\to F(Y)$. (This is equivalent to the familiar James
construction of the free monoid on a space.)

We now consider $Y$ and $F(Y)$ as pointed derived stacks. By
adjunction, any candidate action map $\alpha_Y:Y\times X\to X$ in
$\DSt_*$ canonically extends to a $F(Y)$-action on $X$ such that the
restriction of the action map $\alpha_{F(Y)}:F(Y)\times X\to X$ along
the canonical map $Y\to F(Y)$ recovers the original map $\alpha_Y$.
In other words, any map of derived stacks $Y\to\Aut(X)$ canonically
extends to a group homomorphism $F(Y)\to\Aut(X)$ from the free group
such that the composition $Y\to F(Y)\to \Aut(X)$ is the original map.
Conversely, any $F(Y)$-action on $X$ is determined by the restriction
of its action map $\alpha_{F(Y)}: F(Y)\times X\to X$ along the
canonical map $Y\to F(Y)$.  In other words, any group homomorphism
$F(Y)\to\Aut(X)$ from the free group is determined by the composition
$Y\to F(Y)\to \Aut(X)$.

Now suppose we apply the above constructions when $Y$ is a group
derived stack $G$ pointed by the identity $pt\to G$.  Then we not only
have the canonical map $G\to F(G)$, but also a canonical group
homomorphism $\pi:F(G)\to G$ realizing $G$ as a ``quotient'' of a free
group.  Let us denote by $K$ the kernel (or homotopy fiber above the
identity) of the homomorphism $\pi$ so that we have a pullback diagram
of groups
$$
\xymatrix{
K= pt \times_G F(G) \ar[r]\ar[d] & F(G) \ar[d]^-\pi \\
pt\ar[r]& G
}
$$

Suppose further that we have a $G$-action on a derived stack $X$.
So in particular, by
restriction along $\pi$, we have a $F(G)$-action on $X$.  The
$F(G)$-action on $X$ is equivalent to the action map $G\times X\to X$
alone, while $K$ automatically acts trivially. Conversely, given an
$F(G)$-action, the action of $K$ is the first obstruction for the
action to descend to $G$. In the
classical setting of group actions on sets, this is the only obstruction, measuring the failure of the action map on $G$ to satisfy associativity. In the homotopical setting, however, there are higher obstructions
(since the above diagram is not necessarily a pushout diagram).

 \subsubsection{Unwinding circle actions}\label{unwinding circle}
Next let's specialize to the case when $G$ is the circle $S^1$.
Applying the above constructions, and the identity $\Sigma S^1 \simeq S^2$
so that $F(S^1) \simeq \Omega S^2$, 
we obtain the pullback diagram of groups
$$
\xymatrix{
\Omega S^3\simeq F(S^2) \ar[r] \ar[d]& \Omega S^2\simeq F(S^1) \ar[d]^-\pi\\
\{e\}\ar[r] & S^1
}
$$ 
This is nothing more than a rotation of the familiar Hopf
fibration $S^3\to S^2$. More precisely,  we can  obtain the above diagram by applying the limit
preserving functor $\Omega$ to the classifying map of the Hopf fibration
$$
\xymatrix{
S^3 \ar[r] \ar[d]&  S^2 \ar[d]^-\pi\\
pt\ar[r] & BS^1
}
$$ 

Recall that an action of the free group $\Omega S^2 \simeq F(S^1)$ on a space $X$
amounts to a pointed map $S^1\times X\to X$.
Similarly, since the kernel $\Omega S^3\simeq F(S^2)$ is a free group as well, its action on a
space $X$ amounts to a pointed map $S^2\times X\to X$.  

It is also clear
from this picture that an $\Omega S^2$-action with a
trivialization of the induced $\Omega S^3$-action does {\em not} amount to an
$S^1$-action. More precisely, an $\Omega S^2$-space is a space over
$S^2\simeq\C\pline\subset \C\pinfty\simeq BS^1$, and trivializing the induced $\Omega S^3$-action means extending the map to only $\C\pplane\subset \C\pinfty$ with  an
infinite sequence of higher obstructions remaining to overcome.

\subsubsection{Loop space actions}

Finally, we give an algebraic description of $\Omega X$-actions.
In preceding sections, we have made technical use of when a group $G$ is a finite CW complex. Now we will
make use of the assumption that the space 
$ X$ in which we consider loops is a finite CW complex.

\begin{lemma}\label{finite cw complex} Let $X$ be a finite, simply-connected CW complex.
There is a canonical $k$-linear equivalence 
$$
\QC(X)\simeq
  \cO(X)\on{-mod}
  $$ identifying sheaves with modules over global functions.
  The same assertion holds for the affinization $\Aff(X)$.
\end{lemma}

\begin{proof}
The assertion is a standard ascent argument. Namely, consider the usual $k$-linear adjunction
  $$
\xymatrix{
\pi^*:\QC(\Spec k)\ar@<-0.7ex>[r] &\ar@<-0.7ex>[l]   \QC(X) :\pi_*
}
$$  
associated to the map $\pi:X\to \Spec k$. It provides a natural $k$-linear morphism
$$
\xymatrix{
\tilde \pi_*:\QC(X)\ar[r] &
  \pi_*\pi^*\on{-mod}
  }
  $$
  which we seek to show is a $k$-linear equivalence. Recall that a $k$-linear equivalence is nothing more than a $k$-linear functor
  which is an equivalence on the underlying plain $\oo$-categories. Thus we can forget the $k$-linear structure
  and simply show that $\tilde\pi^*$ is an equivalence.
  
Observe that the above adjunction satisfies the criteria of the monadic form of
Lurie's Barr-Beck theorem. Namely,  $\pi_*$ is conservative (here we use that $X$ is simply connected), and preserves colimits by Lemma~\ref{finite CW}.
Moreover, we may invoke Proposition~\ref{package}, and by repeated use of the projection formula 
see that the monad $\pi_*\pi^*$ is 
equivalent to tensoring with the algebra 
$$
\pi_*\pi^*\cO_{\Spec k} \simeq 
\cO(X).
$$ 
The proof for $\Aff(X)$ is exactly the same.
\end{proof}

Recall that an $\Omega X$-module is simply an object of $\qc(B\Omega X)$, and that when $X$ is connected, there is an equivalence $B\Omega X \simeq X$.

\begin{corollary}
Let $X$ be a finite, connected, simply-connected CW complex.

(1) The $\oo$-category of $k$-modules with $\Omega X$-action is
  equivalent to that of $\cO(X)$-modules.

(2)  The $\oo$-category of affine derived schemes with
  $\Omega X$-action is equivalent to the opposite of the $\oo$-category of connective 
  commutative ring
  objects in the symmetric monoidal $\oo$-category of $\cO(X)$-modules.
\end{corollary}

%
%
%

%
%




\nc\Spf{\on{Spf}}

\section{Loops in derived schemes}\label{affine loop section}
Now we turn to loop spaces, in particular loop spaces of derived
schemes.  We continue to work over a fixed commutative
$\Q$-algebra~$k$.  By a derived scheme, we will always mean a
quasi-compact derived scheme with affine diagonal. In particular, one
could work with the assumption that the derived scheme is
quasi-compact and separated.

\subsection{Loops and tangents}

Recall that we think of the circle $S^1$ as a derived stack, and we
have identified its affinization with the classifying stack  $B\Ga$ of the additive group.

\begin{defn} The {\em loop space} of a derived stack $X$ is the derived
  mapping stack 
  $$\cL X=\Map_{\DSt}(S^1,X).$$
 \end{defn}

If we realize $S^1$ by two $0$-simplices connected by two
$1$-simplices, we obtain a concrete model for the loop space as a
fiber product of derived stacks along diagonal maps
$$
\cL X \simeq X \times_{X\times X} X.
$$ Informally speaking, $\cL X$ consists of pairs of points of $X$
which are equal and then equal again.  It follows immediately that if
$X$ has affine diagonal (in particular for quasicompact quasiseparated
derived schemes), then $\cL X$ is the relative spectrum of Hochschild
chains
$$
\cL X\simeq \Spec_{\cO_X}(\cO_X\ot S^1) =  \Spec_{\cO_X}(\cO_X\ot_{\cO_X\ot \cO_X} \cO_X).
$$

Thanks to the following simple observation, the case of derived
schemes (always quasicompact with affine diagonal) is no more difficult
than that of affine derived schemes.

\begin{lemma}\label{lem affine}
For $X$ a derived scheme, and $U\to X$ a Zariski open, the induced map
$\cL U\to \cL X$ is also Zariski open.  The assignment of loops
$$
\xymatrix{
U\ar@{|->}[r] & \cL U
}
$$
forms a cosheaf on the Zariski site of $X$.
\end{lemma} 

\begin{proof}

Recall that $\cL X$ is the relative spectrum of Hochschild chains of
$X$.  For $u: U\to X$ a Zariski open affine, thanks to the projection
formula and the obvious adjunction identity
$$
\xymatrix{
u^*u_*\cO_U \simeq \cO_U \otimes_{\cO_X} \cO_U \ar[r]^-\sim & \cO_U,
}
$$
we have an equivalence of $\cO_X$-algebras
$$
\xymatrix{
u^*\cO_{\cL X}= \cO_{\cL X}\otimes_{\cO_X} \cO_U \simeq (\cO_X\ot_{\cO_X\ot \cO_X} \cO_X)\ot_{\cO_X} \cO_U 
 \ar[r]^-\sim & 
\cO_U\ot_{\cO_U\ot \cO_U} \cO_U\simeq \cO_{\cL U} 
}$$
In other words, the natural maps form a Cartesian square
$$ \xymatrix{ \cL U \ar[d] \ar[r]^-{\cL u} & \cL X \ar[d]\\ U
  \ar[r]^-u & X }$$ and since $u$ is Zariski open affine, $\cL u$ is
as well.

The cosheaf assertion follows immediately.
\end{proof}

For an affine derived scheme $X=\Spec R$,
since maps $S^1\to X$ factor through the affinization $S^1\to B\Ga$,
we have an equivalence 
$$\cL X\simeq \Map_{\DSt}(B\Ga, X)
$$ intertwining the natural actions of $S^1$ and $B\Ga$.  By
Lemma~\ref{lem affine}, we see that the same holds true for any
derived scheme.

\medskip

Our immediate aim is to establish a more linear model for $\cL X$.
Recall that the cotangent complex of a derived $k$-algebra is a
connective complex \cite[Corollary 7.4.3.6]{HA}, and thus its shifted
symmetric algebra $\Omega^{-\bul}=\Sym \Omega^1_X[1]$ is again
connective. Thus we can consider $\Omega^{-\bul}$ as a derived
$\cO_X$-algebra.

\begin{defn} The {\em odd tangent bundle} of a derived stack $X$ is the 
linear derived stack
  $$\BT_X[-1]=\Spec_{\cO_X}  \Sym(\Omega_X[1])
  $$
  where $\Omega_X$ is the cotangent complex of $X$.
   \end{defn}

When $X$ is an underived smooth scheme, we can use the Koszul
resolution of the diagonal to calculate the Hochschild chain complex.
The following is a generalization of the resulting
Hochschild-Kostant-Rosenberg isomorphism \cite{HKR} to arbitrary
derived schemes.

\begin{prop}\label{prop loops=odd tangents} For $X$ a derived scheme, the loop space
  $\cL X$ is identified with the odd tangent bundle $\BT_X[-1]$ such
  that constant loops correspond to the zero section. Equivalently, we have
an equivalence of $\cO_X$-algebras
$$
\cO_X \otimes_{\cO_{X}\otimes \cO_X} \cO_X \simeq \Omega_X^{-\bul}.
$$
\end{prop}

\begin{proof}
By Lemma~\ref{lem affine}, it suffices to assume $X=\Spec R$ is affine.
Then maps $S^1\to X$ factor through the affinization $S^1\to B\Ga$,
and thus
we have an equivalence 
$$\cL X\simeq \Map_{\DSt}(B\Ga, X).
$$

  Recall that for $S\in\DGA_k^-$, the $S$-points of
  the derived stack $ \Map_{\DSt}(B\Ga, X)$ are
the mapping space
$$\Map_{\DSt}(\Spec S\times B\Ga,X)
\simeq \Map_\DSt(\Spec (S\oplus S[-1]), X) \simeq \Hom_{\DGA_k}(R,
S\oplus S[-1]).$$ In general, for any $S$-module $M$, we denote by
$S\oplus M$ the corresponding split square zero extension.

Next recall (see \cite[1.4.1]{HAG2}, \cite[7.4.1]{HA}) the universal
property of the cotangent complex $\Omega_R\in \Mod_R$ as representing
the functor of derivations. For any $S$-point $x:\Spec S\to X$, and
any connective $S$-module $M$, we have a pullback diagram


$$
\xymatrix{
  \Hom_S(x^*\Omega_R, M)\ar[d]\ar[r]& \Hom_{\DGA^-_k}(R,S\oplus M)\ar[d]\\
  x\ar[r]& Hom_{\DGA^-_k}(R,S) }
$$
or in other words, an identification 
$$
\Hom_S(x^*\Omega_R, M)\simeq\on{Der}_R(S,M).
$$

In fact this universal property of the cotangent complex holds without
any connectivity assumptions, but we can impose connectivity as above
thanks to the connectivity estimates of \cite[7.4.3]{HA} (see
also Remark \ref{complicial}).
Thus we can take $M=S[-1]$ (although $S\oplus S[-1]$ is no longer
connective), and obtain a pullback diagram
$$
\xymatrix{
  \Hom_{S}(x^*\Omega_R, S[-1])\ar[d]\ar[r]& \cL X(S)\ar[d]\\
  x\ar[r]& X(S) }
$$
The upper left hand corner is the fiber of $\cL X(S)$ over $x\in
X(S)$. It can be identified with the fiber $x^*\BT_X[-1]$ of the odd
tangent bundle via the adjunction
$$
\Hom_{S}(x^*\Omega_R[1], S)=\Hom_{\DGA^-_S}(x^*\Sym \Omega_R[1], S).$$
Thus  we have an equivalence $\cL X\simeq \BT_X[-1]$ of derived schemes over $X$.
\end{proof}

\begin{remark}\label{complicial}
An alternate point of view on the nonconnective setting is given by
\cite{HAG2}. Considering the connective $R\in \DGA_k^-$ as a not
necessarily connective algebra, i.e., an object of $\DGA_k$, it
corresponds to an affine derived scheme in the complicial HAG theory
of \cite{HAG2}. Its cotangent complex considered in the complicial
setting agrees with its cotangent complex in the connective setting
since it satisfies the same universal property on a more general class
of test rings.
\end{remark}

\begin{remark}
Given a map of derived schemes $X \to X^+$
we say that $X^+$ is a derived thickening of $X$
if the induced map $\pi_0(\cO_{X^+}) \to \pi_0(\cO_X)$
is an isomorphism. We have seen above that when $X$ is
a scheme, the loop space $\cL X$ is a derived
thickening of the constant loops $X$.
\end{remark}


\subsection{Equivariant sheaves and the de Rham differential}

We now turn to equivariant quasicoherent sheaves on loop spaces, and relate
them to the de Rham differential. We continue to assume $X$ is a quasicompact
scheme with affine diagonal.

Our aim here is to describe the stable $\oo$-category of
$S^1$-equivariant quasicoherent sheaves on the loop space $\cL X$.
Thanks to Corollary~\ref{mixed algebras} and the canonical
identification $\cL X \simeq \BT_X[-1]$ of Proposition~\ref{prop
  loops=odd tangents},
it suffices to describe the
stable $\oo$-category of $B\Ga$-equivariant quasicoherent sheaves on
$\BT_X[-1]$.  We will in fact take advantage of the evident extra
symmetries and describe the stable $\oo$-category of $B\Ga \rtimes
\Gm$-equivariant quasicoherent sheaves on $\BT_X[-1]$.

To start, 
recall that we have the $\cO_X$-algebra of differential forms $\Omega_X^{-\bullet} = \Sym(\Omega_X[1])$,
and the odd tangent bundle $\BT_X[-1]$ is its relative spectrum
$$
\BT_X[-1] =  \Spec_{\cO_X} \Omega_X^{-\bullet}
$$
so that we have an equivalence
$$
\qc(\BT_X[-1]) \simeq \Omega_X^{-\bullet}\module.
$$
Recall that the  differential graded notion
or more relaxed $A_\oo$-notion of module are equivalent, and so there is no risk
of ambiguity in our notation.

Recall as well that since $B\Ga$ is the affinization of the circle
$S^1$, its group algebra is the formal commutative ring $H_{-*}(S^1, k)$.
The $\Gm$-action on $B\Ga$ induces an action on
$\BT_X[-1]$, leading by our conventions to a $\Z$-grading
$\Omega_X^{-\bullet}$ such that $\Omega_X[1]$ is of weight $-1$. Thus
in conclusion, we have equivalences of stable $\oo$-categories
$$
\qc(\BT_X[-1]/\Gm ) \simeq  \Omega_X^{-\bullet}\module_\Z.
$$
$$
\qc( pt/ (B\Ga\rtimes \Gm)) \simeq H_{-*}(S^1, k)\module_\Z.
$$

We will write $\cA_X$ for the sheaf of  
$\Z$-graded $k$-algebras 
governing quasicoherent sheaves on the quotient $\BT_X[-1]/(B\Ga \rtimes \Gm)$,
or in other words, $B\Ga \rtimes \Gm$-equivariant sheaves on $\BT_X[-1]$.
To continue in an $\cO_X$-linear context, one could rather think of
$\cA_X$ as a $\Z$-graded algebra in endofunctors of $\qc(X)$.
Informally speaking, $\cA_X$ is the ``semidirect product" of the algebra 
$\Omega_X^{-\bullet}$ of functions on  $\BT_X[-1]$ 
with the group algebra $H_{-*}(S^1, k)$
of distributions on $B\Ga$ assembled from the $B\Ga$-action on $\BT_X[-1]$.
It will provide us the concrete realization 
$$
\cA_X\module_\Z \simeq  \qc(\BT_X[-1]/( B\Ga\rtimes\Gm)).
$$
Here as elsewhere we only consider modules in the $\oo$-category of quasicoherent
sheaves. So strictly speaking, an object of $\cA_X\module_\Z$ consists of a
$\Z$-graded object
of $\qc(X)$ together with module structure over $\cA_X$ thought
of as a $\Z$-graded algebra in endofunctors of $\qc(X)$.

Let us spell out the construction of $\cA_X$ more precisely.  We will
describe it for $X$ affine and it will evidently form a Zariski sheaf.
Consider the diagram
$$
\xymatrix{
 \BT_X[-1]/B\Ga & \ar[l]_-q  \BT_X[-1] \ar[r]^p & \Spec k = pt
}
$$

First,
 the quotient map 
$
q:\BT_X[-1]  \to \BT_X[-1]/B\Ga
$
provides an adjunction
$$
\xymatrix{
q^*:\qc(\BT_X[-1]/B\Ga) \ar@<-0.7ex>[r] &\ar@<-0.7ex>[l]\qc(\BT_X[-1]) 
\simeq\Omega_X^{-\bullet}\module:q_*
}
$$ and hence a comonad $q^*q_*$ acting on $\qc(\BT_X[-1])$. Following
the same argument as Lemma~\ref{G modules on BG}, we see explicitly
that the comonad $q^*q_*$ is given by tensoring over $k$ with the
group coalgebra $\cO(B\Ga) \simeq H^*(S^1, k)$.
Furthermore, a similar invoking of the Barr-Beck theorem provides a
$k$-linear equivalence
$$
 \qc(\BT_X[-1]/ B\Ga)\simeq \on{Comod}_{q^*q_*}(\qc(\BT_X[-1]))  \simeq 
 \on{Comod}_{\cO(B\Ga)} (\Omega_X^{-\bullet}\module).
$$
Dualizing the group coalgebra $\cO(B\Ga) \simeq H^*(S^1, k)$ gives the group algebra 
$\cO(B\Ga)^\vee \simeq H_{-*}(S^1, k)$, and a further $k$-linear equivalence
$$
 \qc(\BT_X[-1]/ B\Ga)\simeq \on{Mod}_{\cO(B\Ga)^\vee}(\Omega_X^{-\bullet}\module).
$$

Finally,  the projection
$
p:\BT_X[-1]  \to \Spec k = pt
$
provides an adjunction
$$
\xymatrix{
p^*:k\module \simeq \qc(\Spec k) \ar@<-0.7ex>[r] &\ar@<-0.7ex>[l]\qc(\BT_X[-1]) 
\simeq\Omega_X^{-\bullet}\module:p_*
}
$$  
for which $p_*$ is both conservative and preserves colimits since it is an affine pushforward.
Hence it induces an adjunction
$$
\xymatrix{
\tilde p^*:k\module \simeq \qc(\Spec k)  \ar@<-0.7ex>[r] &\ar@<-0.7ex>[l]
\on{Mod}_{\cO(B\Ga)^\vee}(\Omega_X^{-\bullet}\module):\tilde p_*
}
$$
which satisfies the monadic form of the Barr-Beck theorem. Namely, $\tilde p_*$ is conservative
and preserves colimits since it is the composition of such functors: the functor
which forgets $\cO(B\Ga)^\vee$-module structure and the affine pushforward  $p_*$.
Following the same argument as in Lemma~\ref{finite cw complex}, we see that the resulting monad $\tilde p_*\tilde p^*$ is given by tensoring
with a $k$-algebra which we take to be $\cA_X$.

All of the above
constructions are $\Gm$-equivariant, and hence $\cA_X$ is naturally
$\Z$-graded.
As a $\Z$-graded sheaf of $k$-algebras, $\cA_X$ is an extension
$$
\xymatrix{
\Omega^{-\bullet}_X \ar[r] & \cA_X \ar[r] & H_{-*}(S^1, k).
}
$$
By construction, as a $\Z$-graded sheaf of $k$-modules, the extension naturally splits 
$$
\cA_X \simeq \Omega_X^{-\bullet} \otimes_k H_{-*}(S^1, k).
$$ 
To see this, observe that functions $\Omega^{-\bullet}_X$ form a natural $\cA_X$-module
with canonical cyclic vector given by the unit $1\in \cO_X \subset \Omega^{-\bullet}_X$.
Acting on the unit vector and taking the kernel provides a canonical splitting
$
H_{-*}(S^1, k) \to \cA_X.
$

\subsubsection{The de Rham differential and formality}\label{formal section}
Now we will calculate the algebra structure on $\cA_X$.  By
construction, the subalgebras $\Omega_X^{-\bullet}$ and $H_{-*}(S^1,
k)$ appear with their usual relations.  We wish to identify the
relations between them in terms of the de Rham differential, which we can interpret in this general setting as the tautological vector field of degree $-1$ on the shifted
tangent bundle $\BT_X[-1]$.

Recall that the $S^1$-action on $\cL X$ corresponds to the
$B\Ga$-action on $\BT_X[-1]$.  Focusing on the action map and ignoring
higher coherences, this provides a vector field of degree $-1$.

\begin{prop}\label{prop rotation=de rham} 
The $B\Ga$-action map on $\BT_X[-1]$ is given by the canonical degree $-1$
vector field represented by the de Rham differential.
\end{prop}

\begin{proof} 
For any group $G$, and space $X$, consider the $G$-action on the mapping space $X^G=\Map(G, X)$.
The action map $a:G\times X^G\to X^G$ is evidently $G\times G$-equivariant, and hence determined
by the evaluation map 
$$
\xymatrix{
ev:G\times X^G\ar[r]^-a &  X^G \ar[r]^-{ev_e} & X
}
$$Ê
given by composition with evaluation at the identity $ev_e:X^G \to X$.
Under adjunction, the evaluation $ev:G\times X^G \to X$ corresponds to the identity $\id:X^G \to X^G$.

Our task is to unwind the above in the case $G = B\Ga$ and $X$ a smooth scheme.
Passing to functions, we must understand the ring map 
$$
\xymatrix{
\delta:\cO_X \ar[r] & \Omega_X^{-\bullet} \otimes_k H^*(S^1, k) \simeq \Omega_X^{-\bullet}[\eta]/(\eta^2),
& |\eta| = 1,
}
$$
corresponding to the identity map
$
\Omega_X^{-\bullet} \to \Omega_X^{-\bullet}.
$
Such a map is induced by a map
$$
\xymatrix{
\delta:\cO_X \ar[r] & \cO_X \oplus (\Omega_X^1[1]\cdot\eta) \simeq  \cO_X \oplus \Omega_X^1 
&
\delta(f) = (f, \delta_1(f))
}
$$
such that $\delta_1$ is a derivation. 

Now given any $\cO_X$-module $\cM$, we can form the linear space $M = \Spec_{\cO_X} (\cM[1])$ which comes equipped with a natural map $M \to \BT_X[-1]$. Furthermore, given a derivation $\delta_\cM:\cO_X\to \cM$, we obtain a corresponding $B\Ga$-action on $M$ such that the natural map $M \to \BT_X[-1]$ 
is $B\Ga$-equivariant. Returning to action maps, this provides a factorization of the derivation
$$
\xymatrix{
\delta_\cM: \cO_X\ar[r]^-{\delta_1} & \Omega_X^1 \ar[r] & \cM.
}
$$
Thus $\delta_1$ satisfies the universal property of the de Rham differential.
\end{proof}

In the present setting of arbitrary derived schemes over the
rationals, it seems reasonable to adopt the tautological translation
action of the odd line $B\Ga=\aline[1]$ on functions on $\BT_X[-1]$ as
the definition of the de Rham complex. However, in the familiar and concrete setting
of underived smooth schemes, there already is a standard notion of de
Rham complex. Using a rigid model such as strictly commutative
differential graded algebras, we can adjoin the de Rham differential $d$ to the algebra of differential forms $\Omega_X^{-\bul}$ with
the usual relations. In the work
\cite{TV cyclic}, To\"en and Vezzosi prove in a model categorical
setting that these two notions coincide. In the present
$\oo$-categorical setting, this can be seen directly by a simple weight argument which we
now present.

For the rest of this section, we restrict our attention to $X$ a
quasi-compact smooth {underived scheme} with affine diagonal over a
fixed commutative rational algebra $k$.

First, we check that the classical action of the de Rham differential $d$ on the de Rham algebra $\Omega_X^{-\bul}$ canonically extends to an $S^1$-action. We will then check that this $S^1$-action coincides with the rotation $S^1$-action on functions on the loop space.

To begin, recall that since $\Omega S^2$ is the free group on $S^1$,  the action of $d$ presents $\Omega_X^{-\bul}$ as a sheaf of algebras over $S^2$ which we denote by $\cB_1$. Our aim is to show that  $\cB_1$ extends canonically from $S^2\simeq \C\mathbf P^1$ to a sheaf of algebras $\cB$ over all of $BS^1 \simeq \C\mathbf P^\infty$.
We will proceed by induction via the standard gluing diagram
$$
\xymatrix{
\ar[d] S^{2n+1} \ar[r]^-{u_n} & \C\mathbf P^n \ar[d] & \\
D^{2n+2} \ar[r] & \C\mathbf P^{n+1}
}
$$

As a first step, when $n=1$, to confirm that $\cB_1$ extends from $\C\mathbf P^1$ to a sheaf of algebras $\cB_2$ over $\C\mathbf P^2$, we must trivialize the pullback
$u_1^*\cB_1$. The pullback $u_1^*\cB_1$ is governed by its $\Omega S^3$-action, which since $\Omega S^3$ is the free group on $S^2$ is encoded by an endomorphism of degree $-2$. Via the map  $\Omega S^3\to  \Omega S^2$ induced by the Hopf fibration, we see that the endomorphism is nothing more than the square $d^2=0$ of the de Rham differential which is identically zero. Thus we can canonically trivialize the pullback $u_1^*\cB_1$ by the identically zero endomorphism $\flat_{-3}=0$ of degree $-3$ whose coboundary is the square $d^2=0$.

In general, suppose by induction we have constructed the sheaf of algebras $\cB_n$ over $\C\mathbf P^n$ via the trivializations $d$, and the zero endomorphisms $\flat_{-3}= \cdots =\flat_{-2n+1}=0$. Then to confirm that $\cB_n$ extends to a sheaf of algebras $\cB_{n+1}$ over $\C\mathbf P^{n+1}$, we must again trivialize the pullback $u_n^*\cB_n$ over $S^{2n+1}$. Here we consider the obvious gluing diagram
$$
\xymatrix{
\ar[d] S^{1}\times S^{2n-1} \ar[r] & D^2\times S^{2n-1} \ar[d]^-{p_{2n-1}}  \\
S^1\times D^{2n} \ar[r]^-{q_1} & S^{2n+1}
}
$$
Then the pullbacks $q_1^*u_n\cB_{n}$ and $p_{2n-1}^*u_n^*\cB_{n}$ are trivialized by 
the de Rham differential $d$ and the identically zero endomorphism $\flat_{2n+1}=0$ respectively.
It follows directly from standard definitions that the degree $-2n$ endomorphism governing the pullback 
$u_n^*\cB_n$ over $S^{2n+1}$ is nothing more than the Whitehead bracket $[d, \flat_{-2n+1}]$.
But since $\flat_{-2n+1}$ is identically zero, so is its bracket with $d$. Thus we can
canonically trivialize 
$u_n^*\cB_n$ over $S^{2n+1}$ by taking the identically zero endomorphism $\flat_{-2n-1}=0$ of degree $-2n-1$ whose coboundary is the trivial bracket $[d, \flat_{-2n+1}] = 0$.

Thus to summarize, we have exhibited a canonical $S^1$-action on the de Rham algebra $\Omega_X^{-\bul}$ whose action map is given by the standard action of the de Rham differential $d$.

Now we return to the descent algebra $\cA_X$ which describes the rotation $S^1$-action on the algebra $\cO_{\cL X}\simeq \Omega_X^{-\bullet}$ of  functions on the loop space.
By Proposition \ref{prop rotation=de rham}, we know that
the commutation
relations between $\omega \in\Omega_X^{-\bullet} $ and
the generator $\lambda\in H_{-1}(S^1, k)$
are given by the de Rham differential 
$$
[\lambda, \omega] = dw.
$$ Our aim in what follows is to show that $\cA_X$ satisfies the
relations of the de Rham complex precisely. Namely, we consider
$\cO_{\cL X}\simeq \Omega_X^{-\bullet}$ as an algebra with
$S^1$-action given on the one hand by loop rotation, and on the other
hand as described above by the de Rham differential.  Equivalently,
interpreting the $S^1$-action as descent data to $BS^1$, we consider
the corresponding algebra objects over $BS^1$ which we denote by
$\ul{\cO}_{\cL X}$ and $\ul{\Omega}_X^{-\bullet}$ respectively. Using
weights, we will show that these two constructions are equivalent.

\begin{lemma}
Let $\cF$ be a graded sheaf on $BS^1$ with cohomology
concentrated along the diagonal weight-degree bigrading. Then any
section of the restriction of $\cF$ along the standard map $\C\mathbf
P^1\to BS^1$ extends to a section of $\cF$ over all of $BS^1$.
\end{lemma}

\begin{proof}
Suppose by induction we have extended to a section $\sigma_n$ over $\C\mathbf P^{n}$ and seek to extend it to a section over $\C\mathbf P^{n+1}$. Consider the standard gluing diagram
$$
\xymatrix{
\ar[d] S^{2n+1} \ar[r]^-{u} & \C\mathbf P^n \ar[d] & \\
D^{2n+2} \ar[r] & \C\mathbf P^{n+1}
}
$$ Consider the pullback $u^*\cF$ and the pulled back section
$u^*\sigma_n$. Then the obstruction for $\sigma_n$ to extend across
$D^{2n+2}$ is governed by an operator $\flat_{-2n-1} \in
\Ext^{-2n-1}(\cF)$.  In other words, $\sigma_n$ extends if we have
$\flat_{-2n-1}(u^*\sigma_n) = 0$ cohomologically.  Finally, observe
that $\flat_{-2n-1}$ is of weight grading $-n-1$, and so if $n>0$,
our assumption on $\cF$ implies that $\flat_{-2n-1}$ vanishes
in cohomology.
\end{proof}

We will apply the above lemma to the sheaf of algebra morphisms
$$\cF= \IntHom_{Alg/BS^1}(\ul{\cO}_{\cL X},
\ul{\Omega}^{-\bullet}_X).$$ By assumption, $X$ is smooth and
underived, hence the algebra $\Omega_X^{-\bullet}$ is concentrated
along the diagonal with cohomological degree equal to weight.  The
algebra equivalence $\cO_{\cL X}\simeq \Omega_X^{-\bullet}$ implies
that $\cF$ admits a canonical section over a point.  The fact that
$\lambda\in H_{-1}(S^1, k)$ acts by the de Rham differential implies
that the section extends over $\C\mathbf P^1$. Namely, as explained in
Section \ref{unwinding circle}, forgetting from $S^1$-actions (or
objects over $BS^1$) to $\Omega S^2$-actions (or objects over $\C\mathbf P^1$)
corresponds to forgetting all but the action map.  The lemma then
implies the section extends over all of $BS^1$. Thus we arrive at the
following conclusion.

\begin{thm}\label{formality}
There is an equivalence of sheaves of $k$-algebras
$$
\cA_X \simeq \Omega_X^{-\bullet}\la\delta\ra/(\delta^2, [\delta, \omega] - dw).
$$
In other words, $\cA_X$ results from adjoining to the algebra of differential forms
$\Omega_X^{-\bullet}$ a single element
$\delta$ of degree $-1$ satisfying the relations of the de Rham differential.
\end{thm}

\begin{remark}
It is entertaining to contemplate informally the above result from a purely algebraic perspective.

A $k$-algebra with a $H_{-*}(S^1, k)$-action can be encoded by a
degree $-1$ derivation $\delta_{-1}$ given by the action of $\lambda \in
H_{-1}(S^1, k)$ together with higher coherences. 
The derivation
$\delta_{-1}$ encodes the restriction of the algebra to $S^2\simeq \C\mathbf P^1
\subset BS^1 \simeq \C\mathbf P^\oo$.
Without higher coherences, the derivation $\delta_{-1}$ would
encode a $k$-algebra with a $H_{-*}(\Omega S^2, k)$-action, since
 $\Omega S^2 $ is the free group on $S^1$, and hence
$H_{-*}(\Omega S^2, k)$ is the free {associative} algebra on a generator of degree~$-1$.

The higher coherences governing a $H_{-*}(S^1, k)$-action impose the equations that $H_{-*}(S^1, k)$ is in fact {commutative}.
So for example, the first coherence imposes the vanishing of the commutator $2 \delta_{-1}^2 =
[\delta_{-1}, \delta_{-1}]$. (Recall that $k$ is a rational algebra, so integral coefficients which arise are invertible.)  This equation is realized by a degree $-3$ operation $\delta_{-3}$ whose coboundary
is the commutator $[\delta_{-1}, \delta_{-1}]$. In fact, a $k$-algebra with a 
$H_{-*}(\Omega\C \mathbf P^2, k)$-action is completely encoded by the operations $\delta_{-1}, \delta_{-3}$ subject to the relation that the coboundary of $\delta_{-3}$
is the commutator $[\delta_{-1}, \delta_{-1}]$.

In general, as we pass from $\C\mathbf P^n$ to $\C\mathbf P^{n+1}$, we add another operation
killing previous commutators. Each of these higher operations satisfy that its $\Gm$-weight is greater than its cohomological degree. Thus for the $k$-algebra $\Omega_X^{-\bullet}$, which is concentrated along the diagonal of the weight-degree bigrading, all of the higher operations beyond $\delta_{-1}$ must vanish.
\end{remark}

In differential geometry, there is a familiar relationship between 
quasicoherent sheaves $\cE$ with flat connection $\nabla$ on $X$ and modules for the 
algebra $ \cA_{X}$.
Namely, 
the connection $\nabla:\cE\to \cE\ot\Omega_X$ naturally extends to an $ \cA_{X}$-module structure on the pullback
$$
\pi^*\cE = \cE\ot_{\cO_X} \Omega_X^{-\bullet},
\qquad
\mbox{ where }\pi:\BT_X[-1]\to X.
$$
The action of the 
distinguished element $\delta\in \cA_{X}$
is given by the formula
$$
\delta(\sigma\ot \omega) = \nabla(\sigma)\wedge \omega + \sigma\ot d\omega,
\qquad
\sigma\ot \omega \in\cE\ot_{\cO_X}\Omega_X^{-\bullet}.
$$
The fact that $\nabla$ is flat leads directly to the identity $\delta^2=0$.
To further refine this, we will consider
a suitable form of Koszul duality in a later section.



\subsubsection{Periodic localization}


Finally, we will calculate the periodic localization of
$B\Ga\rtimes\Gm$-equivariant quasicoherent sheaves on $\BT_X[-1]$ (see
Section \ref{graded section} for a general discussion). We keep the
hypotheses that $X$ be smooth and underived.

Since  $B\Ga$ is the affinization of the circle $S^1$,
 we can identify the ring of functions on its classifying stack
$$\cO(BB\Ga)  \simeq H^*(BS^1, k) \simeq k[u],
\quad\mbox{ with $u$ of degree $2$ and weight $1$.}
$$

Via the natural projection $\BT_X[-1]/(B\Ga \rtimes \G_m) \to pt /B\Ga = BB\Ga$, we can equip
the stable $\oo$-category 
$$
\qc(\BT_X[-1]/(B\Ga \rtimes \G_m)) \simeq \cA_X\module_\Z 
$$ with a module structure over $k[u]\module_\Z$.
Recall that the periodic localization is the base change
$$
\cA_X\module_{per}
= 
\cA_X\module_\Z  \ot_{k[u]\module_\Z} k[u, u^{-1}]\module_\Z.
$$

Recall as well that the periodic localization is independent of shear shifting the algebra
by even degrees.
To connect with familiar geometric objects,  we will focus on the shear shifted algebra
$$
 \cA_{X, [-2]} \simeq \Omega_X^{-\bullet}\la\delta\ra/(\delta^2, [\delta, \omega] - dw).
$$
where differential forms appear in their usual nonnegative degrees, and the de Rham differential
$\delta$ in degree $1$.

Let $k[t]$ be the polynomial algebra with $t$ of degree $0$ and weight $1$. Then 
the $k[u]\module_\Z$-structure on $\cA_X\module_\Z$ translates
to a $k[t]\module_\Z$-structure on $ \cA_{X, [-2]}$.
We can calculate the periodic localization in the form
$$
\cA_{X}\module_{per} \simeq \cA_{X, [-2]}\module_{per} 
=
\cA_{X,[-2]}\module_\Z  \ot_{\qc(\aline/\Gm)} \qc(pt),
$$
where 
$pt = \Gm/\Gm \subset \aline/\Gm$ is embedded in the usual way.

We will write $\Omega^{\bul}_{X, d}$ for the de Rham complex of $X$, thought of as a sheaf
of non-formal
differential graded algebras.

\begin{thm}\label{local de rham complex}
There is a canonical equivalence of stable $\oo$-categories
$$
\qc(\BT_X[-1])^{B\Ga\rtimes\Gm}_{per} \simeq \Omega^{\bul}_{X, d}\module.
$$

\end{thm}

\begin{proof}
 The localization functor $\cA_{X, [-2]}\module_{\Z}
 \to \cA_{X, [-2]}\module_{per}$ 
preserves colimits. Thus it suffices to identify the sheaves of $k$-algebras
$$
\Omega^{\bul}_{X, d} \simeq \cA_{X, [-2]}\ot_{k[t]} k[t, t^{-1}].
$$ 
This is a standard local calculation.
\end{proof}



\subsection{Holomorphic curvature}

We continue with $X$ a quasi-compact smooth {underived scheme} with affine diagonal
over $k$.

Our aim here is to describe the stable $\oo$-category of
$\Omega S^2$-equivariant quasicoherent sheaves on $\cL X$.
Thanks to the canonical identification $\cL X \simeq \BT_X[-1]$ intertwining the natural $S^1$ and $B\Ga$-actions
along the affinization map $S^1\to B\Ga$, it suffices
to 
describe the stable $\oo$-category of
$\Aff(\Omega S^2)$-equivariant quasicoherent sheaves on $\BT_X[-1]$.
We will keep track of the evident extra symmetries and describe the stable $\oo$-category of
$\Aff(\Omega S^2) \rtimes \Gm$-equivariant quasicoherent sheaves on $\BT_X[-1]$.

%
%
%
%

Following Section~\ref{sect: unwinding} and in particular Lemma~\ref{finite cw complex}, we have the sheaf of $\Z$-graded $k$-algebras  $\tilde\cA_X$ 
governing $\Aff(\Omega S^2)\rtimes \Gm$-equivariant quasicoherent sheaves on $\BT_X[-1]$.
Its construction follows from the adjunction given by the diagram
$$
\xymatrix{
 p:\BT_X[-1]/\Aff(\Omega S^2)  \ar[r] & S^2 \ar[r] & \Spec k =pt
}
$$
More specifically, over an open set, $\tilde\cA_X$ is simply functions along the fibers of $p$. 
It provides us the concrete realization
$$
\qc(\BT_X[-1]/(\Aff(\Omega S^2) \rtimes \G_m))
\simeq
\tilde\cA_X\module_\Z.
$$

We have the following analogue of Theorem~\ref{formality}.
Consider the cohomology 
$$
\cO(S^2) \simeq C^*(S^2, k)  \simeq H^*(S^2, k) \simeq k[u]/(u^2)
$$
as a graded algebra with $u$ in cohomological degree $2$ and $\Gm$-weight grading $1$.
Consider the $\Z$-graded sheaf of 
differential graded $k$-algebras on $X$ given by the tensor product 
$$
\xymatrix{
\Omega_{X, S^2}^{-\bullet} = \Omega_{X}^{-\bullet}\otimes \cO(S^2) \simeq \Omega_{X}^{-\bullet}[u]/(u^2)
}
$$
with differential 
$$
\xymatrix{
\delta(\omega \otimes 1) = d\omega \otimes u & \delta(\omega \otimes u) = 0
}
$$

\begin{thm}\label{free formality}
There is an equivalence of sheaves of $k$-algebras
$$
\tilde \cA_X \simeq \Omega^{-\bullet}_{X, S^2}
$$
 \end{thm}

\begin{proof}
The assertion follows from the calculation of the $S^1$-action map
 in Proposition~\ref{prop rotation=de rham}.
\end{proof}

We will interpret the theorem in terms of sheaves with not necessarily flat connection in the section immediately following.


\section{Koszul dual description}\label{koszul dual section}
This section contains reformulations of results from the previous
section in a Koszul dual language.  In the previous section, we
identified equivariant sheaves on the loop space of a derived scheme
$X$ with equivariant sheaves on the odd tangent bundle $\BT_X[-1]$ or
modules over the de Rham algebra $\cA_X$.  In this section, we will
give a Koszul dual description in terms of modules over the Rees
algebra of differential operators, or in other words, in terms of
sheaves on a noncommutative deformation of the cotangent bundle
$\BT^*_X$, inspired by \cite{Kap dR} and \cite{BD Hitchin} as well as
\cite{GKM}. As a result, we can describe equivariant sheaves on loop
spaces in terms of $\D$-modules. The
discussion will rely on the yoga of gradings, weights and periodic
localization described in Section \ref{graded section}.
Although we provide some details, none of the arguments concerning Koszul duality
are really new.

\medskip

We continue to assume throughout this section that $X$ is a quasi-compact smooth underived scheme
with affine diagonal.
In particular, we can appeal to the well developed theory of differential operators
and Koszul duality. (Presumably the results hold in greater generality once
differential operators are defined appropriately.)


\subsection{Graded Koszul duality}
Let $\cD_X$ be the sheaf of differential operators on $X$, and for
$i\geq 0$, let $\cD^{\leq i}_X\subset \cD_X$ be the differential
operators of order at most~$i$. In particular, $\cD^{\leq 0}_X$ is the
commutative ring of functions $\cO_X$, and $\cD_X^{\leq 1}$ is the
usual Picard Lie algebroid of vector fields $\BT_X$ extended by
functions $\cO_X$.

As usual, by a $\cD_X$-module, we will mean a quasicoherent
$\cO_X$-module with a compatible action of $\cD_X$.

Next, recall the standard graded Rees algebra associated to $\cD_X$ with its filtration
by order of operator. Consider the $\Gm$-action on $\aline=\Spec k[t]$
so that $t$ has weight $1$.
The Rees algebra $\cR_{X}$ is the sheaf of $\Z$-graded $k$-algebras
$$
\cR_{X}=\bigoplus_{i\geq 0} t^i \D_{X}^{\leq i} \subset \cD_X \otimes_k k[t].
$$ Note that $\cR_X$ contains $k[t] = \cO(\aline)$ as a central
subalgebra, and the $\Z$-grading on $\cR_X$ reflects the fact that
$\cR_{X}$ descends from $X\times \aline$ to the quotient $X\times
(\aline/\Gm)$.

By a $\Z$-graded $\cR_X$-module, we will mean a $\Z$-graded
quasicoherent $\cO_{X\times \BA^1}$-module with a compatible action of
$\cR_X$.

The fiber of $\cR_X$ at the origin  is the commutative associated graded
$$ \cR_X \ot_{k[t]} k[t]/(t) \simeq \on{gr}\cD_X\simeq \cO_{\BT^*_X}
\simeq \Sym(\BT_X).
$$ The restriction of $\cR_X$ to the generic orbit $pt = \Gm/\Gm
\subset \aline/\Gm$ is the original algebra $\D_X$, or equivalently,
there is a canonical isomorphism of $\Z$-graded algebras
$$
\cR_{X} \ot_{k[t]} k[t, t^{-1}]
\simeq
\cD_X\ot_k k[t, t^{-1}].
$$
Thus the periodic localization of $\cR_X$-modules
are simply $\cD_X$-modules
$$
\cR_{X}\module_{per} = \cR_X\module_\Z \ot_{k[t]\module_\Z} k[t, t^{-1}]\module_\Z
\simeq
\cD_X\module.
$$



For usual technical reasons, we will work with a completed version of
the Rees algebra.  Consider the $\Gm$-action on the ind-scheme $\wh
{\mathbb A}^1 =\Spf k[[t]]$ so that $t$ has weight $1$.  We define the
completed Rees algebra $\wh \cR_{X}$ to be the sheaf of completed
$\Z$-graded $k$-algebras
$$ \wh \cR_{X}=\lim_{j\to \oo} \cR_X/ \oplus_{i \geq j} t^i
\D_{X}^{\leq i} \subset \cD_X \otimes_k k[[t]].
$$
Recall that the $\Gm$-action on $\aline$ places $t$ in weight $1$,
hence $\wh \cR_{X}$ is the completion of $\cR_X$ with respect to
the positive weight direction.
 The fiber of $\wh \cR_{X}$ at the origin 
 is the ring of functions on the completion of ${\BT^*_X}$ along the zero section

$$
\wh\cR_X \wh\ot_{k[[t]]} k[[t]]/(t) \simeq \lim_{j\to \oo}\Sym(\BT_X)/\oplus_{i\geq j}\on{Sym}^{i}(\BT_X)
\simeq \cO_{\wh \BT^*_X}
$$

%

Now we have the following noncommutative form of Koszul duality (see
also~\cite{Kap dR}, \cite{BD Hitchin}).

\begin{thm}\label{koszul}
There is a canonical equivalence of stable $\oo$-categories
$$
\cA_{X}\module_\Z
\simeq
\wh\cR_{X}\topmodule_\Z
$$
\end{thm}

\begin{proof}
It is more natural to give an equivalence
$$
\cA_{X, [-2]}\module_\Z
\simeq
\wh\cR_{X}\topmodule_\Z
$$
and then use Proposition~\ref{shift indep} to arrive at the assertion.

We will give a proof via the $\oo$-categorical Barr-Beck theorem
of \cite[6.2.2]{HA}. Consider the adjoint functors
$$
\xymatrix{
\ell:\cO_X\module_\Z \ar@<-0.7ex>[r] &\ar@<-0.7ex>[l]\wh\cR_{X}\topmodule_\Z:r
}
$$
$$
\ell(\cF) =  \cF
\qquad
r(\cM) = \Hom_{\wh\cR_{X}}(\cO_X, \cM).
$$
where $\ell$ is simply the pushforward along the natural augmentation $\wh\cR_{X} \to \cO_X$.
Note that $r$ also preserves colimits since $\cO_X$ is dualizable
and so compact (it can be realized by a finite Koszul
resolution). Furthermore, $r$ is conservative by Nakayama's lemma (as
in the proof of Theorem~\ref{formal}).

Let $T= r\circ \ell$ be the resulting monad. By the Barr-Beck theorem,
we have an equivalence
$$
\wh\cR_{X}\topmodule_\Z \simeq \Mod_{T}(\cO_X\module_\Z)
$$ It is a standard calculation (see for example \cite{Kap dR},
\cite{BD Hitchin}) that $T$ is given by the algebra $\cA_{X, [-2]}
\simeq \Hom_{\wh\cR_{X}}(\cO_X, \cO_X)$.
\end{proof}

A shortcoming of  the equivalence of Theorem~\ref{koszul} is that involves the completed
Rees algebra rather than the ordinary
version. To remedy this, we can restrict
 to stable $\oo$-categories of  suitably finite modules.
 Let us write  $\Perf_X(\cL X) $ for the small, stable $\oo$-subcategory of $\qc(\cL X)$ of quasicoherent
sheaves whose pushforward along the canonical map $\cL X\to X$ are perfect.
Note that the structure sheaf $\cO_{\cL X} \simeq \Sym_{\cO_X}(\Omega_X[1])$ is perfect over 
$\cO_X$, and the $\cO_{\cL X}$-augmentation module $\cO_X$ is obviously perfect over $\cO_X$ (though not over $\cO_{\cL X}$). 
 We write $\cR_X\perf$ for the stable $\oo$-category of quasicoherent
sheaves on $X$ equipped with a compatible structure of perfect $ \cR_X$-module.

\begin{corollary}
There is a canonical equivalence
 of small, stable $\oo$-categories 
 $$
 \Perf_{X}(\cL X)^{B\Ga \rtimes \G_m}
\simeq
\cR_{X}\perf_\Z.
$$
\end{corollary}

Finally, we observe that the functors realizing the equivalence of
Theorem~\ref{koszul} are naturally linear over $\cO(BB\G_a) \simeq
H^*(BS^1, k)$, hence the equivalence is linear and we can pass to its
periodic localization. In particular, if we take the image of the
above equivalence under periodic localization, we obtain a canonical
equivalence of perfect modules.

\begin{corollary}\label{cor periodic koszul}
There is a canonical equivalence
 of small, stable $\oo$-categories 
$$
\Perf_{X}(\cL X)^{B\Ga \rtimes \G_m}_{per}
\simeq \cD_X\perf.
$$
\end{corollary}


\subsection{Twisted $\D$-modules}

Our aim here is to interpret 
$\Omega S^2$-equivariant quasicoherent sheaves on $\cL X$
in terms of Rees modules. The arguments follow the same pattern as in the previous section and we leave it to the reader
to repeat them as necessary,

Thanks to the canonical identification $\cL X \simeq \BT_X[-1]$ intertwining the natural $S^1$ and $B\Ga$-actions
along the affinization map $S^1\to B\Ga$, it suffices
to 
describe 
$\Aff(\Omega S^2)$-equivariant quasicoherent sheaves on $\BT_X[-1]$.
We will in fact work with the evident extra symmetries and describe 
$\Aff(\Omega S^2) \rtimes \Gm$-equivariant quasicoherent sheaves.

Recall that for $i\geq 0$, we write $\cD^{\leq i}_X\subset \cD_X$ for
the differential operators of order at most~$i$. In particular,
$\cD^{\leq 0}_X$ is the commutative ring of functions $\cO_X$, and
$\cD_X^{\leq 1}$ is the usual Picard Lie algebra of vector fields
$\BT_X$ extended by functions $\cO_X$. By convention, for $i< 0$, we
set $\cD_X^{\leq i} = \{0\}$.

Consider the $\Gm$-action on $\aline=\Spec k[t]$ so that $t$ has
weight $1$.  Following \cite{BB}, we define the subprincipal Rees
algebra to be the $\Z$-graded algebra
$$
\cR^{sp}_{X}=
\cR_X \ot_{k[t]} k[t]/(t^2) \simeq\bigoplus_{i\geq 0} t^i \left(\D_{X}^{\leq i}/ \D_X^{\leq i -2}\right).
$$ The shear shifted algebra $\cR^{sp}_{X, [-2]}$ admits the following
topological interpretation.  If we think of $\cR_{X, [-2]}$ as an
algebra over $\cO(BS^1) = k[u]$, with $u$ of degree $2$ and weight
$1$, then $\cR^{sp}_{X, [-2]}$ is the base change to $\cO(S^2) =
k[u]/(u^2)$.

Let $\wh \cR^{sp}_X$ be the completion
of $\cR^{sp}_X$ with respect to the positive weight direction
$$
\wh\cR^{sp}_{X}= \lim_{j\to \oo} \cR^{sp}_X/
\bigoplus_{i\geq j} t^i \left(\D_{X}^{\leq i}/ \D_X^{\leq i -2}\right) 
$$

Now Theorem~\ref{free formality} and an application of Koszul duality to the natural $\Omega^{-\bullet}_{X, S^2}$-augmentation module~$\cO_X \otimes \cO(S^2)$ provides the following.

\begin{thm} 
There is a canonical equivalence of stable $\oo$-categories
$$ \QCoh(\cL X)^{\Aff(\Omega S^2)\rtimes\Gm} \simeq
 \wh\cR^{sp}_{X}\topmodule_\Z
$$
\end{thm}

Alternatively, we can consider the natural $\Omega^{-\bullet}_{X, S^2}$-augmentation module~$\Omega^{-\bullet}_{X}$.
Let $\Omega_X^{-\bullet}\langle d\rangle$ be the $\Z$-graded sheaf of 
differential graded $k$-algebras on $X$ 
obtained from the de Rham algebra by adjoining
  an element $d$ of degree $-1$ with $[d, \omega] =
  d\omega$, for $\omega \in \Omega^{-\bullet}_X$, but with no equation
  on powers of $d$.
  We write $\Omega_X^{-\bullet}\langle d\rangle\topmodule$ for the stable $\oo$-category of quasicoherent
sheaves on $X$ equipped with a compatible structure of 
complete differential graded
$\Omega_X^{-\bullet}\langle d\rangle$-module.

\begin{corollary}
There is a canonical equivalence of stable $\oo$-categories
$$ \QCoh(\cL X)^{\Aff(\Omega S^2)\rtimes\Gm} \simeq
\Omega_X^{-\bullet}\langle d\rangle\topmodule_\Z
$$
\end{corollary}

\def\twD{\mathbf D}
\def\twR{\mathbf R}

Recall the discussed relationship between 
quasicoherent sheaves $\cE$ with flat connection $\nabla$ on $X$ and modules for the 
algebra $ \Omega_X^{-\bullet}[d]$. It generalizes to a similar 
relationship between 
quasicoherent sheaves $\cE$ with not necessarily flat connection $\nabla$ and modules for the 
algebra $\Omega_X^{-\bullet}\langle d\rangle$.
Namely, 
the connection $\nabla:\cE\to \cE\ot\Omega_X$ naturally extends to an $ \Omega_X^{-\bullet}\langle d\rangle$-module structure on the pullback
$$
\pi^*\cE = \cE\ot_{\cO_X}\Omega_X^{-\bullet},
\qquad\mbox{ where } \pi:\BT_X[-1]\to X.
$$
The action of the 
distinguished element $d\in \Omega_X^{-\bullet}\langle d\rangle$
is given by the formula
$$
d(\sigma\ot \omega) = \nabla(\sigma)\wedge \omega + \sigma\ot d\omega,
\qquad
\sigma\ot \omega \in\cE\ot_{\cO_X} \Omega_X^{-\bullet}.
$$
Here we do not require that $\nabla$ is flat nor correspondingly that the square of  $d $ vanishes. The square of $d $ is given by the curvature 
$R = \nabla^2:\cE \to \cE \ot \Omega_X^{\wedge 2}$.


\subsubsection{Central curvature}
We specialize the above discussion to fixed central curvature.

We will write $\cK_X$ for the tensor product of differential graded $k$-algebras
$$
\cK_X = \Omega_X^{-\bullet} \ot H_{-*}(\Omega S^3, k) \simeq \Omega_X^{-\bullet}[\kappa]
$$
where $\kappa$ has cohomological degree $-2$ and weight grading $-2$.
Since $\Aff(\Omega S^3) $ lies in the kernel of the morphism $ \Aff(\Omega S^2) \to \Aff(S^1) = B\Ga$,
following the same arguments as above, we can obtain
a description of $\Aff(\Omega S^3) \rtimes \G_m$-equivariant sheaves on $\BT_X[-1]$ in terms of complete
differential graded modules
$$
\qc(\BT_X[-1]/(\Aff(\Omega S^3) \rtimes \G_m)
\simeq
\cK_X\topmodule_\Z.
$$

Observe that the map $\Omega S^3 \to \Omega S^2$ induces a map $\cK_X \to \Omega_X^{-\bullet}\langle d\rangle$
such that $\kappa\mapsto d^2$.
Thanks to Theorems~\ref{formality} and \ref{free formality} and the previous Koszul duality reformulations, we see that a 
$B\Ga\rtimes \Gm$-equivariant sheaf on $\BT_X[-1]$
is nothing more than an $\Aff(\Omega S^2) \rtimes \Gm$-equivariant sheaf on which $\Aff(\Omega S^3)$ acts trivially.
More precisely, if we write $k_0$ for the natural augmentation $k[\kappa]$-module
on which $\kappa$ acts trivially, 
then 
we can understand $\Omega_X^{-\bullet}[d]$-modules via the naive 
identification of differential graded $k$-algebras
$$
\Omega_X^{-\bullet}[d]
 \simeq 
\Omega_X^{-\bullet}\langle d\rangle
 \ot_{k[\kappa]} k_0
$$
In other words,  a differential graded $\Omega_X^{-\bullet}[d]$-module is a
differential graded  $\Omega_X^{-\bullet}\langle d\rangle$-module
whose restriction to a differential graded $\cK_X$-module
factors through the natural projection 
$$
\xymatrix{
\cK_X = \Omega_X^{-\bullet} \ot H_{-*}(\Omega S^3, k) \ar[r] & \Omega_X^{-\bullet}.
}$$

In general, given an $\Omega_X^{-\bullet}\langle d\rangle$-module, we could ask that its restriction to a $\cK_X$-module
factors  through an alternative algebra projection $\cK_X\to  \Omega_X^{-\bullet}$. 
Given a morphism 
$$
\xymatrix{
\chi: H_{-*}(\Omega S^3, k) \ar[r] & \Gamma(X,  \Omega_X^{-\bullet}),
}$$
we will
say that a $\cK_X$-module is induced from $\chi$ if it factors through the natural
map
$$
\xymatrix{
\id\ot\chi:\cK_X =  \Omega_X^{-\bullet} \ot H_{-*}(\Omega S^3, k) \ar[r] &   \Omega_X^{-\bullet}
}$$
induced by restriction and multiplication.
Observe that $\chi$ is completely determined by the function
$
\chi(\kappa)\in\Omega_X^{-2}.
$

Given an $\Omega_X^{-\bullet}\langle d\rangle$-module, we will say that it has curvature $\chi$ if its restriction
to a $\cK_X$-module is induced from $\chi$.
Such modules are nothing more than differential graded modules for the sheaf of 
differential graded $k$-algebras
$$
\cA^\chi_X \simeq 
\Omega_X^{-\bullet}\langle d\rangle \ot_{k[\kappa]} k_\chi
$$
where $k_\chi$ is the $H_{-*}(\Omega S^3, k)$-module given by $\chi$.
Of course, when $\chi$ is trivial, we recover the original differential graded algebra $\cA^0_X = \cA_X \simeq \Omega_X^{-\bullet}[d]$.

Observe that for an arbitrary $\chi$ to arise as such a restriction, its value $\chi(\kappa)$ must be a de Rham closed
element of $\Omega_X^{-2}$ since $\chi(\kappa) = \tilde\delta^2$ commutes with $\tilde\delta$.
Furthermore, two such morphisms $\chi_1$ and $\chi_2$ lead to equivalent algebras 
$\cA^{\chi_1}_X \simeq \cA^{\chi_2}_X$ whenever the difference
$\chi_2(\kappa) - \chi_1(\kappa)$ is the de Rham boundary of some $\alpha$.
In this case, we can define an equivalence
$
\varphi_\alpha:\cA^{\chi_1}_X \to \cA^{\chi_2}_X
$
to be the identity on $  \Omega_X^{-\bullet}$ extended by $\varphi_\alpha(\tilde\delta) =\tilde\delta +\alpha$.
Thus possible central characters are classified by the holomorphic part of the second 
de Rham cohomology of $X$.

Now consider a
quasicoherent sheaf $\cE$ with connection $\nabla$ on $X$. Assume that the curvature
$R = \nabla^2:\cE \to \cE \ot \Omega_X^{\wedge 2}$ is central in the sense that
it factors as a tensor product $R = \id_\cE\ot \omega$, for some $\omega \in  \Omega_X^{\wedge 2}$.
Then 
the $\Omega_X^{-\bullet}\langle d\rangle$-module
structure on the pullback
$$
\pi^*\cE = \cE\ot_{\cO_X}  \Omega_X^{-\bullet},
\qquad\mbox{ where } \pi:\BT_X[-1]\to X,
$$
descends to an $\cA^\chi_{X}$-module structure
with $\chi$ defined by $\chi(\kappa) = \omega$.

To summarize, sheaves with connection on $X$ naturally give $\Aff(\Omega S^2)$-equivariant sheaves
on $\BT_X[-1]$. When the connection is flat, they descend to $B\Ga$-equivariant sheaves.
When the connection is central, they can be understood as $\Aff(\Omega S^2)$-equivariant
sheaves on which $\Aff(\Omega S^3)$ acts by a character.

\section{Loops in geometric stacks}\label{general loops section}

In this section, we turn to loops in derived stacks and generalize the
results of previous sections. We continue to work over a fixed
$\Q$-algebra $k$.  We will only consider
geometric derived stacks, by which we connote Artin derived stacks
with affine diagonal \cite{HAG2}. Such stacks have two related
advantages: first, they admit smooth affine covers by affine derived schemes
(in other words, they can be realized as the colimits of simplicial affine derived schemes with
smooth face maps), enabling us to reduce questions to the affine
case; and second, they have well behaved infinitesimal theory as captured by
their cotangent complexes. We will show that all of our results on
loop spaces of derived schemes extend to {\em formal} loop spaces of
arbitrary geometric stacks.


\subsection{Loops, unipotent loops, and formal loops}
For  $X$ a derived scheme, we found equivalences
$$
\xymatrix{
\BT_X[-1]\ar[r]^-\sim &\Map(B\Ga, X) \ar[r]^-\sim & \Map(S^1, X) = \cL X
}
$$ identifying the zero section $X\to \BT_X[-1]$ with the constant loops
$X\to\cL X$. Furthermore, $\BT_X[-1]$ is complete along its zero
section, and equivalently $\cL X$ is complete along constant loops.

For a general derived stack, these identifications break down. We
therefore make the following definition to distinguish between the
possible constructions.

\begin{defn} 
The {\em unipotent loop space} of a derived stack $X$ is the
  derived mapping stack 
  $$\cL^u X=\Map_{\DSt}(B\Ga,X).
  $$ 
  
  The {\em formal loop space} of $X$ is the formal completion of $\cL
  X$ along the constant loops
 $$\Lhat X=\widehat{\cL X}_X.
 $$

 The {\em formal odd tangent bundle} of $X$ is the formal completion
 of $\BT_X[-1]$ along its zero section

 $$
\wh\BT_X[-1]=\widehat{\BT_X[-1]}_X =\on{Spf}_{\cO_X} \lim_{n\to \oo} \left(\Sym(\Omega_X[1])/\on{Sym}^{>n}(\Omega_X[1]) \right). $$
 \end{defn}

\begin{remark} 
By the formal spectrum $Spf$ above we connote the corresponding
ind-derived scheme over $X$, i.e., the direct limit
functor.
\end{remark}

Recall that if we realize $S^1$ by two $0$-simplices connected by two $1$-simplices,
we obtain the identification $$ \cL X \simeq X \times_{X\times X} X.
$$ 
This guarantees that for $X$ an Artin derived stack, the loop space $\cL
X$ is again Artin. Moreover, when $X$ has affine diagonal, the natural projection $\cL
X\to X$ is affine since it is a basechange of the diagonal of $X$.

\subsubsection{Relation to inertia stacks}
It is illuminating to note that the derived loop space of a stack is
an infinitesimal thickening of the classical {\em inertia stack}.
First recall that to any derived stack $Z$, there is an underived
stack $t(Z)$ obtained by restricting the functor of points of $Z$ to
discrete rings. One refers to $t(Z)$ as the truncation of $Z$.

\begin{defn} For an underived stack $\ul{X}$, the {\em inertia stack}
$I\ul{X}$ is the underived loop space
$$
I\ul{X}={\on{Map}'}(S^1,\ul{X})\simeq \ul{X}{\times'}_{\ul{X}{\times} \ul{X}}\ul{X}.
$$ 
Here both the mapping functor $\on{Map}'$ and fiber product $\times'$
are calculated in the $\oo$-category
of stacks over discrete $k$-algebras. 
\end{defn}

Thus for $S$ a discrete $k$-algebra, the $S$-points $I \ul X(S)$ of
the inertia stack are pairs consisting of an $S$-point $x\in
\ul{X}(S)$, and an automorphism $\gamma\in \Aut_{\ul X(S)}(x)$ of the
point (or in other words, a loop based at $x$). Composition of loops
makes $\ul I X$ a group stack over $\ul X$, and when $\ul{X}$ has
affine diagonal, $I\ul{X}$ is an affine group scheme over
$\ul{X}$. Likewise one can consider the unipotent inertia stack $I^u
\ul X$. which classifies algebraic one-parameter subgroups of the
inertia stack.

We observe that the truncation $t(\cL X)$ of the loop space $\cL X$
can be described concretely as the inertia stack $I\ul{X}$ of the
truncation $\ul{X}$ of $X$. This is an immediate consequence of the
fact that truncation of derived stacks preserves limits (it is right
adjoint to inclusion of stacks).

above




\subsection{Descent for differential forms}\label{descent for forms}
In this section,
we explain a descent pattern for differential forms,
or in other words, functions on odd tangent bundles.

We begin with some notation and standard constructions.
First, for a geometric stack $Z$, the quasicoherent sheaf $\Omega^k_Z\in \qc(Z)$ 
will denote the $k$th symmetric power
of the shifted cotangent complex $ \Omega_Z^1= \Omega_Z[1]$.

For $f:Z\to W$ a representable map of derived stacks, and any integer $k\geq 0$, we define the quasi-coherent
sheaf $\Omega^k_{Z/W}\in \qc(Z)$ of {\em relative $k$-forms} to be the cone
$$
\xymatrix{
f^*\Omega^k_W \ar[r]^-{f^*} & \Omega^k_Z \ar[r] & \on{Cone}(f^*) =\Omega^k_{Z/W}.
}
$$ Of course, when $k=0$, we have $\Omega^0_{Z/W} \simeq 0$, and when
$k=1$, we recover the usual sheaf of relative $1$-forms. Informally
speaking, we can think of $\Omega^k_{Z/W}$ as the $k$-forms on $Z$
which vanish when contracted with vector fields along the fibers of
$f$.

For a diagram of representable maps 
$$
\xymatrix{
 Z\ar[r]^-f & W \ar[r] & V,
 }$$
 and for any integer $k\geq 0$,
 we have a natural triangle
 $$
\xymatrix{
f^*\Omega^k_{W/V} \ar[r] & \Omega^k_{Z/V} \ar[r] & \Omega^k_{Z/W}.
}
$$
Moreover, when $f$ is smooth, then locally the sequence admits a splitting
$$ \xymatrix{ f^*\Omega^k_{W/V} \ar[r]_-{f^*} & \ar@/_0.5pc/[l]_-{s}
  \Omega^k_{Z/V} \ar[r] & \Omega^k_{Z/W}, & s\circ f^* \simeq \id.  }
$$ (This is a consequence of the definition of smoothness as local
projectivity of the relative tangent complex, see \cite{HAG2}.)

Now we will address the descent pattern for differential forms.  Fix a
geometric stack $X$.  By a {cover} $u:U\to X$, we will always mean a
faithfully flat map, i.e., a cover in the flat topology on derived
rings as defined in \cite[5]{dag7}.  By \cite[Theorem 6.1]{dag7},
quasicoherent sheaves satisfy descent for flat covers, so that if
$u:U\to X$ is a cover, then pullback provides an equivalence
$$ \xymatrix{ \qc(X) \ar[r]^-\sim & \lim \qc(U_\bullet) }$$ where
$U_\bullet \to X$ is the corresponding augmented Cech simplicial
derived stack. We will also need the smooth topology on derived rings,
which is the refinement of the flat topology in which covering maps
are required to be smooth (this forms a Grothendieck topology thanks
to the criterion of \cite[Proposition 5.1]{dag7}). By a smooth affine
cover, we mean a cover $u:U\to X$ for which $U$ is affine and $u$ is
smooth.  Note that in this case, $U_\bullet$ will be a simplicial
affine derived scheme with smooth face maps.

Some notation:  let $u_\bullet:U_\bullet\to X$ be the unique augmentation map,
so that in particular $u_0:U_0= U\to X$ is the original cover $u$.
 
\begin{prop}\label{descent for functions}
For a geometric stack $X$ 
and smooth maps $U \to X$,
the assignment
of the complete graded algebra of
  differential forms
$$
\xymatrix{
U \ar@{|->}[r] & \wh{\on{Sym}}^\bullet\Omega_X[1]
=
 \lim_{n\to \oo} \left(\Sym(\Omega_X[1])/\on{Sym}^{>n}(\Omega_X[1]) \right)
}$$
forms a graded sheaf on the smooth site of $X$. In particular, for all $k\geq 0$,
the assignment
of shifted $k$-forms 
$$
\xymatrix{
U \ar@{|->}[r] & \Omega^k_U = \on{Sym}^k(\Omega_U[1])
}$$
forms a sheaf on the smooth site of $X$.
\end{prop}

\begin{proof} 
Fix a smooth affine cover $u:U\to X$
with corresponding augmented Cech simplicial affine derived scheme
$U_\bullet \to X$. It suffices to show that pullback induces an equivalence
$$
\xymatrix{
\wh{\on{Sym}}^\bullet\Omega_X[1]
\ar[r]^-\sim &
 \on{Tot}[u_{\bullet *}( \wh{\on{Sym}}^\bullet\Omega_{U_\bullet}[1])]
}$$
where we write $\on{Tot}$ for the limit of a cosimplicial object.

Let us first fix $k\geq 0$ and consider the assertion for $k$-forms alone.
In other words, we seek to confirm that pullback induces an equivalence
$$
\xymatrix{
\Omega^k_X
\ar[r]^-\sim &
 \on{Tot}[u_{\bullet *}(\Omega^k_{U_\bullet})]
}$$
Note that for $k=0$, this is simply descent for functions.

Since cones can be interpreted as limits, and limits always commute with limits, it suffices to establish the equivalence
$$
\xymatrix{
0\simeq \Omega^k_{X/X}[1]
\ar[r]^-\sim &
 \on{Tot}[u_{\bullet *}(\Omega^k_{U_\bullet/X}[1])],
 & \mbox{for any $k>0$}.
}$$

Let us pull back the augmented cosimplicial object $0 \to u_{\bullet
  *}(\Omega^k_{U_\bullet/X}[1])$ along the initial augmentation map
$u_0:U_0\to X$. Using base change and the splitting of cotangent
sequences for smooth maps, it is a diagram chase to confirm that the
pullback $0 \to u_0^*u_{\bullet *}(\Omega^k_{U_\bullet/X}[1])$ is a
split cosimplicial object, and hence a limit diagram.  Since $u_0$ is
a smooth, in particular flat, cover (and hence $u_0^*$ is conservative
and preserves $\on{Tot}$), we conclude that the
above totalization is indeed trivial, and thus $k$-forms satisfy the
asserted descent.

Now it remains to consider the full completed algebra of forms. We have seen that for each $k\geq 0$,
we have the asserted descent for $k$-forms. Thus it suffices to check that the totalization is complete
with respect to $k$. But by construction, each $U_n$ is an affine derived scheme, and hence
 the algebra of differential forms on $U_n$ is complete with respect to $k$.
 This immediately establishes the assertion and concludes the proof.
\end{proof}

Before continuing to a similar statement for sheaves, we record an interesting independent consequence
of the above proposition. 
Recall that for a quasi-compact derived
scheme $X$ with affine diagonal, we have established a functorial equivalence
between Hochschild chains and differential forms
$$
 HC(\cO_X) \simeq \Sym(\Omega_X[1]).
$$ Recall as well that the
algebra of differential forms on $X$ is connective, in particular complete:
the natural map  between differential forms and
their completion is an equivalence
$$
\xymatrix{
 \Sym(\Omega_X[-1])\ar[r]^-\sim & \wh{\on{Sym}}^\bullet (\Omega_X[-1]).
}$$

\begin{corollary}
For a quasi-compact derived scheme $X$ with affine diagonal,
the assignment
of Hochschild chains
$$
\xymatrix{
U  \ar@{|->}[r] & HC(\cO_U)=\cO_U\ot_{\cO_U\ot \cO_U}\cO_U
}$$
forms a sheaf on the smooth site of $X$.
\end{corollary}

In the next section, we will explain the relationship between the formal loop space and
formal odd tangent bundle in a global setting.

\medskip

The main aim of this section is the following descent for quasicoherent sheaves.

\begin{thm}\label{descent for sheaves} 
For a geometric stack $X$ 
and smooth maps $U \to X$,
the assignment of $\Gm$-equivariant
quasicoherent sheaves
$$ \xymatrix{ U \ar@{|->}[r] & \qc( \wh\BT_U[-1])^{\G_m} }
$$ forms a
sheaf of stable $\oo$-categories with $B\Ga$-action on the smooth site of $X$. 
In particular, the assignment of $B\Ga\rtimes \Gm$-equivariant sheaves also forms a sheaf.
\end{thm}

\begin{proof}
All of our arguments will be manifestly $B\Ga$-equivariant.
Fix a smooth affine
cover $u:U\to X$ with corresponding augmented Cech simplicial affine
derived scheme $U_\bullet \to X$. It suffices to show that pullback
induces an equivalence
$$
\xymatrix{
\qc( \wh\BT_X[-1])^{\Gm}
\ar[r]^-\sim &
 \on{Tot}[\qc( \wh\BT_{U_\bullet}[-1])^{\Gm}].
}$$

For $k \geq 0$, let $\bS_X^{\leq k} = \Sym \Omega_X[1] / \on{Sym}^{>k} \Omega_X[1]$
be the natural graded quotient algebra.
Observe that the left hand side of the above asserted equivalence can be calculated by the limit of modules
$$
\xymatrix{
\qc( \wh\BT_X[-1])^{\Gm}
\ar[r]^-\sim &
 \lim_k  (\bS_X^{\leq k}\module_\Z).
 }$$
 Similarly,
 the right hand side  is the totalization of the limits of graded modules
 $$
\xymatrix{
\qc( \wh\BT_{U_\bullet}[-1])^{\Gm}
\ar[r]^-\sim &
 \lim_k  (\bS_{U_\bullet}^{\leq k}\module_\Z).
 }$$

Now by the previous proposition, we  have equivalences
$$
\xymatrix{
\bS_X^{\leq k}\module_\Z
\ar[r]^-\sim &
 \on{Tot}[\bS_{U_\bullet}^{\leq k}\module_\Z].
 }$$
 Since limits commute with limits, we obtain the expression
 $$
\xymatrix{
\qc( \wh\BT_X[-1])^{\Gm}
\ar[r]^-\sim &
 \on{Tot}[ \lim_k (\bS_{U_\bullet}^{\leq k}\module_\Z)]
 }$$

But we also have the standard identity
 $$
  \xymatrix{
  \qc( \wh\BT_{U_\bullet}[-1])^{\Gm}
  \simeq
\wh{\on{Sym}}^\bullet_{U_\bullet}\topmodule_\Z
\simeq  \lim_k (\bS_{U_\bullet}^{\leq k}\module_\Z).
  }$$
Assembling the above equivalences completes the proof of the theorem.
\end{proof}



\subsection{Formal loops and odd tangents}
In this section, we will compare the formal loop space
and formal odd tangent bundle  of a geometric stack.
To begin, we will compare the loop space and
odd tangent bundle infinitesimally.

\subsubsection{Infinitesimal identification}
Recall that we have a canonical projection and section
$$
\xymatrix{
\pi:\cL X \ar[r] & X:u \ar@/_0.5pc/[l]\\
}
$$
where $\pi$ is evaluation at 
the unit $e\in S^1$, and $u$ is the inclusion of constant loops.


By the cotangent complex to $\pi:\cL X\to X$ 
along the section $u: X\to \cL X$, we mean the pullback $u^*\Omega_{\cL X}\in \qc(X)$.

\begin{lemma}\label{linear ident} Let $X$ be a geometric stack.
The cotangent complex to $\pi:\cL X\to X$ 
along the section $u: X\to \cL X$ of constant loops is functorially equivalent
to the direct sum $\Omega_X\oplus\Omega_X[1]$, and the relative cotangent complex of $\pi$
along $u$
is functorially equivalent to the shifted cotangent complex $\Omega_X[1]$.
\end{lemma}

\begin{proof} By definition, the loop space $\cL X$ is the intersection
$$
\cL X = X \times_{X\times X} X
$$
of the diagonal $\Delta: X\to X\times X$ with itself.
Therefore there is a natural Mayer-Vietoris distinguished triangle
$$
\xymatrix{
\pi^*\Delta^*\Omega_{X\times X}\ar[r] & \pi^*\Omega_X \oplus \pi^*\Omega_X 
\ar[r] & 
\Omega_{\cL X}
}
$$
of sheaves on $\cL X$. 
Our convention is that the first map is induced by restricting a one-form on $X\times X$
to the diagonal $X$, then mapping it to the corresponding anti-diagonal one-form in the direct sum. 

Pulling back along the unit section $u$, and applying the identity $\pi\circ u = \id_X$, gives a distinguished triangle
$$
\xymatrix{
\Omega_X\oplus \Omega_X\simeq \Delta^*\Omega_{X\times X}\ar[r] & \Omega_X \oplus \Omega_X 
\ar[r] & 
u^*\Omega_{\cL X}
}
$$
of sheaves on $X$. 
Under our conventions, the first map is nothing more than the map of one-forms
$$
(\omega_1,\omega_2) \mapsto (\omega_1+\omega_2, -(\omega_1+\omega_2)).
$$

Thus we conclude that $u^*\Omega_{\cL X} \simeq \Omega_X \oplus \Omega_X[1]$,
and the relative part is the summand $\Omega_X[1]$.
\end{proof}

\begin{remark} For a perhaps more conceptual alternative proof, 
one can identify the relative cotangent complex as the
stabilization functor from affine schemes over $X$ to quasicoherent
sheaves on $X$. This implies that it intertwines the based loops functor $\Omega$ with
the shift functor~$[1]$. Hence one can identify the relative cotangent complex to 
the free loop space of $X$
with the shifted cotangent complex $\Omega_X[1]$.
\end{remark}

\subsubsection{The derived exponential map}
Now we arrive at the main construction of this section.  

Let $X$ be a
geometric stack.  Fix a smooth affine cover $U\to X$ with
corresponding augmented Cech simplicial affine derived scheme
$U_\bul\to X$. Taking formal loops, we obtain a corresponding
augmented simplicial affine derived scheme $\Lhat U_\bul\to \Lhat X$.
While it is not true that $\Lhat X$ is the geometric realization of
$\Lhat U_\bul$, we will prove that $\cO_{\Lhat X}$ is the totalization
of the corresponding cosimplicial complete algebra $\cO_{\Lhat
  U_\bullet}$.  To achieve this, we will compare with formal odd
tangent bundles where we have seen such descent holds.
  
  Working relative to $X$,
  we have a morphism of complete $\cO_X$-algebras
  $$
  \xymatrix{
  \cO_{\Lhat X} \ar[r] & \on{Tot}[ \cO_{\Lhat U_\bul}]\simeq
   \on{Tot}[ \cO_{\wh\BT_{U_\bul}}] \simeq 
  \cO_{\wh{\BT}_X[-1]}.
  }
  $$
  Here the first equivalence is the identification of formal loop spaces and formal odd tangent bundles for 
  affine schemes; the second equivalence is the assertion of Proposition~\ref{descent for functions}.
 Passing to formal spectra relative to $X$, we obtain a morphism of derived stacks
  $$
  \xymatrix{
exp:\wh{\BT}_X[-1]\ar[r] &
  \Lhat X.
  }
  $$ 
  Since all of the above constructions are equivariant,  $exp$ intertwines the translation
$B\Ga$-action on $\wh\BT_X[-1]$ with the loop rotation $S^1$-action on $\Lhat X$.

  Let us write $ \wh{\cL}^u X$ for the formal completion of unipotent loops $\cL^u X$
  along constant loops. Since all loops into affines are unipotent, we find by
  construction that $exp$ factors through the formal completion of
  unipotent loops 
  $$
  \xymatrix{
 exp: \wh{\BT}_X[-1]\ar[r] &  \wh{\cL}^u X\ar[r] & \Lhat X.
}  $$
Observe as above that as a map to the formal unipotent loop space, $exp$ is $\Gm$-equivariant.

Recall that the canonical $\Gm$-action on the affinization
$\Aff(S^1)\simeq B\Ga$ induces one on the unipotent loop space $\cL^u
X = \Map(B\Ga, X)$.  We also have a canonical $\Gm$-action on the odd
tangent bundle $\BT_X[-1]$ which descends to the formal odd tangent
bundle $\wh\BT_X[-1]$.

\begin{thm} \label{formal}
For $X$ a geometric stack, the exponential map
  $$
  \xymatrix{
exp:\wh{\BT}_X[-1]\ar[r] &
  \Lhat X.
  }
  $$  is an equivalence. 
\end{thm}

\begin{proof}
Let us first deduce the theorem from the claim that the relative
cotangent complex of $ exp:\wh{\BT}_X[-1]\to \Lhat X $ vanishes.
By filtering by powers of the augmentation ideals, we may write
the complete $\cO_X$-algebras $\cO_{\Lhat X}$ and
$\cO_{\wh{\BT}_X[-1]}$ as inverse limits of a sequence of square-zero
extensions. The morphism $exp$ automatically preserves this
filtration. At each step we apply the result of \cite[7.4.2]{HA} (see
also \cite[1.4.2]{HAG2}) asserting that the cotangent complex controls
square zero extensions. (More specifically, the obstruction to lifting a
homomorphism to a square-zero extension and the space of lifts are
both described by shifted homomorphisms out of the relative cotangent
complex, see \cite[7.4.2.3]{HA}.) Thus we deduce inductively that
$exp$ is an equivalence.

We will now establish the vanishing of the relative cotangent complex
of the exponential map.  

First, we claim that the restriction of the relative cotangent complex
of $exp$ along the zero section $z:X\to \wh\BT_X[-1]$ vanishes.  To
see this, note that Lemma~\ref{linear ident} is functorial and so by
Proposition~\ref{descent for functions} compatible with descent.  Thus
the identification of Lemma~\ref{linear ident} agrees with the
linearization of $exp$.

Next,  we claim an easy version of Nakayama's lemma holds: if $\cM\in
\qc(\wh\BT_X[1])$ satisfies $z^*\cM \simeq 0$, then in fact $\cM
\simeq 0$. We include the proof for the reader's convenience.  
For
each $k\geq 0$, consider the morphism
$$
\xymatrix{
z_k:\BT^{\leq k}_X[1] = \Spec_{\cO_X}(\Sym(\Omega_X[1])/\on{Sym}^{>k}(\Omega_X[1]))
\ar[r] & \wh\BT_X[1].
}$$
So in particular, for $k=0$, we have $\BT^{\leq 0}_X[1] = X$ and $z_0 = z$. 
To prove the claim, it suffices
to show that $z_k^*\cM\simeq 0$, for all $k\geq 0$.

We proceed by induction. For $k=0$, we have seen $z_0^*\cM\simeq 0$. Now for any $k\geq 0$,
inside of $\qc(\wh\BT_X[1])$,  we have a distinguished triangle of modules
$$
\xymatrix{
z_*\on{Sym}^{k+1}(\Omega_X[1])
\ar[r] &
\cO_{\BT^{\leq k+1}_X[1]} \ar[r] &
\cO_{\BT^{\leq k}_X[1]} 
}
$$
By the projection formula, we have the vanishing
$$
z_*\on{Sym}^{k+1}(\Omega_X[1])\ot_{ \cO_{\wh\BT_X[1]}} \cM \simeq 
z_*(\on{Sym}^{k+1}(\Omega_X[1])\ot_{ \cO_X} z^*\cM) \simeq 0,
$$
and by induction, we have the vanishing
$$
z_k^*\cM = \cO_{\BT^{\leq k}_X[1]} \ot_{ \cO_{\wh\BT_X[1]}} \cM \simeq 0.
$$
Thus we also have the vanishing
$$
z_{k+1}^*\cM = \cO_{\BT^{\leq k+1}_X[1]} \ot_{ \cO_{\wh\BT_X[1]}} \cM \simeq 0.
$$ This completes the proof of the vanishing of relative cotangents,
and hence of the theorem.
\end{proof}

\begin{corollary}
 For $X$ a geometric stack, formal loops are unipotent: the canonical map $\Lhat X\to \cL X$
  factors through the canonical map $\cL^u X \to \cL X$. Moreover, the
  composite map
  $$
  \xymatrix{
  \wh{\BT}_X[-1]\ar[r]^-{exp} & \Lhat X \ar[r] & \cL^u X
  }
  $$
  is canonically
$B\Ga\rtimes\Gm$-equivariant.
\end{corollary}

\begin{remark}[Unipotent loops embed in odd tangents]
Consider the deformation to the normal cone $\BT_X[-1]$ of the loop space
$\cL X$ along the constant loops $X$. The $\Gm$-action on the unipotent loop space $\cL^u X$ trivializes the restriction of this
deformation to unipotent loops. In this way, we obtain a canonical embedding 
$$
\xymatrix{
\cL^u X\ar[r]& \BT_X[-1].
}
$$ Taking formal completions along constant loops, and using
Theorem~\ref{formal} to identify the formal loop space $\Lhat X$ with
the formal unipotent loop space $\Lhat^u X$, we obtain from the above
morphism the inverse to $exp$.
\end{remark}


\subsection{Equivariant sheaves on formal loop spaces}

Our locality results for formal odd tangent bundles together
with their identifications with formal loop spaces now allow us to
immediately transport all of our local results to formal loop spaces of smooth geometric underived stacks.

Combining Theorems \ref{descent for sheaves} and \ref{formal}, we obtain the
following corollary. It can be viewed as a linear version of locality for
formal loops. 

\begin{corollary}
For a geometric stack $X$ 
and smooth maps $U \to X$,
the assignment of $S^1$-equivariant
quasicoherent sheaves
$$
\xymatrix{
U \ar@{|->}[r] & \qc(\Lhat U)^{S^1}
}$$
forms a sheaf on the smooth site of $X$.
\end{corollary}

Recall that for $X$ a derived scheme, we write $\Omega_X^{-\bul}[d]$
for its extended de Rham algebra, and
$\Omega^{-\bul}_{X, d}$ for its de Rham complex. Since
both form sheaves of $k$-algebras in the smooth topology, it makes
sense to consider them for a geometric
stack.  
Finally, note as well that the identification $\Lhat X\simeq \wh\BT_X[-1]$ for a geometric stack $X$ allows
us to consider $B\Ga\rtimes \Gm$-equivariant sheaves. 

\begin{thm}
For $X$ a smooth underived geometric stack,
there are canonical equivalences of stable $\oo$-categories
$$
\qc(\Lhat X)^{B\Ga\rtimes\Gm} \simeq \Omega^{\bul}_{X}[d]\module_\Z
$$
$$
\qc(\Lhat X)^{B\Ga\rtimes\Gm}_{per} \simeq \Omega^{\bul}_{X, d}\module
$$
\end{thm}

Similarly, we also have generalizations to geometric stacks
of Koszul dual descriptions and further variations
described for derived schemes.  We leave the statements to the reader.

\end{document}